\documentclass[]{article}
\usepackage[mode=buildnew]{standalone}

\usepackage{amssymb, amsmath, amsfonts, amsthm}
\usepackage{mathtools}
\usepackage{xcolor}
\usepackage{soul}
\usepackage{arxiv}
\usepackage{hyperref}
\usepackage{nameref}

\newcommand{\bb}[1]{\mathbb{#1}}

\newcommand{\R}{\bb{R}}
\newcommand{\Z}{\bb{Z}}
\newcommand{\N}{\bb{N}}
\newcommand{\V}{\bb{V}}
\newcommand{\U}{\bb{U}}
\newcommand{\W}{\bb{W}}
\newcommand{\E}{\bb{E}}
\newcommand{\X}{\bb{X}}
\newcommand{\T}{\bb{T}}
\newcommand{\rank}[3]{\mathrm{rank}({#1}_{#2}\rightarrow{#1}_{#3})}
\newcommand{\VInterval}{I}
\newcommand{\FunctionalFamily}{\mathcal{F}}
\newcommand{\KolmogP}{r_2}
\newcommand{\KolmogR}{r_1}
\newcommand{\Borel}{\mathcal{B}}

\newcommand{\vmin}{v_{\min}}
\newcommand{\vmax}{v_{\max}}
\newcommand{\bigO}{\mathcal{O}}

\newcommand{\Lebesgue}{\mu}
\newcommand{\Frac}{\mathrm{frac}}
\newcommand{\iid}{i.i.d}
\newcommand{\SampleSpace}{\Omega}
\newcommand{\gammaSpace}{\mathcal{G}} 
\newcommand{\sAlgebra}{\mathcal{A}}
\newcommand{\pattern}{\phi}
\newcommand{\functional}{\rho}
\newcommand{\normalizedfunctional}{\overbar{\functional}}
\newcommand{\kernel}{k}
\newcommand{\FunctionalSpace}{\mathcal{H}}
\newcommand{\FunctionalDomain}{\T}
\newcommand{\overbar}[1]{\mkern 2mu\overline{\mkern-2mu#1\mkern-2mu}\mkern 2mu}
\newcommand{\restr}[2]{{
		\left.\kern-\nulldelimiterspace 
		#1 
		\right\vert_{#2} 
}}
\newcommand{\pushforward}[1]{(#1)_\star}

\newcommand{\pers}{\mathrm{pers}}
\newcommand{\indicator}[1]{1_{#1}}
\newcommand{\Uniform}{\mathcal{U}}
\DeclareMathOperator{\support}{supp}
\newcommand{\Proj}{\pi}

\newcommand{\RowVec}[1]{\left(#1\right)}
\newcommand{\nOne}{n_1}
\newcommand{\nzero}{n_0}
\newcommand{\ScalarProd}[2]{{#1}^T {#2}}
\newcommand{\ScalarProdR}[2]{{#1}\cdot{#2}}
\newcommand{\density}[1]{f_{#1}}
\newcommand{\TPK}{P}
\newcommand{\BiTPK}{\tilde{\TPK}}
\newcounter{parcounter}
\newcommand{\stepparagraph}[1]{\stepcounter{parcounter} \paragraph{Step \arabic{parcounter}} \textbf{#1}\newline}

\newtheorem{theorem}{Theorem}[section]
\newtheorem{proposition}[theorem]{Proposition}
\newtheorem{lemma}[theorem]{Lemma}
\newtheorem{example}[theorem]{Example}
\newtheorem{remark}[theorem]{Remark}

\title{Topological signatures of periodic-like signals}
\author{
\href{https://wreise.github.io/}{Wojciech Reise}\\
	DataShape, Inria Saclay
\And
\href{https://bertrand.michel.perso.math.cnrs.fr/}{Bertrand Michel}\\
\'Ecole Centrale de Nantes
\And
\href{https://geometrica.saclay.inria.fr/team/Fred.Chazal/}{Fr\'ed\'eric Chazal}\\
	DataShape, Inria Saclay
}

\begin{document}

\maketitle

\begin{abstract}
	We present a method to construct signatures of periodic-like data. Based on topological considerations, our construction encodes information about the order and values of local extrema. Its main strength is robustness to reparametrisation of the observed signal, so that it depends only on the form of the periodic function. The signature converges as the observation contains increasingly many periods.  We show that it can be estimated from the observation of a single time series using bootstrap techniques.
\end{abstract}



\section{Introduction}
We consider the problem of constructing a descriptor of a periodic function $\pattern: \R\rightarrow\R$, based on an observation of a reparameterized and noisy signal. Specifically, we assume that $\pattern$ is 1-periodic and we let $\gamma:[0,T]\rightarrow [0,R]$ be an increasing bijection, $W:[0,1]\rightarrow \R$ a noise process. We consider an observation $S$ of the form
\begin{equation}
\label{eq:model_continuous_signal}
S:[0,T]\rightarrow\R,\qquad t\mapsto(\pattern\circ\gamma)(t) + W(t).
\end{equation}
Our aim is to construct a signature $F:S\mapsto F(S)$ which contains information about $\pattern$ while remaining robust to $W$ and to changes in $\gamma$.

Time series or functional observations of the form~\eqref{eq:model_continuous_signal} appear in many applications, where $\pattern$ is somehow characteristic of a population: child growth dynamics~\cite{ramsayAppliedFunctionalData2002}, physiological signals~\cite{goldbergerPhysioBankPhysioToolkitPhysioNet2000}, bird migration curves~\cite{suStatisticalAnalysisTrajectories2014}. The reparametrisation $\gamma$ is the main source of variability in the pointwise evaluations of the signals, as in the `phase variation' model in Functional data analysis (FDA), see~\cite{marronFunctionalDataAnalysis2015} for a review. The problems typically considered in FDA consist in aligning a population of curves or computing a representative curve, for which methods with guarantees have been proposed~\cite{gasserAlignmentCurvesDynamic1997, khorramTrainableTimeWarping2019, tangPairwiseCurveSynchronization2008}. Underlying most of the models is the assumption that the start and end points ($\gamma(0)$ and $\gamma(T)$ here) are common for all curves.

In applications like magnetic odometry~\cite{bonisTopologicalPhaseEstimation2022} or gait analysis~\cite{boisTopologicalDataAnalysisbased2022}, a single observation is composed of several periods of $\pattern$ and the number of periods varies across observations. For example, in the former, $S$ is the magnetic signal recorded in a moving car and the problem consists in inferring its displacement. The periodic function $\pattern$ models the magnetic signature of the angular position, $\gamma$, of that cars' wheel. The problem consists in estimating $\gamma$ from $S$. There is little reason for two observations to have the same number of periods, unless the initial angular position of the wheel and the trajectory are exactly the same across those two observations. Therefore, in contrast with FDA, the assumption of common endpoints is not satisfied and the problem changes from describing the whole signal, to that of describing its constituent parts, that is, the periods of $\pattern$.

Techniques from topological data analysis (TDA) are said to describe the `shape of data' and have been increasingly used to extract geometric or topological information from observations~\cite{chazalIntroductionTopologicalData2021a}. The arguably most popular TDA technique for analyzing a time series consists in computing the homology of the time-delay embedding (TDE) of the time series, in order to verify whether the underlying phenomenon is periodic or not~\cite{pereaTopologicalTimeSeries2019}. In applications, it has also been used to understand dynamical systems behind climate change~\cite{ghilReviewArticleDynamical2023}, to identify market crashes~\cite{gideaTopologicalDataAnalysis2018a} or to propose biomarkers to detect seizures~\cite{fernandezTopologicalBiomarkersRealtime2022}.
The TDE of a time series $(S_{n})_{n=1}^N$ is a point cloud in $\R^d$, where each point is of the form $(S_n, S_{n+\tau},\ldots, S_{n+(d-1)\tau})$ for parameters $d,\tau\in\N$. If $S$ is periodic, a simplicial complex constructed on the TDE at the right scale will have a non-trivial homology group in dimension one, as illustrated in the top row in Figure~\ref{fig:takens_ellipse_change}. In signals with phase variation however, the length of the periodic structure changes and so does the geometry of the TDE, as shown in the bottom row in Figure~\ref{fig:takens_ellipse_change}. This is corroborated by the fact that the geometry of the delay embedding contains information about the frequencies supporting the signal~\cite[section 5]{pereaTopologicalTimeSeries2019}.

Techniques other than the TDE have been proposed to extract topological information from time series. In~\cite{corcoranModellingTopologicalFeatures2017}, the swarm behavior over time has been described with the zig-zag persistent homology of sublevel sets of a density estimator. In~\cite{khasawnehChatterDetectionTurning2016}, the authors count revolutions of a machine in an industrial process by counting the number of `significant' changes in a binary signal, where the significance of a change is defined in terms of persistence of homology generators.

Building on the invariance of homology to reparametrisation of the domain, we propose to use the persistent homology of sublevel sets of the signal to describe this last. This descriptor summarizes the height, order and number of local extrema. The idea of quantifying the shape of the curve is not new: for example, the landmark method extracts visual features like local extrema or inflection points~\cite{perngLandmarksNewModel2000}.

In many statistical applications, it is convenient to map a persistence diagram to a vector or a function, via a functional representation~\cite{chazalIntroductionTopologicalData2021a}. Numerous functionals~\cite{carrierePersLayNeuralNetwork2020a, adams_persistence_2017} are `linear in the diagram' and their properties have been well-studied~\cite{divolChoiceWeightFunctions2019}. In our case, it seems natural to renormalize the functionals by the total persistence of the diagram, a proxy for the number of periods. Building on~\cite{divolChoiceWeightFunctions2019} and a recent characterization of the stability of total persistence for H\"older regular processes~\cite{perez_c0-persistent_2022}, we study the robustness of the signatures we propose.

Guarantees on the estimation of functionals of persistence diagrams, in both asymptotic and non-asymptotic cases, have been provided in~\cite{chazal_stochastic_2014, berry_functional_2018}, under the assumption that the persistence diagrams (or functionals thereof) in the collection are all independent.
In a setting motivated by magnetic odometry problem~\cite{bonisTopologicalPhaseEstimation2022}, we have a single time series of which we would like to estimate the signature. The natural procedure is to construct a sample by taking contiguous vectors from that observation, what leads to a collection of shorter and dependent observations. We study two reparametrization models inspired by~\cite{marronFunctionalDataAnalysis2015} and, building on the theory of strong mixing~\cite{doukhanMixing1995, dedecker_weak_2007}, we show that the dependence between observations decreases. When the $\beta$-mixing coefficients decrease sufficiently fast, the estimators of the functionals also converge in the dependent setting~\cite{radulovicBootstrapEmpiricalProcesses1996, buhlmannBlockwiseBootstrapGeneral1995, kosorok_introduction_2008}, not unlike in the independent setting~\cite{chazal_stochastic_2014}. So far, estimation of topological signatures from dependent data has been less explored:~\cite{krebsLimitTheoremsPersistent2021} gives a concentration inequality for persistent Betti numbers from dependent data.
\begin{figure}
	\centering
	\includegraphics[width=0.98\textwidth]{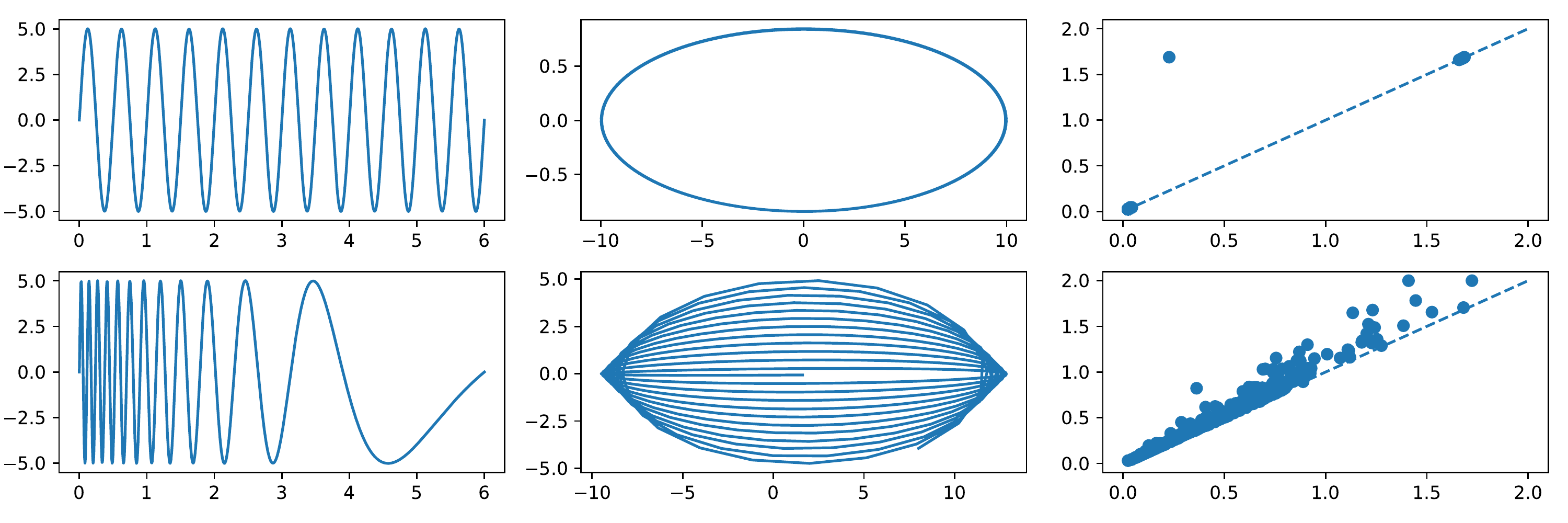}
	\caption{On the top row, from left to right, a periodic function, a projection of the sliding window embedding on the first two principal components, and the persistence diagram in dimension 1 of the Vietoris-Rips complex of the embedding. We can see that the embedding looks circular, what is reflected in the persistence diagram: there is one, persistent $H_1$ generator. The bottom row shows the analogue, for a reparametrisation of the signal in the top row. Due to the homeomorphism, the (projection of the) embedding looks different and this is also reflected in the diagram, which is not that of a closed curve.}
	\label{fig:takens_ellipse_change}
\end{figure}

\subsection*{Contributions and outline}
In this article, we present a topology-inspired signature of reparametrized periodic functions of the form~\eqref{eq:model_continuous_signal}, which is robust to reparametrisation and noise, and which can be estimated from data. In Section~\ref{sec:continuous_signatures}, we concisely introduce the signature and show its key invariance properties. In Section~\ref{sec:signatures_discrete}, we introduce models for reparametrisations and we discuss the guarantees of estimation of the signatures defined for time-series. Section~\ref{sec:persistence_results} contains background on the persistent homology and states stability properties essential for the previous sections. Our main contributions are the the following:

\begin{enumerate}
	\item We demonstrate that the signature converges as the number of observed periods grows, in case there is no additive noise (Theorem~\ref{thm:convergence_to_limit}). In the process, we provide a characterization of the persistence diagram of sublevel sets of several periods of a function.
	\item We show that the signature is invariant under changes of the distribution of $\gamma$ for fixed endpoints $\gamma(0)$ and $\gamma(T)$ (Theorem~\ref{thm:signature_stability_to_reparametrisation}). Recent results on regularity of total persistence allow us to obtain quantitative stability bounds.
	\item We provide a technique to estimate the signature from a single time--series observation (Theorem~\ref{thm:main_clt}). We show that two reparametrisation models that we consider exhibit exponentially-decaying mixing properties.
\end{enumerate}
Finally, Section~\ref{sec:numerical_experiments} provides a simple numerical illustration of the signature and its invariance properties.

\section{Signatures of reparametrized periodic functions and their properties}
\label{sec:continuous_signatures}
The signatures we propose are functions constructed using the local minima and maxima of the signal. We define those signatures with persistent homology of the sublevel sets of the signal and its functional representations. In Section~\ref{sec:persistence_intro_short}, we state the properties of persistence diagrams and the corresponding functionals, with the aim of justifying the quantities of interest. For the sake of readability, we defer the details of the construction of the persistence diagrams, the functionals, their properties and proofs of some propositions to Sections~\ref{sec:persistence_results} and~\ref{sec:functionals} respectively and we include an illustration of these concepts in Figure~\ref{fig:pipeline_illustration}.

\subsection{Normalized functionals of truncated persistence}
\label{sec:persistence_intro_short}
The persistence diagram $D(S)$ of $S\in C([0,T],\R)$ is a multiset of points in $\R^2$, where the coordinates are the values of local extrema of $S$. Each point can be interpreted as a local minimum paired with a local maximum. That pairing is constructed by tracking the evolution of connected components in sublevel sets $S^{-1}(\rbrack-\infty,t\rbrack)$ as $t$ changes. It is done in a consistent manner, so that if $S$ is reparametrised, the persistence diagram remains unchanged, see Lemma~\ref{lemma:invariance_to_reparametrisation}. In addition, if $S$ is periodic, the multiplicity of any point in the diagram reflects the number of periods in $S$, except for some extra points due to incomplete periods close to the lower- and upper- endpoints of the domain $[0,T]$, see Lemma~\ref{lemma:limit_diagram}.
Therefore, for a periodic function $\pattern$, the only component of the parametrisation that the persistence diagram depends on is the starting point, $\gamma(0)$, and the number of observed periods, $\gamma(T)-\gamma(0)$.
\begin{figure}
	\centering
	\includegraphics[width=0.3\textwidth]{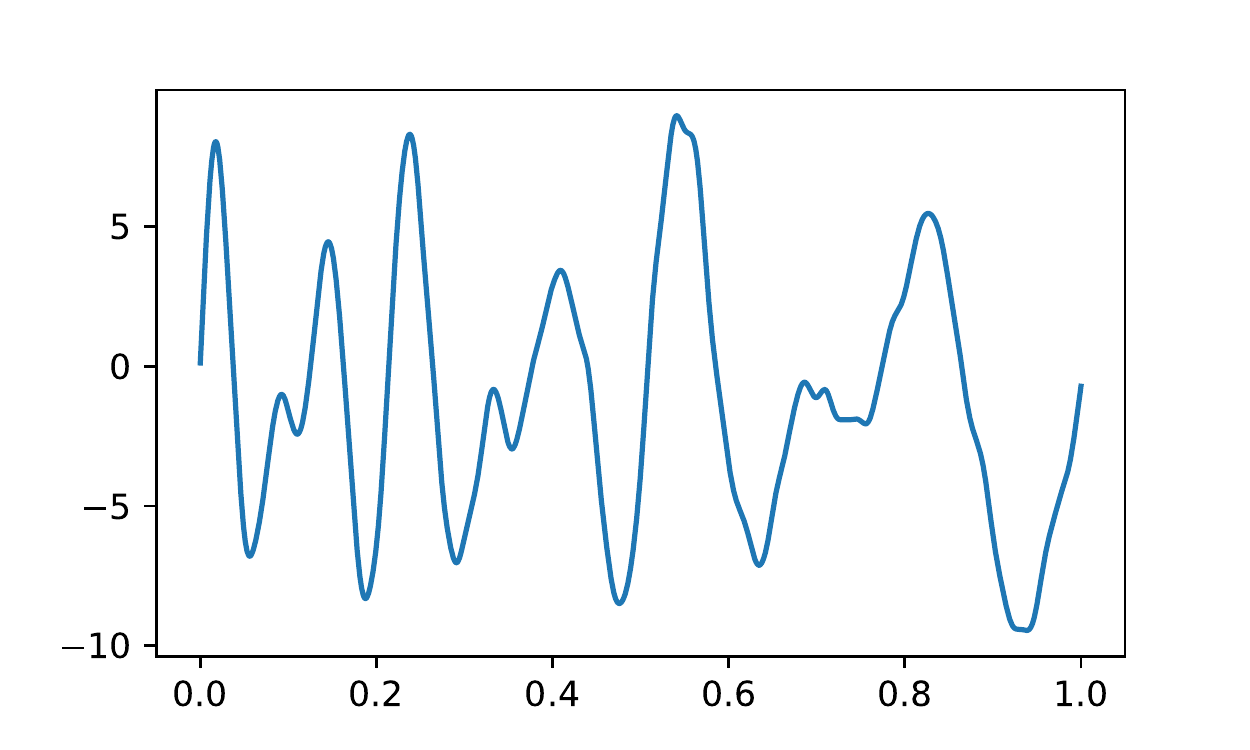}
	\includegraphics[width=0.3\textwidth]{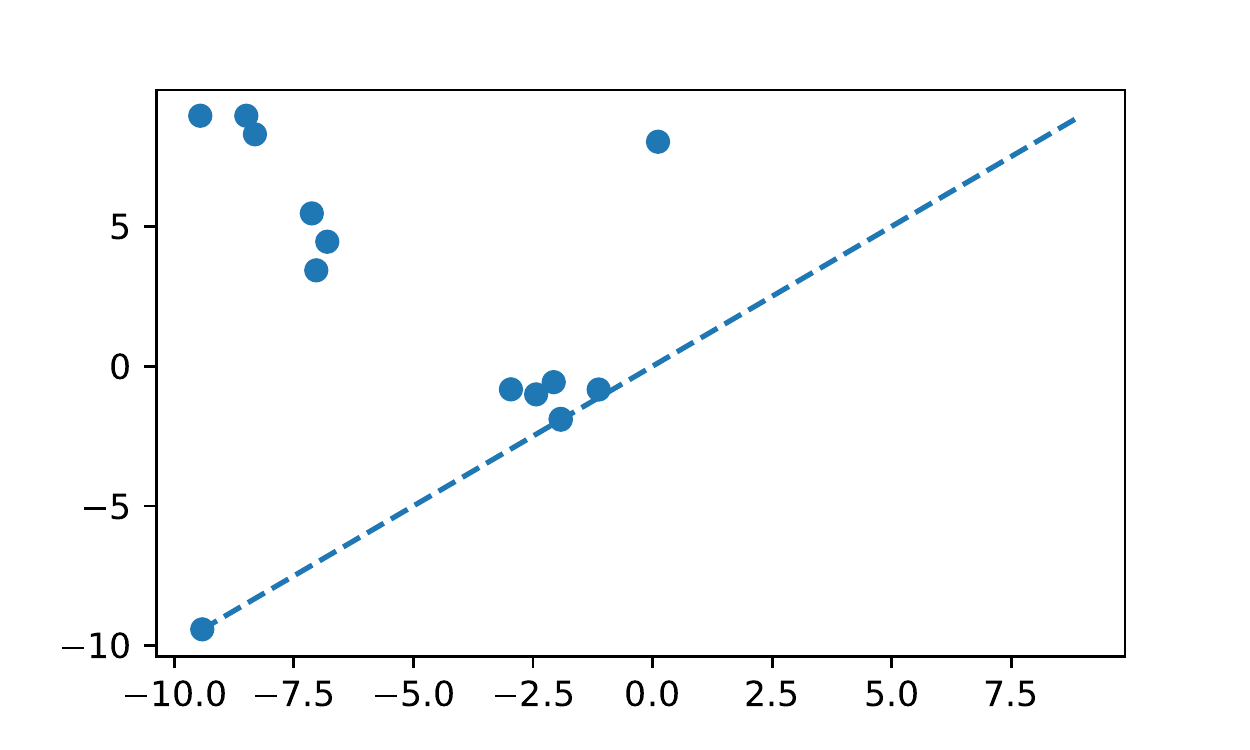}
	\includegraphics[width=0.3\textwidth]{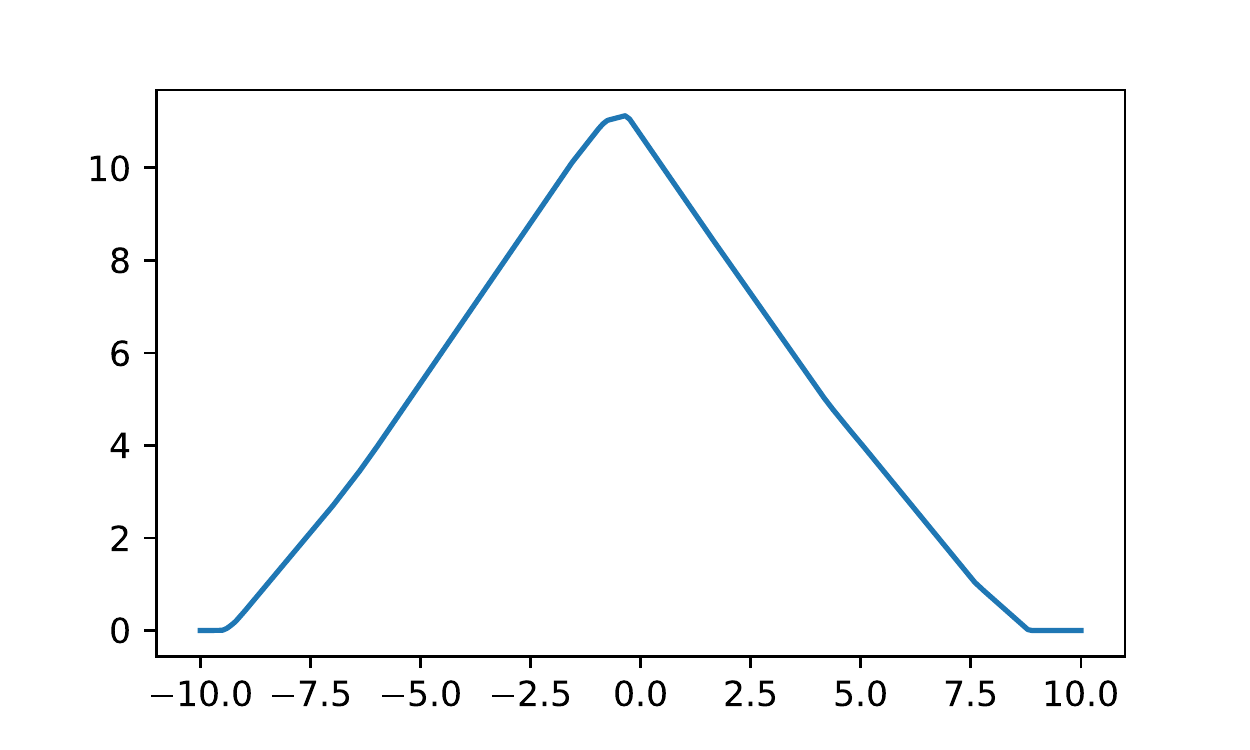}
	\caption{An example of a noisy observation of a reparametrised periodic function (left), its persistence diagram (center) and a functional summary (right): the persistence silhouette.}
	\label{fig:pipeline_illustration}
\end{figure}

\begin{lemma}[Invariance to reparametrisation]
	\label{lemma:invariance_to_reparametrisation}
	Consider a continuous function $f:\R\rightarrow\R$ (not necessarily periodic) and let $\gamma_1,\gamma_2:[0,T]\rightarrow \R$ be two increasing and continuous functions, such that $\gamma_1(0)=\gamma_2(0)$ and ${\gamma_1(T)=\gamma_2(T)}$.
	Then,
	\begin{equation*}
	D(f\circ\gamma_1) = D(f\circ\gamma_2).
	\end{equation*}
\end{lemma}
\begin{proof}
	For any $t\in\R$, the homeomorphism $g\coloneqq (\gamma_1^{-1}\circ\gamma_2): [0,T]\rightarrow [0,T]$ maps the $t$-sublevel set of $f\circ\gamma_2$ to $f\circ\gamma_1$. Indeed,
	\begin{align*}
	(f\circ\gamma_1)^{-1}(\rbrack -\infty, t\rbrack)
	&= \{y\in[0,T]\mid (f\circ\gamma_1)(y)\leq t\} \\
	&= \{y=g(x)\mid (f\circ\gamma_1)(g(x))=(f\circ\gamma_2)(y)\leq t\}\\
	&= g(\{y\in[0,T]\mid (f\circ\gamma_2)(y)\leq t\}).
	\end{align*}
	Therefore, $g$ induces an isomorphism between the two corresponding persistence modules. So the corresponding persistence diagrams are the same (as well as any invariants there--of).
\end{proof}

Consider $\pattern:\R\rightarrow\R$ a 1-periodic and continuous function. We denote by $\restr{\pattern}{A}$ the restriction of $\pattern$ to $A\subset\R$ and by $D\sqcup D'$ the union of two multisets.
\begin{lemma}[Additivity of diagrams]
	\label{lemma:limit_diagram}
	For any $R>1$, there exist persistence diagrams $D_1$ and $D'$, such that
	\begin{equation}
	\label{eq:additivity}
	D(\restr{\pattern}{[0,R]}) = \left(\bigsqcup_{k=1}^{\lfloor R-1\rfloor} D_1\right) \sqcup D',
	\end{equation}
	with $\pers_p(D')\leq 2\pers_p(D_1)$, where $\pers_p(D)=\left(\sum_{(y_1, y_2)\in D} (y_2 - y_1)^p\right)^{1/p}$.
	In addition, there exists $c\in[0,1]$ such that $D_1=D(\restr{\pattern}{[c,c+1]})$.
\end{lemma}
In the proof, we first choose $c\in[0,1]$ to be a global maximum of $\pattern$ and define ``the period" to be $\restr{\pattern}{[c,c+1]}$. This allows us to decompose the diagram as a sum of diagrams of individual periods. Thanks to the periodicity of $\pattern$, these diagrams are the same and we obtain~\eqref{eq:additivity}.
The proof requires notions introduced in~\ref{lemma:limit_diagram} and the theory of rectangular measures introduced in~\cite{chazalStructureStabilityPersistence2016}.
\begin{proof}
	Let $M\coloneqq \max\pattern$, $c\coloneqq\inf\{x\in[0,1]\mid \pattern(x)=M\}$ and $N = \max \{n\in\N\mid c + n \leq R\}$.
	Consider the persistence modules defined by~\eqref{eq:persistent_module_definition} for $\restr{\pattern}{[0,c]}$, $\restr{\pattern}{[c, c+N]}$ and $\restr{\pattern}{[c+N,R]}$. For $t<M$, $\restr{\pattern}{[0,c]}^{-1}(\rbrack-\infty, t\rbrack)\cap \restr{\pattern}{[c,c+N]}^{-1}(\rbrack-\infty, t\rbrack)\subset \{c\}$ and $\pattern(c)=M$, so that intersection is empty and the same holds for $\restr{\pattern}{[c+N,R]}$ and $\restr{\pattern}{[c,c+N]}$. Therefore,
	\begin{equation}
	\label{eq:isomorphism}
	\begin{aligned}
	H_0(\restr{\pattern}{[0,R]}^{-1}(\rbrack -\infty, t\rbrack)) \simeq&
	H_0(\restr{\pattern}{[0,c]}^{-1}(\rbrack -\infty, t\rbrack)) \oplus H_0(\restr{\pattern}{[c,c+N]}^{-1}(\rbrack -\infty, t\rbrack))\\
	&\oplus H_0(\restr{\pattern}{[c+N,R]}^{-1}(\rbrack -\infty, t\rbrack)).
	\end{aligned}\end{equation}
	Since the isomorphism is induced by inclusions, it is an isomorphism between the persistence modules restricted to $t\in \rbrack-\infty,M\lbrack$.
	By definition~\eqref{eq:persistent_module_definition}, the persistence modules are all 0 for $t\geq M$, so both sides of~\eqref{eq:isomorphism} are trivially isomorphic for $t\geq M$. Therefore, the persistence modules (on $t\in\R$) are isomorphic.
	
	By repeating the same argument as above, we can show that the persistence module of $\restr{\pattern}{[c, c+N]}$ is the direct sum of the persistence modules of $(\restr{\pattern}{[c+n, c+n+1]})_{n=0}^{N-1}$. Then, for any $n=0,\ldots, N-1$, $g_n:x\mapsto x+n$ is an isomorphism between the sub level set of $\restr{\pattern}{[c, c+1]}$ and $\restr{\pattern}{[c+n, c+n+1]}$, so the persistence module of $\restr{\pattern}{[c, c+N]}$ is isomorphic to the direct sum of $N$ copies of $\restr{\pattern}{[c, c+1]}$. Thus,~\eqref{eq:isomorphism} becomes
	\begin{align*}
	H_0(\restr{\pattern}{[0,R]}^{-1}(\rbrack -\infty, t\rbrack)) \simeq
	&\left(\bigoplus_{n=0}^{N-1} H_0(\restr{\pattern}{[c,c+1]}^{-1}(\rbrack -\infty, t\rbrack))\right) \oplus H_0(\restr{\pattern}{[0,c]}^{-1}(\rbrack -\infty, t\rbrack))
	\\&\oplus H_0(\restr{\pattern}{[c+N,R]}^{-1}(\rbrack -\infty, t\rbrack)).
	\end{align*}
	
	The second crucial observation is that the diagram of a direct sum of two persistence modules is the union of diagrams. The case of interval decomposable modules is treated in~\cite[Proposition 2.16]{chazalStructureStabilityPersistence2016}. The persistence modules that we consider are $q$-tame~\cite[Theorem 3.33]{chazalStructureStabilityPersistence2016}, so they do not necessarily admit an interval decomposition. Recall that the persistence diagram is computed via rectangle measures~\cite[Section 3]{chazalStructureStabilityPersistence2016}, defined with ranks of inclusion morphisms. For two persistence modules $\V=(V_t)_{t\in\R}$, $\W=(W_t)_{t\in\R}$ and any $s,t\in\R$, we have that $\rank{(V\oplus W)}{s}{t} = \rank{V}{s}{t} + \rank{W}{s}{t}$. This shows that the two rectangle measures $(\mu_V + \mu_W)$ and $\mu_{V\oplus W}$ are equal and so are their persistence diagrams. If we denote by $D_1\coloneqq D(\restr{\pattern}{[c+n, c+n+1]})$ and by $D'$ the diagram of the sum of the rectangle measures of the $\restr{\pattern}{[0,c]}$ and $\restr{\pattern}{[c+N,R]}$,
	then~\eqref{eq:additivity} follows.
	
	We now need to bound the $p$-persistence of the remainder. Denote by $\U$ and $\V$ the persistence modules associated to $\restr{\pattern}{[0,c]}$ and $\restr{\pattern}{[0,c]}$ respectively. For any $t\in \R$, $\restr{\pattern}{[0,c]}^{-1}(\rbrack-\infty, t\rbrack)\subset \restr{\pattern}{[c-1,c]}^{-1}(\rbrack-\infty, t\rbrack)$ induces a map $U_t\rightarrow V_t$. We claim that it is injective
	and that it is in fact a morphism between persistence modules.
	Hence, $\rank{U}{s}{t}\leq\rank{V}{s}{t}$ for any $s<t\in\R$ and both are finite. Hence, to every point $(b,d)\in D(\restr{\pattern}{[0,c]})$ with $b<d$, we can assign a point $(b',d')\in D(\restr{\pattern}{[-1+c,c]})$ in such a way that this assignment is injective (considered with multiplicity) and such that $b'\leq b < d\leq d'$. So, $\pers_{p,\epsilon}^p(D(\restr{\pattern}{[0,c]}))\leq \pers_{p,\epsilon}^p(D(\restr{\pattern}{[-1+c,c]}))$. A similar argument shows that $\pers_{p,\epsilon}^p(D(\restr{\pattern}{[c+N,R]}))\leq \pers_{p,\epsilon}^p(D(\restr{\pattern}{[c+N,c+N+1]}))$.	
\end{proof}

As presented above, a persistence diagram is a multi-set of points in $\R^2$. To gain algebraic and statistical properties, it is often convenient to map the diagram to a functional representation.
In such a representation, we typically associate to each point $(y_1,y_2)$ from the persistence diagram a function, $\kernel_{y_1,y_2}:\FunctionalDomain\rightarrow\R$, for some metric space $\FunctionalDomain$. A functional representation is then a weighted sum of such functions, where the weights are commensurate with a measure of importance of each point, for example, the $\epsilon$-truncated persistence $w_\epsilon(y_1,y_2)= \max(y_2-y_1-\epsilon, 0)$ for some $\epsilon>0$. In this work, we will typically consider \emph{normalized functionals of $\epsilon$-truncated $p$-persistence}, for some $p>1$ and for any $t\in\FunctionalDomain$,
\begin{equation}
\label{eq:normalizedfunctional}
\normalizedfunctional_{\kernel, \epsilon, p}(S)(t) = \frac{\sum_{(y_1,y_2)\in D(S)} w_\epsilon(y_1,y_2)^p \kernel_{y_1,y_2}(t)}{\sum_{(y_1,y_2)\in D(S)} w_\epsilon(y_1,y_2)^p},
\end{equation}
if the denominator is positive and $\normalizedfunctional_{\kernel, \epsilon, p}(S)(t)=0$ otherwise. 
We omit the dependence of $\normalizedfunctional$ on $\kernel, \epsilon, p$, writing $\normalizedfunctional = \normalizedfunctional_{\kernel, \epsilon, p}$. Example functionals are shown in Figure~\ref{fig:truncated_persistence}, for the kernels introduced in Examples~\ref{example:persistence_silhouette} and~\ref{example:persistence_image}.

We will see in {Sections~\ref{sec:continuous_signatures_properties} and~\ref{sec:continuous_signatures_properties_noise}} that the normalization of the functional makes it invariant to the number of periods of $\pattern$ in $S$ to a certain extent.
The details of the construction of the persistence diagram, examples of $\normalizedfunctional$ and a study of the properties of the truncated persistence
${\pers_{p, \epsilon}(D) \coloneqq \left(\sum_{(y_1,y_2)\in D} w_\epsilon(y_1,y_2)^p\right)^{1/p}}$ are included in Section~\ref{sec:persistence_results}.

We can now define what we will call the topological signature. When $\gamma\sim\mu$ and $W\sim\nu$ are independent random variables, $S$ is also random. For each path and $t\in\T$, we can calculate $\normalizedfunctional(S)(t)\in\R$. We define the signature of $S$ point-wise as
\begin{equation}
\label{eq:signature_definition}
F(S)(t) \coloneqq \E[\normalizedfunctional(S)(t)],
\end{equation}
where the expectation is taken with respect to the law of the process, induced by the product measure of $\mu$ and $\nu$. It is clear that $\normalizedfunctional(S)(t)$ is a real-valued random variable, but we will show in Proposition~\ref{prop:measurability} that we can also consider $\normalizedfunctional(S)\in C(\FunctionalDomain,\R)$ as a random variable. 

\subsection{Properties of functionals of a periodic function}
\label{sec:continuous_signatures_properties}
We examine the consistency of the signature~\eqref{eq:signature_definition} and its invariance with respect to reparametrisations for noiseless obseravtions. That is, we consider the case $W=0$, so that~\eqref{eq:model_continuous_signal} becomes $S(t) = \pattern(\gamma(t)).$ Recall that $\gamma:[0,T]\rightarrow \R$ is a continuous and increasing function, whose distribution we denote by $\mu$.

For consistency, normalizing the functional by the total truncated $p$-persistence is akin to normalizing by the number of periods. As $\gamma(T)-\gamma(0)$ increases, the contribution of the boundary effects becomes less significant and we gain invariance to the number of observed periods. Theorem~\ref{thm:convergence_to_limit} is in fact a corollary of Lemma~\ref{lemma:limit_diagram}. It also justifies calling the limit the ``signature of a periodic function". 
\begin{theorem}[Consistency]
	\label{thm:convergence_to_limit}
	Assume that $\kernel$ satisfies~\eqref{eq:kernel_lipschitz} and~\eqref{eq:kernel_bounded_on_diag}. Then, as $R\rightarrow\infty$,
	\begin{equation*}
	\normalizedfunctional(D(\restr{\pattern}{[0,R]})) \xrightarrow{\Vert \cdot \Vert_\infty} \normalizedfunctional(D(\restr{\pattern}{[c,c+1]})).
	\end{equation*}
\end{theorem}
\begin{proof}
	Let $D_1=D(\restr{\pattern}{[c,c+1]})$, $D'$ be given by Lemma~\ref{lemma:limit_diagram} and let $D_R=D(\restr{\pattern}{[0,R]})$. In addition, we will write
	$\functional(D) = \functional_{\kernel, \epsilon, p}(D)= \sum_{x\in D}w_\epsilon(x)^p \kernel_{x}(t)$ for the linear (non-normalized) version of the functional $\normalizedfunctional$ from~\eqref{eq:definition_normalized_linear_functional}.
	Then, for any $t\in \FunctionalDomain$,
	\begin{align*}
	\left\vert\tfrac{\functional((R-1) D_1) + \functional(D')}{\pers_{p,\epsilon}^p(D_R)}
	- \tfrac{\functional(D_1)}{\pers_{p,\epsilon}^p(D_1)}\right\vert
	&\leq \left\vert \tfrac{\functional(D')}{\pers_{p,\epsilon}^p(D_R)}\right\vert
	+ \left\vert \tfrac{\functional((R-1)D_1)}{\pers_{p,\epsilon}^p(D_R)} - \tfrac{\functional(D_1)}{\pers_{p,\epsilon}^p(D_1)}\right\vert \nonumber\\
	&\leq \left\vert \tfrac{\functional(D')}{\pers_{p,\epsilon}^p(D_R)}\right\vert
	+ \left\vert \tfrac{\pers_{p,\epsilon}^p(D_1)\functional((R-1) D_1) - (\pers_{p,\epsilon}^p((R-1)D_1) + \pers_{p,\epsilon}^p(D'))\functional(D_1)}{\pers_{p,\epsilon}^p(D_R)\pers_{p,\epsilon}^p(D_1)}\right\vert \nonumber\\
	&\leq \tfrac{\vert\functional(D')\vert}{\pers_{p,\epsilon}^p(D_R)}
	+ \tfrac{\vert\pers_{p,\epsilon}^p(D')\functional(D_1)\vert}{\pers_{p,\epsilon}^p(D_R)\pers_{p,\epsilon}^p(D_1)},\nonumber
	\end{align*}
	where we have used that for any $N\in\N$,
	$$\pers_{p, \epsilon}^p(ND_1)\functional(D_1) = N\pers_{p, \epsilon}^p(D_1)\functional(D_1) = \pers_{p, \epsilon}^p(D_1)\functional(ND_1).$$
	Now, we observe that $\pers_{p,\epsilon}^p(D_R) = \pers_{p,\epsilon}^p((R-1)D_1) + \pers_{p,\epsilon}^p(D')\geq (R-1)\pers_{p,\epsilon}^p(D_1)$ and $\pers_{p,\epsilon}^p(D')\leq 2\pers_{p,\epsilon}^p(D_1)$ to obtain that
	\begin{equation}
	\label{eq:convergence_to_limit_proof}
	\Vert\normalizedfunctional(D(\restr{\pattern}{[0,R]})) -\normalizedfunctional(D(\restr{\pattern}{[c,c+1]}))\Vert_\infty
	\leq\tfrac{\vert\functional(D')\vert}{\pers_{p,\epsilon}^p(D_R)}
	+ \tfrac{\vert\pers_{p,\epsilon}^p(D')\functional(D_1)\vert}{\pers_{p,\epsilon}^p(D_R)\pers_{p,\epsilon}^p(D_1)}
	\leq \tfrac{\vert \functional(D')\vert + 2\vert \functional(D_1)\vert}{(R-1)\pers_{p,\epsilon}^p(D_1)}
	\end{equation}
	Using the Minkowski inequality,
	\begin{equation*}
	\vert \functional_t(D')\vert
	= \vert \sum_{x\in D'}w_\epsilon(x)^p \kernel_{x}(t)\vert
	\leq \sum_{x\in D'}\vert w_\epsilon(x)^p\vert \max_{x\in D'} \vert k_x(t)\vert
	\leq \pers_{p, \epsilon}^p(D') \max_{x\in D'}\Vert \kernel_x\Vert_\infty.
	\end{equation*}
	Because $\kernel$ is $L_\kernel$-Lipschitz by~\eqref{eq:kernel_lipschitz}, for any $x\in D'$, we have $\Vert \kernel_{x}\Vert \leq L_\kernel \Vert x - \pi(x)\Vert + \Vert\kernel_{\pi(x)}\Vert$,
	where $\pi(b,d) = (\tfrac{b+d}{2}, \tfrac{b+d}{2})$. Using~\eqref{eq:kernel_bounded_on_diag} on one hand, and the fact that the distance of any point in the diagram to $\Delta$ is bounded by $A_\pattern$, we obtain $\Vert \kernel_{x}\Vert \leq \tfrac{L_\kernel A_\pattern}{2} + C$.
	A similar bound holds for $\functional_t(D_1)$.
	Going back to~\eqref{eq:convergence_to_limit_proof}, we have that 
	\begin{align*}
	\Vert\normalizedfunctional(D(\restr{\pattern}{[0,R]})) -\normalizedfunctional(D(\restr{\pattern}{[c,c+1]}))\Vert_\infty
	&\leq \frac{(2\vert \pers_{p, \epsilon}^p(D_1)\vert + \vert \pers_{p, \epsilon}^p(D_1)\vert)\max_{x\in D'}\Vert \kernel_x\Vert_\infty}{(R-1)\pers_{p,\epsilon}^p(D_1)}\\
	&\leq \frac{4(C+L_\kernel A_\pattern)}{R-1},
	\end{align*}
	what converges uniformly to 0 as $R\rightarrow \infty$.
\end{proof}

Without noise, Lemma~\ref{lemma:invariance_to_reparametrisation} implies that the functional depends only on the number of periods. As a consequence, the signature $F$ is also robust to the distribution of reparametrisations, but only to a certain extent. Consider $\gamma_1\sim\mu_1$ and $\gamma_2\sim\mu_2$ such that the distributions of endpoints $(\gamma_1(0),\gamma_1(T))$ and $(\gamma_2(0),\gamma_2(T))$ are the same.
When $\mu_1$ and $\mu_2$ are such that we can condition on the endpoints, then
\begin{equation}
\label{eq:signature_invariance_noiseless}
F(\pattern\circ\gamma_1)=F(\pattern\circ\gamma_2).
\end{equation}
In light of Lemma~\ref{lemma:invariance_to_reparametrisation},~\eqref{eq:signature_invariance_noiseless} is not surprising, but requires a strong disintegration condition. That condition holds when $\mu_1,\mu_2$ are measures on a closed subspace of $(C([0,T]), \Vert\cdot\Vert_\infty)$. In particular, for any $\vmin>0$, an example is given by
\begin{equation}
\label{eq:gamma_closed_example}
\Gamma_{\vmin} = \{\gamma\in C([0,T],\R)\mid \gamma(s)-\gamma(t)\geq \vmin(s-t),\ \text{for all}\ s\geq t\}.
\end{equation}
We give more details in Appendix~\ref{appendix:disintegration}, notably, we restate~\eqref{eq:signature_invariance_noiseless} in more precise terms in Proposition~\ref{prop:invariance}.

We stress that relaxing the assumption on the equality of distributions is not straightforward. In short, the main problem lies in obtaining a fine control on the persistence diagram when `cutting' a domain, $[0,T_2]$, into $[0,T_1]$ and $[T_1,T_2]$, for any $0<T_1,T_2$. Specifically, we need to consider the difference between $D(\restr{\pattern}{[0,T_2]})$ and $D(\restr{\pattern}{[0,T_1]})\cup D(\restr{\pattern}{[T_1, T_2]})$. When $T_1$ is a global maximum of $\pattern$, we can reason as in the proof of Lemma~\ref{lemma:limit_diagram}. However, this is far from the general situation, in which case the cut at $T_1$ might induce some spurious points in the diagram. 

\subsection{Properties of functionals of noisy periodic functions}
\label{sec:continuous_signatures_properties_noise}
Consider now the noisy observations as in~\eqref{eq:model_continuous_signal}, we loose the invariance with respect to $\gamma$ as given in Lemmata~\ref{lemma:invariance_to_reparametrisation}~and~\ref{lemma:limit_diagram}.
We explore two strategies. For fixed endpoints, we control the differences produced by the noise (Theorem~\ref{thm:signature_stability_to_reparametrisation}). Otherwise, more generally, we can compare the functionals of noisy observations with the signature of the periodic function (Proposition~\ref{prop:functional_difference_via_bias}). Let us detail the assumptions on $W$ and $\gamma$.

We impose three conditions on the noise $W$, whose distribution we will denote by $\nu$. First, we assume that $\Vert W\Vert_\infty$ is almost--surely bounded by a constant smaller than the amplitude of the signal: there is $q>0$ such that $\Vert W\Vert_{\infty}\leq (A_\pattern - \epsilon- q)/2$, where ${A_\pattern= \max\pattern - \min\pattern}$. Second, we assume a path-wise regularity condition, which states that for some $0<\KolmogR<\KolmogP$,
\begin{equation}
\label{eq:Kolmogorov_condition}
\text{there exists } K=K_{\KolmogP, \KolmogR},
\text{ such that } \E[\vert W_t - W_s\vert^{\KolmogP}]\leq K_{\KolmogP, \KolmogR} \vert t-s\vert^{1+\KolmogR}, \text{ for all } s,t\in[0,T].
\end{equation}
Proposition 1.11 in~\cite{azais_level_2009} (which we restate in~\ref{proposition:kolmogorov_implies_holder}) shows that~\eqref{eq:Kolmogorov_condition} implies that $W$ has a version with $\alpha$-H\"older continuous sample paths, for any $0<\alpha<\KolmogR/\KolmogP$.
It is therefore a reasonable condition and $K_{p,r}$ can be explicitly calculated for homogeneous processes with exponentially-decreasing covariance functions.
Finally, we assume that $W$ is independent of $\gamma$. It implies that the law of $S$ is the image measure of the product of $\mu$ and $\nu$ by the map $(x,y)\mapsto \pattern(x)+y$.
\begin{remark}
	Difficulties in treating $W$ come both from controlling its amplitude and the regularity. The tools that we use are sensitive to many, small fluctuations. Condition~\eqref{eq:Kolmogorov_condition} allows us to control the regularity, without imposing a uniform H\"older character on all paths.
\end{remark}

For Theorem~\ref{thm:signature_stability_to_reparametrisation}, we assume that $\gamma$ has a lower-bounded modulus of variation and fixed endpoints. Specifically, let $0<T,R$ and consider
\begin{equation*}
\Gamma_{T,R,\vmin}\coloneqq\{\gamma\in C([0,T],[0,R])\mid\gamma(0)=0,\gamma(T)=R,\ 0\leq\vmin (t-s)\leq \gamma(t)-\gamma(s), \forall s\leq t\}.
\end{equation*}
The set $\Gamma_{T,R,\vmin}$ is convex. It is also included in $C([0,T],\R)$, so it can be naturally endowed with the sup-norm, for which it is a closed, complete and separable space. In particular, it is a Radon space, so that all measures on $(\Gamma_{T,R,\vmin}, \Borel(\Gamma_{T,R,\vmin}))$ are inner--regular and locally-finite. Hence, we can equip the space of probability measures on $(\Gamma_{T,R,\vmin}, \Borel(\Gamma_{T,R,\vmin}))$ with the Wasserstein distance $W_{1,\Vert\cdot\Vert_\infty}$~\cite{panaretosInvitationStatisticsWasserstein2020}.
Another reason for working with $\Gamma_{T,R,\vmin}$ is that $\gamma^{-1}$ all have the same domain, what allows us to take full advantage of the invariance properties of homology. Finally, the lower--bound on the modulus provides a relation between
$\Vert\gamma^{-1}_1 - \gamma^{-1}_2\Vert_\infty$ and $\Vert \gamma_1 - \gamma_2\Vert_\infty$.

\begin{theorem}[Stability]
	\label{thm:signature_stability_to_reparametrisation}
	Let $\mu_1,\, \mu_2$ be two probability measures on $\Gamma_{T,R,\vmin}$ and let $\gamma_k\sim\mu_k$, for $k=1,2$. If $p\geq 1+\max(\KolmogP,\KolmogP/(\KolmogR-1))$, $\normalizedfunctional = \normalizedfunctional_{\epsilon,p,\kernel}$,
	\begin{equation*}
	\Vert F(\pattern\circ\gamma_1 + W) - F(\pattern\circ\gamma_2+W) \Vert_\infty
	\leq
	\frac{\tilde{C}(K_{\KolmogP,\KolmogR})}{\vmin^\alpha}
	W_{1,\Vert\cdot\Vert_\infty}(\mu_1,\mu_2)^\alpha,
	\end{equation*}
	where ${\tilde{C}(x)=\mathcal{O}(x^{1/\KolmogP} (1+x^{1/(\KolmogR-1)}))}$ depends on $\pattern,\, \epsilon,\, p,\, q$ and $\kernel$.
\end{theorem}
The proof of Theorem~\ref{thm:signature_stability_to_reparametrisation} is differed to Appendix~\ref{app:proof_stability}.

Two cases show that the control in Theorem~\ref{thm:signature_stability_to_reparametrisation} is satisfying. First, suppose that $\mu_k = \delta_{\gamma_k}$ for $k=1,2$, for some fixed $\gamma_1,\gamma_2\in\Gamma_{T,R,\vmin}$. Then, we obtain that the silhouette is H\"older, with respect to the distance $\Vert \gamma_1 - \gamma_2\Vert_{\infty}$. It is expected that we do not have complete invariance: for a fixed path $W$, the reparametrisation $\gamma$ can influence how the points in the persistence diagram are displaced.
Consider now the case of vanishing noise. If $K_{\KolmogP, \KolmogR}$ decreases to zero, then so does the H\"older constant $\Lambda_W$ and we have indeed that the right-hand side becomes zero.

Note that controlling $\Vert W\Vert_\infty$ is not sufficient for the stability. When $A_W<\epsilon$, the constant factor in $\tilde{C}(x)$ is $C_{\Lambda_W}=L_\kernel(1+\tfrac{8p^2 A_\pattern(A_\pattern-\epsilon)\pers_{p-2,\epsilon}^{p-2}(\pattern)}{(R-2)q^p})$. We can take the truncation parameter $\epsilon$ small, in which case $q\approx(A_\pattern-\epsilon)$ and so, for a function with a single maximum and minimum, we have $C_{\Lambda_W}\approx L_\kernel(1+8p^2)>0$, which is not zero. Even though the amplitude of the noise is smaller than the cut-off $\epsilon$, it still has an influence on the signature. Therefore, it is important that as the amplitude decreases, the noise does not become increasingly irregular: it is the case of $aW$, with $a \rightarrow0^+$. We require the almost-sure bound on $\Vert W\Vert_\infty$ for a different reason: it gives us the lower--bound on $\pers_{p, \epsilon}^p(\pattern\circ\gamma + W)$, which appears in the denominator of $\normalizedfunctional$.

For processes of decreasing amplitude but increasingly irregular, it is more advantageous to bound $\Vert W_{\gamma^{-1}_1}-W_{\gamma^{-1}_2}\Vert_\infty\leq 2\Vert W\Vert_\infty$ in the proof. In such a scenario however, we ignore the reparametrisations so the distance $\Vert \gamma_1^{-1} - \gamma_2^{-1}\Vert_\infty$ disappears from the bound.
\begin{remark}
	When both endpoints are fixed and common to all reparametrisations, there is no reason to normalize by the total persistence. The stability comes from the continuity of the functional, not the renormalisation. Proposition~\ref{prop:functional_continuity} states that linear functionals of the form $\sum_{x\in D}w_\epsilon(x)^p \kernel_x$ are also continuous for H\"older functions, so a statement analogue to Theorem~\ref{thm:signature_stability_to_reparametrisation} also holds for such functionals.
\end{remark}

We now discuss relaxing some assumptions in Theorem~\ref{thm:signature_stability_to_reparametrisation}. First, note that the lower--bound on the modulus of continuity ($\vmin>0$)  allows us to upper--bound $\Vert \gamma_1^{-1}-\gamma_2^{-1}\Vert_\infty$ by $\tfrac{1}{\vmin}\Vert \gamma_1-\gamma_2\Vert_\infty$. But, if we remove this assumption ($\vmin=0$), it is not clear whether $\Gamma_{T,R,0}$ is a complete space for $\Vert \gamma_1^{-1}-\gamma_2^{-1}\Vert_\infty$.

Second, we could also allow $R$ to vary. A simple example is to let $\gamma_k=R \tilde{\gamma}_k$, where $R$ is a random variable on a compact set of $\rbrack 0,\infty\lbrack$ and $\tilde{\gamma}_k\sim\tilde{\mu}_k$ is a random element of $\Gamma_{T,1,\vmin}$, with $\tilde{\gamma}$ independent of $R$. In that case, we do obtain the distance $W_1(\tilde{\mu}_1,\tilde{\mu}_2)$ in the bound, but it is not clear that it lower--bounds $W_1(\mu_1,\mu_2)$.

The final extension is robustness in the case where the distributions of $\gamma_k(T)-\gamma_k(0)$ are not the same for $k=1,2$. However, we are short of understanding it already in the noiseless case, as stated in Section~\ref{sec:continuous_signatures_properties} and Appendix~\ref{appendix:disintegration}.

Below, we include Proposition~\ref{prop:functional_difference_via_bias}, a much weaker and deterministic statement valid under milder hypotheses.
\begin{proposition}
	\label{prop:functional_difference_via_bias}
	Let	$(\gamma_k:[0,T]\rightarrow [0,R_k])_{k=1,2}$ be two fixed reparametrisations, for $R_k>2$. Consider perturbations ${W_1,W_2 \in C^\alpha_\Lambda([0,T],\R)}$, with $\Vert W_k\Vert_\infty<A_\pattern/2$. Then,
	\begin{equation*}
	\Vert \normalizedfunctional(\pattern\circ\gamma_1 + W_1) - \normalizedfunctional(\pattern\circ\gamma_2 + W_2)\Vert \leq L_\kernel \left( \tfrac{4A_\pattern}{\min(R_1,R_2)-2}+ P(\max(\Vert W_1\Vert_\infty,\Vert W_2\Vert_\infty))\right),
	\end{equation*}
	where the expression of $P(x)=\mathcal{O}(x)$ is given explicitly in Lemma~\ref{lemma:perturbed_pathwise_version}.
\end{proposition}
Note that the right--hand side is strictly positive, even in the noiseless case $W=0$ and $\mu_1=\mu_2$. It is not surprising, because the bounds we use are very crude: we remove the noise and we compare the respective signatures to the limit object $\normalizedfunctional(\pattern)$. The H\"older regularity assumption on the noise is a consequence of the fact that the statement is deterministic and pathwise: a similar proof could be carried out for signatures (in expectation), using regularity assumption~\eqref{eq:Kolmogorov_condition}.

\section{Statistical inference of signatures from time--series}
\label{sec:signatures_discrete}
We have defined the signature and studied its properties for continuous observations. In practical applications, we do not have access to $S$, but to observations in the form of a time--series $(S_n)_{n=1}^N$.
In our case, this time series is composed of samples from a continuous process. In some situations, it is reasonable to assume that we have access to a collection of independent time--series from the same model, what allows to conveniently estimate $F$. We consider the case in which we observe a single time series.
The purpose of this section is to show asymptotic statistical guarantees for signatures of windows of a discretized signal.

\subsection{Time series model}
Similarly to the continuous model~\eqref{eq:model_continuous_signal}, the observations are a reparametrisation of a 1-periodic function $\pattern$
\begin{equation}
\label{eq:model_discrete_signal}
S_n = \pattern(\gamma_n) + W_n \in \R,\qquad n=1,\ldots, N,
\end{equation}
where $(\gamma_n)_{n=1}^N$ is a strictly increasing time series and $(W_n)_{n\in\N}$ is a stationary noise time series satisfying $\E[W_n]=0$.
It is also convenient and straightforward to consider the limit of a time series of infinite length, $(S_n)_{n\in\N}$.

We will present a class of reparametrisation processes, defined as discrete integrals of another, positive time series $V_n$. Specifically, let 
\begin{equation}
\label{eq:gamma_Markov_chain}
\gamma_{n+1} = \gamma_{n} + hV_n = \gamma_0 + h\sum_{k=0}^n V_k,
\end{equation}
where $(V_n)_{n=0}^N$ is a sequence of random variables in $\VInterval \coloneqq[\vmin, \vmax]$, independent of $\gamma_0$ and $0<h$ is a time step.
This model is inspired by dynamics, where the sequence $(\gamma_n)_{n\in\N}$ could model the displacement of a body over time and $V_n$ should be thought of as the instantaneous speed, in which case $h=\tfrac{T}{N}$.
We will consider two models for $(V_n)_{n\in\N}$. In the first one, consecutive velocities are independent. Since we do not expect a moving body to change speed abruptly, we also consider $V_n$ as a Markov process on $\VInterval$.
\paragraph{Model 1: $(V_n)_{n\in\N}$ \iid}
We assume that $V_n$ are independent and follow the same, unknown distribution on $\R_+^*$, which satisfies the following property: there exists $0<a,b, c$ such that, for all $A\in \Borel(\rbrack a,b\lbrack)$ measurable, $P(V_k\in A)\geq c\Lebesgue(A)$, where $\Lebesgue$ is the Lebesgue measure.

\paragraph{Model 2: $(V_n)_{n\in\N}$ a Markov Chain}
Let $(V_n)_n$ be a Markov Chain with transition kernel $\TPK$. Specifically,
\begin{enumerate}
	\item $v\mapsto \TPK(v,A)$ is $\Borel(\VInterval)$-measurable, for all $A\in\Borel(\VInterval)$,
	\item $A\mapsto \TPK(v,A)$ is a probability measure on $(\VInterval, \Borel(\VInterval))$.
\end{enumerate}
We assume that $\TPK(x,\cdot)$ is a probability measure that has a density $\density{x}$ with respect to $\Lebesgue$ and that
\begin{enumerate}
	\item the density is lower--bounded in a small neighborhood: there exists $\eta,\mu_0>0$, such that 
	\begin{equation}
	\label{eq:borel_lower_bound_density}
		\restr{\density{v}}{[v-\eta, v+\eta]\cap I} \geq \mu_0,
	\end{equation}
	\item $v\mapsto \density{v}(x)$ is continuous for any $x\in I$.
\end{enumerate}

Note that if $\density{x}=\density{},$ for all $x\in\VInterval$, we are in a particular case of the \iid\ setting, where $P$ has density $\density{}$, $a=\vmin, b=\vmax$ and $c=\mu_0$.

\begin{example}
Set $V_0\sim\Uniform(\VInterval)$ and let $0<\eta<\tfrac{\vmax-\vmin}{4}$. 
An example of a kernel satisfying the above assumption is a truncated Gaussian kernel. The truncation is such that the support is $\VInterval$ and $\sigma=\eta$. In Figure~\ref{fig:example_reparametrisations_discrete}, we show the kernel and several sample trajectories from this model.
\end{example}

\begin{figure}
	\centering
	\includegraphics[width=0.33\textwidth]{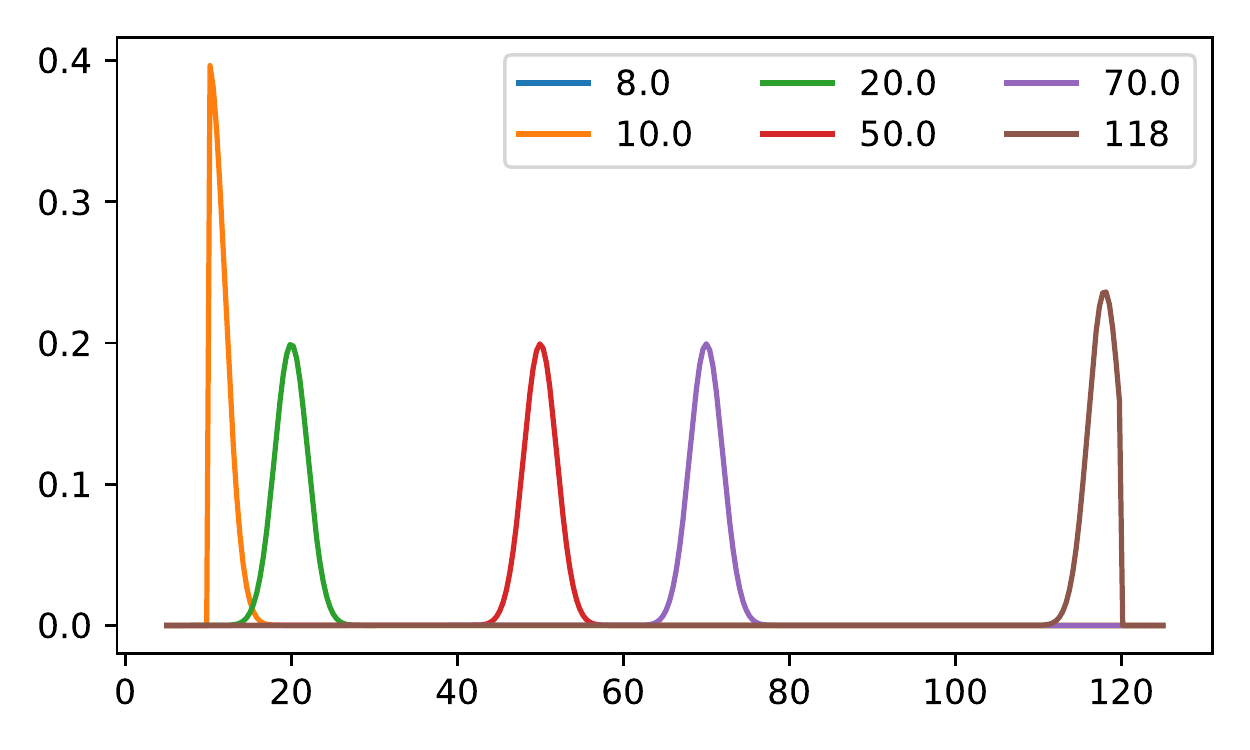}
	\includegraphics[width=0.66\textwidth]{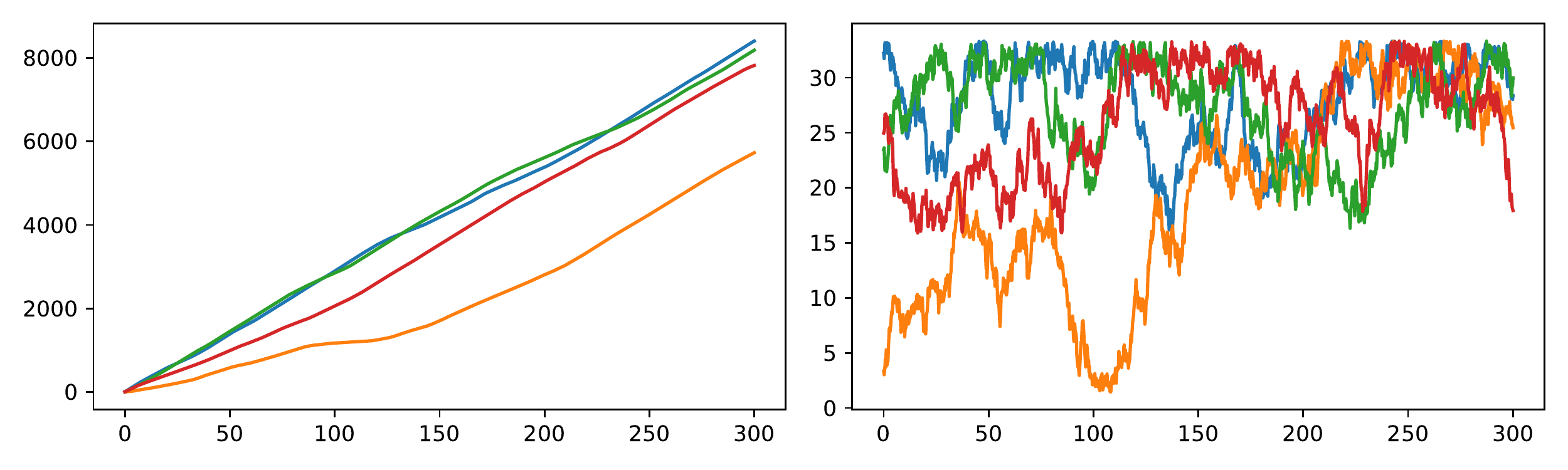}
	\caption{On the left, the truncated Gaussian Kernel centered at different point in $\VInterval$, with $\eta=2$. In the center, several reparametrisation paths, integrated from the Markov chain realisations on the right. Those were generated with $\eta= 1.1$.}
	\label{fig:example_reparametrisations_discrete}
\end{figure}

\subsection{Estimation of signatures}
Consider the situation where we observe a single time--series $(S_n)_{n=1}^N$, $S_n\in \R$. To define the functionals in this setting, we can see $(S_n)_{n=1}^N$ as a piece--wise linear function. This allows us to calculate $\normalizedfunctional$ from the resulting persistence diagram in the same way,
$\normalizedfunctional_t((S_n)_n) = \normalizedfunctional(D((S_n)_{n=1}^N)) = \frac{\sum_{x\in D}w_\epsilon(x)^p \kernel_x(t)}{\sum_{x\in D}w_\epsilon(x)^p}$. We provide more details in Section~\ref{sec:sublevelset_persistence}.

With a single observation, we cannot expect to reliably estimate $\normalizedfunctional(S)$. Persistent homology is a global descriptor, which can link two events, even if they happen far in time. Even though the descriptor $\normalizedfunctional$ effectively represents the average homological feature, it is not immediately clear that it benefits from the same properties as the empirical mean. Understanding this poses the same challenges as those explained at the end of Section~\ref{sec:continuous_signatures_properties} and Appendix~\ref{appendix:disintegration}.

Instead, we will fix a window length $M\in\N$ and we will use as a signature $F_M(S)\coloneqq F(X)=\E[\normalizedfunctional(X)]$, where $X=(S_1,\ldots S_M)$. It is a quantity which we can estimate with an empirical mean, whose distribution we can also characterize by bootstrap techniques. This choice is justified by our considerations on the continuous model, where in the noiseless case, Theorem~\ref{thm:convergence_to_limit} shows that $F_M(S)$ converges to $F(S)$, as $M\rightarrow\infty$. In the discrete (time--series) setting, the invariance to reparametrisation (Lemma~\ref{lemma:invariance_to_reparametrisation}) no longer holds: the discretisation means that the extrema of $\pattern$ are not necessarily attained by the time series and attenuation occurs at high sampling frequencies.

From $(S_n)_{n=1}^N$, we generate a sample $X_1,\ldots,X_{N-M+1}$,
\begin{equation}
\label{eq:signal_window_definition}
X_n=(S_{n},\ldots, S_{n+M-1}).
\end{equation}
The empirical counter-part of $F_M(S)$ is the empirical mean ${\hat{F}_M(S)=\tfrac{1}{N-M+1}\sum_{n=1}^{N-M+1}\normalizedfunctional(X_n)}$, whose distribution we will estimate by Moving Block Bootstrap (MBB)~\cite{buhlmannBootstrapsTimeSeries2002}. To be specific, let $L = L(N-M+1)\in\N$ be the block length. The MBB consists of sampling $B$ blocks, each composed of $L$ consecutive vectors $X$: that is, $(X_{n},\ldots, X_{n + L -1})$, for $n\in\{1,\ldots N-M+1\}$. The MBB sample is then
$$X^*_1,\ldots, X^*_{N-M+1} = X_{n_1},\ldots, X_{n_1 + L}, X_{n_2}, \ldots, X_{n_2+L},\ \ldots, X_{n_B},\ldots X_{n_B+L},$$
where $n_1,\ldots n_B\sim \Uniform(1,\ldots, N-M+1)$ are independent.
We denote by ${F^*_M(S)=\tfrac{1}{N-M+1}\sum_{n=1}^{N-M+1}\normalizedfunctional(X^*_n)}$.

Note that the bootstrap sample contains overlapping samples, at two different levels. Not only are the windows $X_1,\ldots, X_N$ overlapping, but also the different blocks can overlap.

The purpose of this section is to prove that the empirical mean $\hat{F}_M(S)$ converges to $F_M(S)$ and that we can approximate the distribution of $\hat{F}_M$ by that of $F^*_M$, as $N\rightarrow\infty$.
The core idea is to control how the dependence between $X_1$ and $X_{1+k}$ changes as $k$ increases. For this, we recall the definition of $\beta$-mixing coefficients~\cite[Section 1.2]{dedecker_weak_2007}.

For a stationary sequence $(X_n)_{n\in\N}$, denote by $\sigma_{a,b}$ the $\sigma$-algebra generated by $X_a,\ldots X_b$. The $k$-th $\beta$-mixing coefficient is
\begin{equation*}
\beta_X(k) = \tfrac{1}{2}\sup_{\substack{\mathcal{A}\subset \sigma_{-\infty, 0},\\ \mathcal{B}\subset \sigma_{k,\infty}}}\sum_{A\in \mathcal{A}, B\in\mathcal{B}} \vert P(A\cap B) - P(A)P(B)\vert,
\end{equation*}
where $\mathcal{A},\,\mathcal{B}$ are countable partitions of the sample space. We say that $(X_n)_{n\in\N}$ is \emph{absolutely regular} (or $\beta$-mixing) if $\beta_X(k)\rightarrow 0$ as $k\rightarrow\infty$. A process for which $\beta(k)\leq a^{k}$, for some $0<a<1$ is called \emph{exponentially $\beta$-mixing}.
\begin{theorem}
	\label{thm:main_clt}
	Consider $(\gamma_n)_{n=1}^N$ as in~\eqref{eq:gamma_Markov_chain} with $(V_n)_{n=1}^N$ as in Model 1 or 2.
	Assume that $W$ is exponentially $\beta$-mixing.
	Then, 
	\begin{equation}
	\label{eq:main_clt_gaussian_approx}
	\sqrt{N-M+1}(\hat{F}_M(t)- F_M(t)) \rightarrow G_d
	\end{equation}
	where $G_d$ is a zero--mean Gaussian process with covariance
	\begin{equation}
	\label{eq:limit_covariance}
	(s,t)\mapsto \lim_{k\rightarrow\infty}\sum_{n=1}^\infty\mathrm{cov}(\normalizedfunctional(X_k)(s), \normalizedfunctional(X_n)(t)).
	\end{equation}
	Then, if $L(N)\rightarrow \infty$  and $L(N)=\bigO(N^{1/2 - \epsilon})$ for some $\epsilon>0$ as $N\rightarrow\infty$, then
	\begin{equation}
	\label{eq:main_clt_bootstrap}
	\sqrt{N-M+1}(\hat{F}^* - \hat{F}) \rightarrow^* G_d(t)\qquad \text{ in probability,}
	\end{equation}
\end{theorem}
This result is a functional central limit theorem, similar to many in the literature of topological data analysis, see for example~(\cite{chazal_stochastic_2014}, \cite[Proposition 2 and 3]{berry_functional_2018}), except that the samples are not independent. For \iid~data, it is sufficient to control the complexity of the functional family. Since this aspect has been covered extensively, we only recall Proposition~\ref{prop:bracketing_number_silhouette} for completeness.
The novel aspect of Theorem~\ref{thm:main_clt} is the consideration of dependence and it is what we treat with more care.
The rest of this section is devoted to a proof of Theorem~\ref{thm:main_clt}, with details being differed to appendices.
\begin{proof}[{Sketch of proof of Theorem~\ref{thm:main_clt}}]
	Consider $(\gamma_n)_{n\in\N}$ as in~\eqref{eq:gamma_Markov_chain} with Model 1 or 2. Notice that this series is not stationary, because $P(\gamma_n<\gamma_{n+1})=1$.
	However, the crucial observation is that composition $(\pattern(\gamma_n))_{n\in\N}$ is stationary. In fact, it can be written as $\pattern(x) = \pattern(\Frac(x))$, through $\Frac(x) = x- \lfloor x\rfloor$ the fractional part of a real number.
	In Appendix~\ref{appendix:mixing_proof}, we show that $(\Frac(\gamma_n))_{n\in\N}$ is exponentially $\beta$-mixing (Proposition~\ref{prop:gamma_beta_mixing}).
	While this is not a surprising result, it is the most technical part of the proof. We show a Doeblin-type condition:
	we find a non-trivial measure which lower-bounds the $n$-step transition measure of $(\Frac(\gamma_n), V_n)$, uniformly in the initial conditions $(\Frac(\gamma_0), v_0)$.
	With the assumptions on the kernel in our model, we show that, for $n$ sufficiently large, this lower--bound can be taken to be a uniform measure on $[0,1]\times\VInterval$ with small but non--zero mass. The fact that the process is $\beta$-mixing then follows from general results in dependence theory.
	
	Next, we analyze how the dependence of $(\pattern(\gamma_n))_{n\in\N}$ and $(W_n)_{n\in\N}$ shapes the dependence of $(S_n)_{n\in\N}$ and that between the windows $X_1,\ldots, X_{N-M+1}$. Specifically, Appendix~\ref{appendix:mixing_by_measurable_mapping} contains a proof of the following inequality
	\begin{equation*}
	\beta_X(k) \leq \beta_S(k -(M+1)) \leq \beta_{\Frac(\gamma)}(k-(M+1)) + \beta_W(k -(M+1)),\qquad \text{ for $k\geq M+1$}.
	\end{equation*}
	Since $(W_n)_{n\in\N}$ is exponentially mixing by assumption, $(X_n)_{n\in\N}$ is exponentially-mixing.
	
	The Gaussian approximation~\eqref{eq:main_clt_gaussian_approx} is a consequence of Theorem~\ref{thm:kosorok}. By the arguments above, the mixing condition~\eqref{eq:beta_condition} is verified for $X$ for any $r>2$. It remains to verify that the bracketing entropy of the functional family $\{\normalizedfunctional_t\}_{t\in \T}$ is controlled. This is done in Proposition~\ref{prop:bracketing_number_silhouette}.
	
	The approximation of the distribution of the empirical mean by the bootstrap distribution~\eqref{eq:main_clt_bootstrap} is a consequence of Theorem~\ref{thm:buhlmann}, for which we only need the aforementioned results.
\end{proof}

\begin{remark}
The literature of functional central limit theorems for dependent data is rich in results for various functional classes and dependence assumptions. We believe it might be possible to use more recent and stronger results than Theorem~\ref{thm:buhlmann}. This would allow us to relax the decay of $\beta_W$ from an exponential to a polynomial one. For instance,~\cite[Theorem 1]{radulovicBootstrapEmpiricalProcesses1996} is written for VC-classes functionals, but the proof seems to rely on the bracketing entropy bound that the functionals considered in the present work also satisfy.
\end{remark}

\subsection{Discussion}
Theorem~\ref{thm:main_clt} motivates the use of $\normalizedfunctional$ as a descriptor of a phase--modulated, periodic signal. A possible application of the asymptotic guarantees is the construction of valid confidence intervals. While it is tempting to use the proposed framework to test for $\pattern_1=\pattern_2$ based on observations $S_k=\pattern_k(\gamma_k) + W_k$, we do not provide theoretical guidance on how to calibrate such a test.

We do not present a theory for the choice of the window length $M$. Increasing $M$ reduces the probability of not capturing a whole period in a window of length $M$. In addition, it also reduces the variance of the signature due to a non--integer number of periods in $X$ (Lemma~\ref{thm:convergence_to_limit} provides a bound in the noiseless case $W=0$), but it reduces the total number of observations.

Modeling $((\gamma_n, V_n))_{n\in \N}$ as a Markov Chain of order 1 is restrictive. For applications where $\gamma_n$ represents a position in time, we should rather specify $\gamma_n$ as generated by the acceleration $(a_n)_{n\in\N}$, itself a Markov chain, possibly with hidden states. We believe that under similar ergodicity assumptions, similar kinds of arguments should be sufficient to show that such models lead to exponentially-mixing reparametrisation sequences $(\Frac(\gamma_n))_n$.

\section{Persistent homology and its functional representations}
\label{sec:persistence_results}
We provide basic background information on homology of sublevel sets and total $p$-persistence in subsection~\ref{sec:sublevelset_persistence}. We also introduce the truncated $p$-persistence and show some continuity properties. Then, we detail the definition of $\functional$ and $\normalizedfunctional$.

\subsection{Persistent homology for uni-dimensional signals}
\label{sec:sublevelset_persistence}
We briefly describe the construction of the persistence diagram of sub level--sets of a continuous function $h:\X\rightarrow \R$.
For $t\in\R$, we consider the $t$-sublevel set of $h$, $\X_t = h^{-1}(\rbrack -\infty, t\rbrack)$ and we define the module
\begin{equation}
\label{eq:persistent_module_definition}
V_t =
\begin{cases*}
H_0(\X_t),&\text{ if $t<\max h$}\\
0, &\text{ otherwise},
\end{cases*}
\end{equation}
where $H_0$ is the $0$-dimensional singular homology. For any $s\leq t<\max h$, the inclusion $\X_s\rightarrow \X_t$ induces a morphism between the singular homology groups $\iota_{s}^{t}: V_s\rightarrow V_t$. For $t\geq \max h$, $\iota_{s}^{t}$ is the zero morphism. We call $\V = ((V_t)_{t\in\R}, (\iota_{s}^{t})_{s<t\in\R})$ the persistence module associated to $h$.

When $\X$ is compact and $h$ continuous, the persistence module is q-tame~\cite[Theorem 3.33]{chazalStructureStabilityPersistence2016} and we define the persistence diagram $D(h)$ to be the multi--set associated to the rectangular measure $\mu_\V$ by~\cite[Theorem 3.19]{chazalStructureStabilityPersistence2016}. In our case, $\X=[0,T]$, so the persistence diagram is well-defined and encodes the values of local minima and maxima, with the following properties
\begin{itemize}
	\item each point from $\Delta \coloneqq \{(r,r)\mid r\in\R\}$ is in $D(h)$ with infinite multiplicity,
	\item for any $s<t\in\R$, $\vert\{(b,d)\mid b\leq s<t\leq d\}\vert =\rank{V}{s}{t}<\infty$.
\end{itemize}
For more details, see~\cite{chazalStructureStabilityPersistence2016}.

The above describes the persistence diagram of a continuous function $h:[0,T]\rightarrow\R$. When the input data is a time series of length $M$, we can define a function by discretizing $[0,T]$, prescribing the values at the nodes and linearly interpolating in between. Carrying out the construction described above, we obtain a persistence diagram that has at most $\tfrac{M}{2}$ points. In that case, $\dim(H_0(\X_r))\leq M<\infty$ for all $r\in\R$, so the persistence module has an interval decomposition: it is isomorphic to $(\oplus_{(b,d)\in L} I_{\lbrack b,d\lbrack}(r))$, where $I_{\lbrack b,d\lbrack}$ is an interval module. Representing each interval by its endpoints, its diagram $D_h$ is the multiset $\{(b,d)\mid (b,d)\in L\}$.

One distance which is often used to compare diagrams is the bottleneck distance
\begin{equation*}
d_b(D_1,D_2) = \inf_{\Gamma} \sup_{x\in D_1\cup\Delta} \Vert x-\Gamma(x)\Vert_\infty,
\end{equation*}
where $\Gamma: D_1\cup\Delta \rightarrow D_2\cup\Delta$ is a bijection between the two diagrams, which allows some points to be matched to the diagonal $\Delta$. With respect to the supremum norm between functions, the persistence diagram is stable in that distance.
\begin{theorem}[{\cite{cohen-steinerStabilityPersistenceDiagrams2007, chazalStructureStabilityPersistence2016}}]
	\label{thm:bottleneck_stability}
	For two functions $f,g:\X\rightarrow\R$ with persistence diagrams $D_f$ and $D_g$ respectively,
	\begin{equation*}
	d_b(D_f, D_g)\leq \Vert f-g\Vert_\infty.
	\end{equation*}
\end{theorem}

\begin{remark}
	The persistence module that we define in~\eqref{eq:persistent_module_definition} is slightly different to those usually considered in the literature. For $\X=[0,T]$ and $t\geq \max h$, $H_0(\X_t) = H_0([0,T])$, which is not trivial. Using the usual definition $(H_0(\X_t))_{t\in\R}$, we would obtain exactly one essential point in the persistence diagram. By modifying the module and setting $V_t\equiv 0$ for $t\geq \max h$, we make sure that there are no essential components. While it changes the decorated persistence diagrams, for the usual (undecorated) diagram, it amounts to setting the death value for the essential component to $\max h$. The module is still locally--finite, stable in the bottleneck distance. The choice of~\eqref{eq:persistent_module_definition} is motivated by the proof Lemma~\ref{lemma:limit_diagram}.
\end{remark}

\subsection{Total truncated $p$-persistence}
The total persistence of a persistence diagram quantifies the oscillations of the filtering function. It is similar to total variation for functions on the interval~\cite{plonka_relation_2016}. The \emph{persistence} of $(b,d)\in \R^2$ is $w(b,d)\coloneqq d-b$. For $p\in\N^+$, the total $p$-persistence of a persistence diagram $D$ is the sum of $p$-powers of the lifetimes of points, 
$\pers_p(D) = \left(\sum_{(x,y)\in D} (d-b)^p\right)^{1/p}.$

In the case of sublevel set persistence, points with small persistence might be attributed to noise and quantify the regularity of the function, while the more persistent ones capture the biggest oscillations of $\pattern$. The functionals we propose in Section~\ref{sec:functionals} use persistence and total persistence to give different weights to certain features, reflecting the intuition given above. In order to study the stability of the signatures with respect to the generating process, we need some lower-- and upper-bounds, as well as stability of total persistence with respect to the input function.

Continuous functions on compact domains are bounded and attain their extremal values, but, similarly to total variation, it is not enough to bound their total persistence because of possible small oscillations. In fact, the total $p$-persistence of $\alpha$-H\"older functions is finite for $p>1/\alpha$, but it is not continuous for functions with regularity strictly less than Lipschitz~\cite{perez_c0-persistent_2022}.
As a remedy,~\cite{perez_c0-persistent_2022} suggest to truncate the persistence of $D\cap \Delta_\epsilon$, for some $\epsilon>0$, where $\Delta_\epsilon\coloneqq\{(b,d)\in D\mid d-b\geq \epsilon\}$. This leads to a bounded total persistence, but continuity with respect to the sup norm is lost, even on very regular functions.

To guarantee both boundedness and continuity, we introduce the truncated persistence. Let $w_\epsilon(b,d)\coloneqq (d-b-\epsilon)_+$, where $(a)_+ = \max(a,0)$ denotes the positive part. The $\epsilon$-truncated total $p$-persistence is
\begin{equation*}
\pers_{p,\epsilon}(D) = \left(\sum_{(x,y)\in D} w_\epsilon(b,d)^p\right)^{1/p}
\end{equation*}
and $\support(w_\epsilon)\subset \Delta_\epsilon \coloneqq \{(b,d)\in \R^2\mid d-b\geq \epsilon\}$.

We can think of $\epsilon$-truncated persistence as shifting the diagonal by $\epsilon$ in the normal direction, $\tfrac{1}{\sqrt{2}}(-1,1)$,
$\pers_{p,\epsilon}^p(D) = \pers_p^p(\{(b-\epsilon/2, d-\epsilon/2)\mid (b,d)\in D\})$, as illustrated in Figure~\ref{fig:truncated_persistence}. Proposition~\ref{prop:truncated_persistence_continuity_bottleneck} shows that truncated persistence is continuous in the bottleneck distance between diagrams. Note that the modulus of continuity in the proof is not uniform, since it depends on the number of points and the maximal persistence of a point in the diagram.
\begin{proposition}
	\label{prop:truncated_persistence_continuity_bottleneck}
	The $\epsilon$-truncated $p-$persistence is continuous with respect to the bottleneck distance. 
\end{proposition}
\begin{proof}
	Consider a persistence diagram $D_1$. By the second property from the persistence diagrams, we have $M\coloneqq\vert D_1\cap\Delta_{\epsilon/4}\vert<\infty$ and $d-b \leq U<\infty$, for some $U\in\R$. Let $D_2$ be such that $d_B(D_1,D_2)<\epsilon/4$. Then, $\vert D_2\cap\Delta_{\epsilon}\vert \leq \vert D_1\cap\Delta_{\epsilon/2}\vert\leq M$ and the persistence of a point in $D_2$ is bounded by $U+\epsilon/2$. Trivially, the truncated persistence of a point is 2-Lipschitz,
	\begin{equation*}
	\vert w_\epsilon(b,d) - w_\epsilon(b', d') \vert \leq (d-b -\epsilon)_+ - (b' - d' - \epsilon)_+ \leq \vert d-b-(d' - b')\vert\leq 2\Vert (b,d) - (b', d')\Vert_\infty.
	\end{equation*}
	Then, we use the technique from the proof of the~\cite[Total Persistence Stability Theorem]{cohen-steiner_lipschitz_2010}: writing $\vert x_2^p-x_1^p\vert = \vert p\int_{x_1}^{x_2} t^{p-1}dt\vert \leq p \vert x_2 -x_1\vert \max(x_1^{p-1}, x_2^{p-1})$, we get
	\begin{align*}
	\left\vert \sum_{x\in D_1} w_\epsilon(x)^p - w_\epsilon(\Gamma(x))^p \right\vert
	\leq& p\sum_{D_1} \vert w_\epsilon(x) - w_\epsilon(\Gamma(x))\vert (w_\epsilon(x)^{p-1} + w_\epsilon(\Gamma(x))^{p-1})\\
	\leq& 4p M(U+\epsilon/2)^{p-1} d_B(D_1,D_2).
	\end{align*}
\end{proof}

Let us now go back to functions: by abuse of notation, we will define $\pers_{p,\epsilon}(h)\coloneqq\pers_{p,\epsilon}(D_h)$. In the special case of H\"older functions, their $\epsilon$-total truncated $p$-persistence is bounded.
\begin{proposition}
	\label{prop:upper_bound_p_persistence}
	Let $h\in C^\alpha_\Lambda([0,T], \R)$. For $p\geq 0$ such that $(p-1)\alpha >1$,
	\begin{equation*}
	\pers_{p,\epsilon}^p(h)\leq (A_h-\epsilon)^p\left(1+pT\left(\tfrac{2\Lambda}{\epsilon}\right)^{1/\alpha}\right)
	\eqqcolon C_{p, \Lambda,\alpha,T},
	\end{equation*}
	where $A_h\coloneqq \max h- \min h$ is the amplitude of $h$.
\end{proposition}
By compacity of $[0,T]$, $A_h$ is finite so the upper-bound is not trivial.
Using Proposition~\ref{prop:truncated_persistence_continuity_bottleneck}, we could immediately show that total $p$-persistence is also continuous with respect to the input function. However, we can show that it is Lipschitz, following the proof of~\cite[Lemma 3.20]{perez_c0-persistent_2022}. The argument is based on a H\"olders' inequality and the uniform upper-bound on persistence from Proposition~\ref{prop:upper_bound_p_persistence}. The proof of Propositions~\ref{prop:upper_bound_p_persistence} and~\ref{prop:continuity_truncated_persistence} are presented in Appendix~\ref{appendix:proof_continuity_truncated_persistence}.
\begin{proposition}[{Continuity of truncated $p$--persistence}]
	\label{prop:continuity_truncated_persistence}
	The total truncated persistence ${\pers_{p,\epsilon}^p: C([0,T],\R)\rightarrow \R}$ is continuous.
	In addition, $\pers_{p,\epsilon}^p$ is Lipschitz over H\"older functions: for any $f,g\in C^\alpha_\Lambda([0,T])$ such that $p-1>\tfrac{1}{\alpha}$,
	\begin{equation*}
	\begin{aligned}
	\vert \pers_{p,\epsilon}^p(f) - \pers_{p,\epsilon}^p(g)\vert
	&\leq p \Vert f-g\Vert_\infty \left(\pers_{p-1,\epsilon}^{p-1}(f) + \pers_{p-1,\epsilon}^{p-1}(g)\right)\\
	&\leq C_{p-1,\Lambda,\alpha,T}\Vert f-g\Vert_\infty.
	\end{aligned}
	\end{equation*}
\end{proposition}

\begin{figure}
	\centering
	\begin{minipage}{0.4\textwidth}
		\includestandalone[width=.9\textwidth]{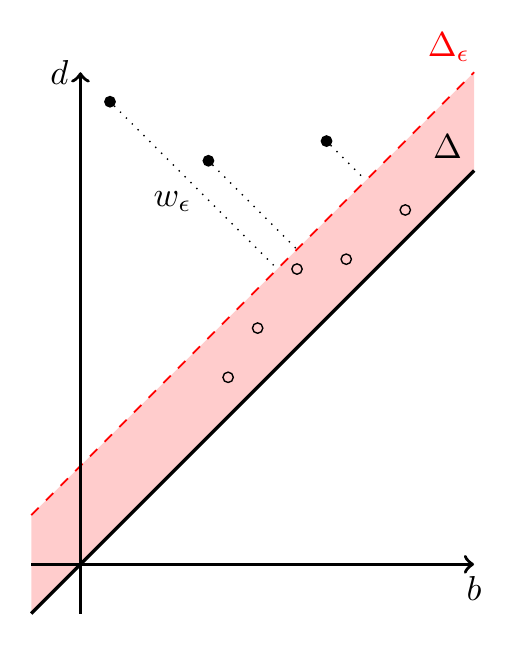}
	\end{minipage}
	\begin{minipage}{0.5\textwidth}
		\centering
		\includegraphics[width=.65\textwidth]{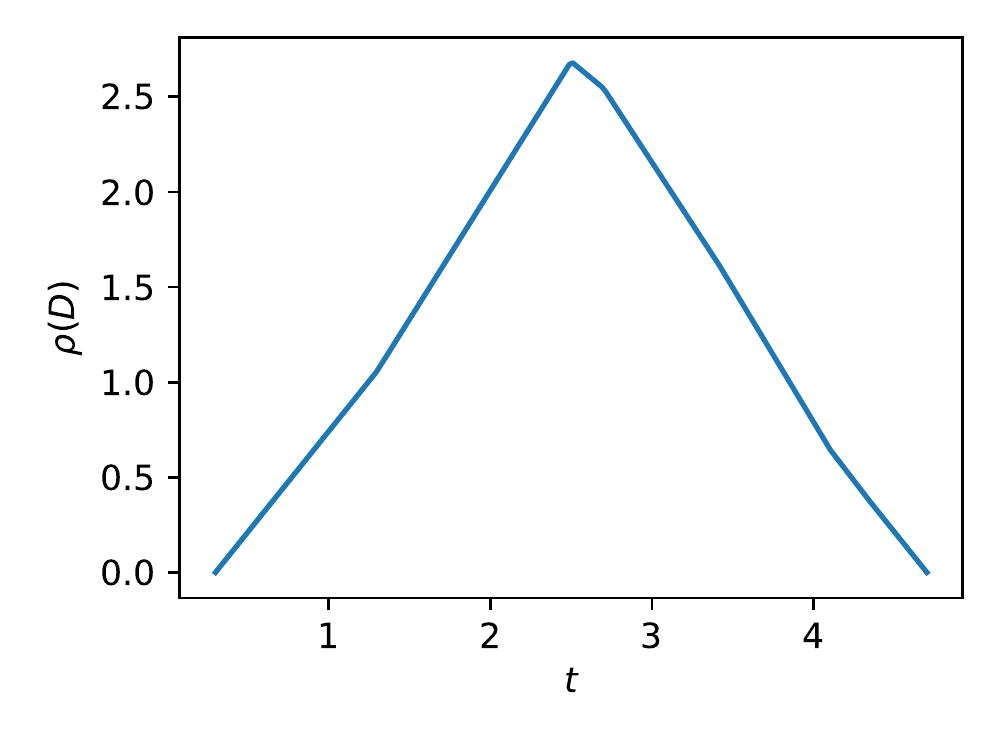}
		\includegraphics[width=.69\textwidth]{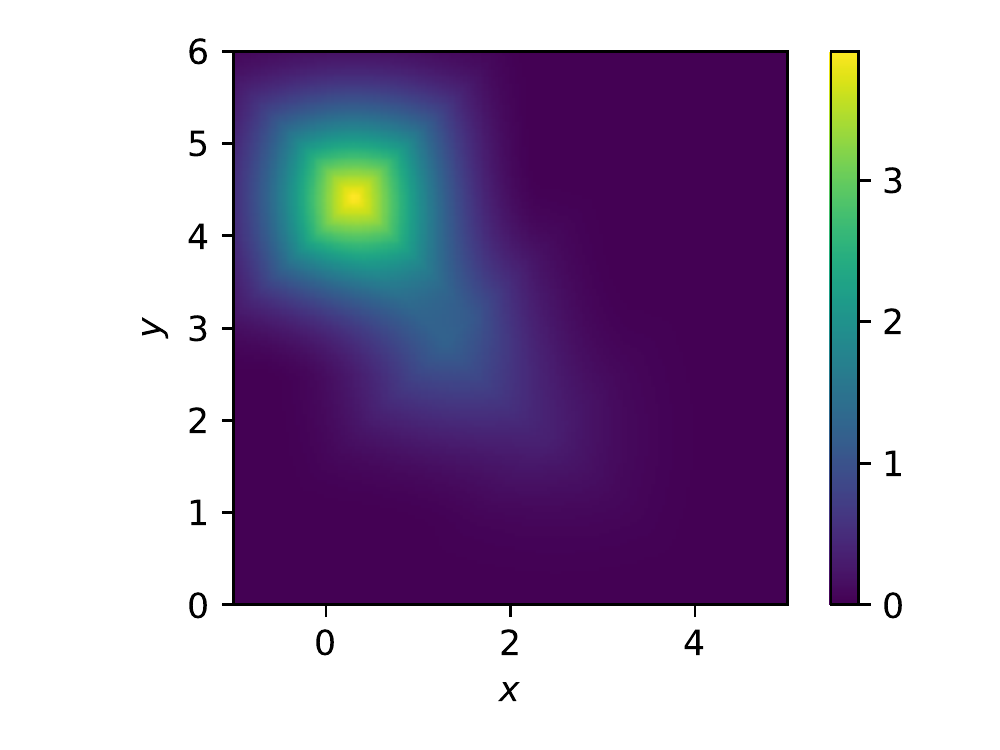}
	\end{minipage}
	\caption{On the left, a persistence diagram and $\Delta_\epsilon$, the diagonal shifted by $\epsilon=1$. The truncated persistence $w_\epsilon$ is twice the distance to $\Delta_\epsilon$. The points in the shaded region have zero truncated persistence. On the right, two functionals evaluated at the $\normalizedfunctional$ with $\kernel^{s}$ at the top and $\functional$ with $\kernel^{pi,t}$ $(\sigma=1, r=1.1)$, both weighted by the truncated persistence $w_\epsilon$ with $p=2$.}
	\label{fig:truncated_persistence}
\end{figure}
Finally, we give a lower--bound for the truncated persistence. Such a result will be necessary to show the continuity of the class of normalized functionals $\normalizedfunctional$, which we define in Section~\ref{sec:functionals}.
\begin{proposition}[{Lower-bound on $p-$persistence}]
	\label{prop:lower_bound_p_persistence}
	For continuous functions $f, W:[0, T]\rightarrow \R$,
	\begin{equation*}
	\pers_{p,\epsilon}^p(f+W) \geq \pers_{p,\epsilon + A_W}^p(f).
	\end{equation*}
\end{proposition}
\begin{proof}
	Since $\pers$ is translation-invariant ($\pers_{p, \epsilon}(f+c) = \pers_{p, \epsilon}(f)$, for any constant $c>0$), we can assume that $A_W = 2\Vert W\Vert_\infty$. Let $\Gamma: D(f)\rightarrow D(f+W)$ be a matching between the diagrams and denote by $c(\Gamma)$ the associated cost. Thanks to the bottleneck stability theorem, $\inf_\Gamma c(\Gamma)\leq\Vert W\Vert_\infty$.
	Then, for any $(b,d)\in D(f)$ and $(b',d')=\Gamma((b,d))\in D(f+W)$, we have $d'-b' \geq d - b - 2c(\Gamma)$ and, for any $\delta>0$, ${D(f)\cap \Delta_{2c(\Gamma) + \delta}\subset \Gamma^{-1}(D(f+W)\cap \Delta_{\delta})}$. Then,
	\begin{align*}
	\pers_{p,\epsilon}^p(f+W)
	&= \sum_{(b',d')\in D(f+W)} w_\epsilon(b',d')^p \\
	&\geq \sum_{(b',d')\in D(f+W) \cap \Delta_{\delta}} w_\epsilon(b', d')^p \\
	&\geq \sum_{(b,d) \in \Gamma^{-1}(D(f+W) \cap \Delta_\delta)} w_\epsilon((b, d) - c(\Gamma)(-1,1))^p \\
	&\geq \sum_{(b,d)\in D(f) \cap \Delta_{2c(\Gamma)+\delta}} w_{\epsilon+2c(\Gamma)}(b,d)^p.
	\end{align*}
	For $\delta=\epsilon$, the last quantity is equal to $\pers_{p, \epsilon+2c(\Gamma)}^p(f)$. By taking the infimum over all matchings $\Gamma$, we obtain $\pers_{p,\epsilon}^p(f+W)\geq \pers_{p, \epsilon+ 2\Vert W\Vert_\infty}^p(f)$.
\end{proof}
The result is very weak, but it is tight. If we take $f$ such that $\max f - \min f = 2\Vert f \Vert_\infty$ and $W=-\alpha f$, then $f+W = (1-\alpha)f$ and $\Vert W\Vert_\infty = \alpha \Vert f\Vert_\infty$. Then, 
$\pers_{p, \epsilon}^p((1-\alpha)f) = \pers_{p, \epsilon + 2\alpha}^p(f)$.

\subsection{Linear and normalized functionals}
\label{sec:functionals}
The space of persistence diagrams is not a vector space and is ill-suited for statistical purposes. It is common to map diagrams to a functional Banach space~\cite{chazalIntroductionTopologicalData2021a}. Many such mapping have been proposed~\cite{carrierePersLayNeuralNetwork2020a, bubenik_statistical_2015, adams_persistence_2017, chung_persistence_2022} and their properties are studied extensively. We will distinguish between linear and normalized functionals.
As it is usually the case with functionals of persistence, we present a general set of assumptions and we show examples of functionals from the literature (or of their adaptation) which fit within the prescribed framework.

Consider $(\FunctionalDomain, d)$ a Euclidean space and let $\FunctionalSpace \subset \{\FunctionalDomain\rightarrow \R\}$ be a functional Banach space.
Finally, let $k:\R^2\rightarrow \FunctionalSpace$ be a map, which to a point $(b,d)$ in the plane associates a function $k(b,d)$.
For a persistence diagram $D$ with $\pers_{p,\epsilon}(D)>0$, the linear and the normalized functionals are of the form
\begin{equation}
\label{eq:definition_normalized_linear_functional}
\functional(D) = \sum_{x\in D} w_\epsilon(x)^p\kernel(x),\qquad \normalizedfunctional(D) = \frac{\functional(D)}{\sum_{x\in D} w_\epsilon(x)^p}.
\end{equation}
Otherwise, $\functional(D)=0=\normalizedfunctional(D)$.
In this work, we are interested exclusively in diagrams of sublevel sets of functions defined on a compact interval. Therefore, we will abuse notation and write $\functional(f)\coloneqq\functional(D(f))$.
Compared to how the functionals are usually introduced, we use the $\epsilon$-truncated $p$-persistence instead of $p$-persistence. As shown below in Proposition~\ref{prop:functional_continuity}, this guarantees their continuity but leads to some problems, notably because truncation can make non--empty diagrams empty. 
\begin{proposition}
	For any continuous function $f:[0,T]\rightarrow \R$ with $\max f - \min f>\epsilon$, the linear and normalized functionals are well-defined.
\end{proposition}
\begin{proof}
	Since $f$ is continuous on a compact domain, it is also uniformly continuous and bounded. Let $\delta>0$ be such that $\vert f(t)-f(s)\vert < \epsilon$, whenever $\vert s-t\vert<\delta$.
	By the reasoning of the proof of ~\cite[Persistence Cycle Lemma]{cohen-steiner_lipschitz_2010},
	$\vert \omega^{-1}(\rbrack \epsilon, \infty\lbrack)\cap D(f)\vert \leq \tfrac{T}{\delta} + 1$.
	Let $M_f = \max(f),\, m_f = \min(f)$. Then,
	\begin{equation*}
	\Vert\functional(D(f))\Vert_\infty \leq \sum_{(b,d)\in D(f)} w_\epsilon(d-b)^p \Vert\kernel(b,d) \Vert_\infty \leq (\tfrac{T}{\delta}+1)\cdot w(M_f - m_f)\max_{(b,d)\in D(f)\cap \Delta_\epsilon^+}\Vert \kernel(b,d)\Vert_\infty.
	\end{equation*}
	As stated above, the number of points is bounded from above, and so is the total persistence. For the normalized functional,
	\begin{equation*}
	\Vert \normalizedfunctional(D)\Vert \leq (\tfrac{T}{\delta}+1)\max_{(b,d)\in D(f)\cap \Delta_\epsilon^+}\Vert \kernel(b,d)\Vert_\infty
	\overbrace{\tfrac{w_\epsilon(M_f - m_f)^p}{\sum_{x\in D}w_\epsilon(x)^p}}^{\leq 1}.
	\end{equation*}
\end{proof}

We only consider functionals $\normalizedfunctional$ with $k$ which satisfies the following assumptions:
\begin{enumerate}
	\item $\kernel(x)$ has a uniformly bounded support, for all $x\in\R^2$
		\begin{equation}
		\label{eq:kernel_compact_support}
		\exists K\subset \T\text{ compact}, \restr{\kernel(x)}{\T\setminus K} \equiv 0,\qquad \forall x.
		\end{equation}
	\item $k(x)$ is Lipschitz, uniformly over $x\in\R^2$
		\begin{equation}
		\label{eq:kernel_lipschitz_in_time}
		\exists L>0,\ \vert \kernel(x)(t) - \kernel(x)(s)\vert \leq L d(s, t), \qquad \forall x\in\R^2,\, \forall s,t\in \FunctionalDomain.
		\end{equation}
	\item $x\mapsto k(x)$ is Lipschitz
		\begin{equation}
		\label{eq:kernel_lipschitz}
		\exists L_\kernel>0, \Vert \kernel(x) - \kernel(x')\Vert_\FunctionalSpace \leq L_\kernel \Vert x-x'\Vert_\infty,\qquad \forall\ x,x'\in\R^2.
		\end{equation}
	\item $k(x)$ is uniformly-bounded on the diagonal
		\begin{equation}
		\label{eq:kernel_bounded_on_diag}
		\exists C\leq 0,\ \Vert \restr{\kernel}{\Delta}\Vert_\infty \leq C. 
		\end{equation}
\end{enumerate}
As we will see later, hypotheses (\ref{eq:kernel_lipschitz},\ref{eq:kernel_bounded_on_diag}) ensure continuity of the functional, while (\ref{eq:kernel_compact_support},\ref{eq:kernel_lipschitz_in_time}) control the complexity of the family of functionals, in a way adapted to statistical results.

Many functionals proposed in the literature do not satisfy~\eqref{eq:kernel_compact_support} as is. To adapt them this assumption, we precompose the usual kernels with a projection, which we illustrate in Figure~\ref{fig:projection}. Specifically, let $L<U\in\R$ and consider $\Proj_{L,U}:\Delta_{\geq 0}\rightarrow \Delta_{\geq 0}$ the operator which maps points above the diagonal, onto the upper triangle with corner at $(L,U)$
\begin{equation}
\label{eq:projection}
\begin{array}{r r c l}
\Proj_{L,U}: &\Delta_{\geq 0}&\rightarrow&\Delta_{\geq 0}\\
&(b,d)&\mapsto&(b,d) + (1,-1)\min(\max(d-U, L-b, 0), \tfrac{d-b}{2}).
\end{array}
\end{equation}
\begin{figure}
	\centering
	\includestandalone[width=0.3\textwidth]{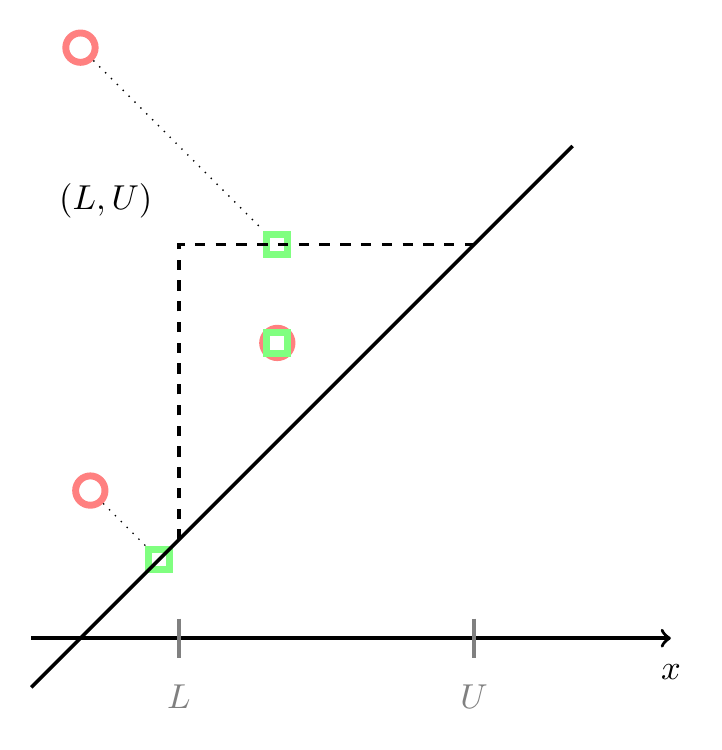}
	\caption{Illustration of the projection in~\eqref{eq:projection}. Points from a diagram and their projections are illustrated with circular and square markers.}
	\label{fig:projection}
\end{figure}
\begin{remark}
	It is common in the topological data analysis literature that the proposed functionals do not satisfy~\eqref{eq:kernel_compact_support}. Instead, it is assumed that all realisations off functionals have a compact support~(\cite{berry_functional_2018}) or that all diagrams have uniformly bounded birth and death values~\cite{chazal_stochastic_2014}. While in some cases, such assumptions are compatible with the model for the data, we do not make such assumptions in Section~\ref{sec:signatures_discrete}.
\end{remark}
Two examples of such functionals are given below and sample realisations are shown in Figure~\ref{fig:truncated_persistence}. The calculations of the Lipschitz constants are carried out in Appendix~\ref{appendix:persistence_image_lipschitz_constant}.
\begin{example}[Persistence Silhouette]
\label{example:persistence_silhouette}
The persistence silhouette~\cite{chazal_stochastic_2014} is a weighted sum of landscape functions $\Lambda_{(b,d)}(t)=\left(\tfrac{d-b}{2} - \vert t-\tfrac{b+d}{2}\vert\right)_+$, for $\FunctionalDomain= \R$. It is clear that $\support(\Lambda^s(b,d))=\lbrack b,d\rbrack$. We set $\kernel^s(b,d)(t) = \Lambda_{(\Proj_{L,U}(b,d))}(t)$., so that $\support(\kernel^s(b,d))\subset [L,U]$. Since $t\mapsto \kernel^s(b,d)$ is piece--wise linear with slopes $0$, $1$ and $-1$, $L=2=L_\kernel$.
The kernel is zero on the diagonal, so $C=0$ is enough to satisfy~\eqref{eq:kernel_bounded_on_diag}.
\end{example}

\begin{example}[Persistence Image]
\label{example:persistence_image}
The kernel that would correspond to the persistence image~\cite{adams_persistence_2017} is
${\kernel^{pi}(b,d)(x,y) = \tfrac{1}{2\pi\sigma^2}\exp\left(-\tfrac{(b-x)^2+(d-y)^2}{2\sigma^2}\right)}$, for some $\sigma>0$.
We propose to modify this functional to have bounded support. Unfortunately, a simple truncation is not enough, because the kernel would not be continuous at the truncation interface. In order to preserve the Lipschitz character, we propose to multiply by the distance to a square of size $2\sigma$ to $(b,d)$, namely, for some $r>1$, set
\begin{equation*}
\kernel^{pi,r}(b,d)(x,y)=\left(2-\tfrac{\Vert \Proj_{L,U}(b,d)-(x,y)\Vert_\infty}{\sigma}\right)_+^r \kernel^{pi}(\Proj_{L,U}(b,d))(x,y)
\end{equation*}
Thus, the original persistence image kernel corresponds to $r=0$ and $L=\infty, U=\infty$.
The function $(x,y)\mapsto \exp(-(x^2+y^2))$ is {$(4/e)$-Lipschitz} and $(x,y)\mapsto \left(2-\tfrac{\Vert (b,d)-(x,y)\Vert_\infty}{\sigma}\right)_+^r$ is $(r2^r/\sigma)$-Lipschitz, for the Minkowski distance. Hence, $L_{k^{pi,r}}=\tfrac{2^{r-1}}{\pi\sigma^3}\left(r + 2\right)$ and $L=\tfrac{2^{r+1}}{\pi e \sigma^3}$.
\end{example}

\begin{remark}
	We note a few differences with PersistenceCurves introduced in~\cite{chung_persistence_2022}. In that article, the aggregation operator can be different from the sum used here. However, the vectorizations are only curves, i.e $\FunctionalDomain=\R$. In addition, for normalized functionals, the authors restrict themselves to kernels of the form $k(b,d)(t) = c\indicator{[b,d]}(t)$, for some $c>0$.
\end{remark}

Continuity of functionals has been studied, notably in~\cite{divolChoiceWeightFunctions2019} and~\cite{chung_persistence_2022}. In the first, it was fully characterized, but only for linear functionals. In the latter, functionals were considered under the $L^1$ metric. Due to the nature of the statistical results in Section~\ref{sec:signatures_discrete}, we are particularly interested in $\Vert\cdot\Vert_\infty$, so we repeat the proof of~\cite[Theorem 3]{divolChoiceWeightFunctions2019} for linear functionals $\functional$ and we derive results for normalized functionals $\normalizedfunctional$.

\begin{proposition}
	\label{prop:functional_continuity}
	Suppose that the persistence of any point in $D_1$ and $D_2$ is bounded by a uniform constant $U$ and that $\kernel$ satisfies~\eqref{eq:kernel_compact_support},~\eqref{eq:kernel_lipschitz} and~\eqref{eq:kernel_bounded_on_diag}. Then,
	\begin{align}
	\Vert \functional(D_1) - \functional(D_2)\Vert_\infty \leq&   \left(L_\kernel \pers_{p,\epsilon}^p(D_1) + p(L_\kernel U + C)(\pers_{p-1,\epsilon}^{p-1}(D_1) + \pers_{p-1,\epsilon}^{p-1}(D_2))\right)d_B(D_1,D_2),\label{eq:functional_bottleneck_bound}\\
	\Vert \normalizedfunctional(D_1) - \normalizedfunctional(D_2)\Vert_\infty \leq& \left(L_\kernel + 2p(L_\kernel U + C)\frac{\pers_{p-1,\epsilon}^{p-1}(D_1) + \pers_{p-1,\epsilon}^{p-1}(D_2)}{\pers_{p,\epsilon}^p(D_1)}\right) d_B(D_1, D_2).\label{eq:normalizedfunctional_bottleneck_bound}
	\end{align}
\end{proposition}
\begin{proof}
	Let $\Gamma: D_1\rightarrow D_2$ be a matching between the two diagrams. For any $t\in\FunctionalDomain$,
	\begin{align*}
	\vert \functional(D_1)(t) - \functional(D_2)(t)\vert
	\leq& \sum_{x \in D_1} w_\epsilon(x)^p\vert \kernel(x)(t) - \kernel(\Gamma(x))(t)\vert + \kernel(\Gamma(x))(t)\vert w_\epsilon(x)^p - w_\epsilon(\Gamma(x))^p\vert\\
	\leq& \sup_{x\in D_1}\vert k(x)(t)-k(\Gamma(x))(t)\vert \sum_{x\in D_1}w_\epsilon(x)^p\\
	&+ \sup_{x\in D_1}\vert k(\Gamma(x))(t)\vert \sum_{x\in D_1}\vert w_\epsilon(x)^p - w_\epsilon(\Gamma(x))^p\vert\\
	\leq&\ L_\kernel d_B(D_1,D_2)\pers_{p,\epsilon}^p(D_1)\\
	&+ p(L_\kernel U + C)\sum_{x\in D_1}\vert w_\epsilon(x)-w_\epsilon(\Gamma(x))\vert(w_\epsilon(x)^{p-1} +w_\epsilon(\Gamma(x))^{p-1}),
	\end{align*}
	where in the last inequality, we used that $\Vert \kernel(\Gamma(x))\Vert_\infty\leq L_\kernel \Vert (x_1,x_2) - (\tfrac{x_1+x_2}{2}, \tfrac{x_1+x_2}{2})\Vert_\infty + \Vert \kernel\left(\tfrac{x_1+x_2}{2}, \tfrac{x_1+x_2}{2}\right) \Vert_\infty = L_\kernel\tfrac{x_2-x_1}{2}+C$.
	The sum in the second term can be bounded from above by $d_B(D_1,D_2)(\pers_{p-1,\epsilon}^{p-1}(D_1) + \pers_{p-1,\epsilon}^{p-1}(D_2))$.
	
	Consider now the normalized version.
	\begin{align*}
	\vert\normalizedfunctional(D_1)(t) - \normalizedfunctional(D_2)(t)\vert
	\leq& \frac{\vert \functional(D_1)(t) - \functional(D_2)(t)\vert}{\sum_{x\in D_1}w_\epsilon(x)^p} + \normalizedfunctional(D_2)\frac{\vert \sum_{x\in D_1}w_\epsilon(x)^p - \sum_{y\in D_2}w_\epsilon(y)^p\vert}{\sum_{x\in D_1}w_\epsilon(x)^p}\\
	\leq& d_B(D_1,D_2)\left(L_\kernel + p(L_\kernel U + C)\frac{\pers_{p-1,\epsilon}^{p-1}(D_1) + \pers_{p-1,\epsilon}^{p-1}(D_2)}{\pers_{p,\epsilon}^p(D_1)}\right)\\
	&+ p(L_\kernel U + C)d_B(D_1,D_2)\frac{\pers_{p-1,\epsilon}^{p-1}(D_1) + \pers_{p-1,\epsilon}^{p-1}(D_2)}{\pers_{p,\epsilon}^p(D_1)}\\
	\leq&\left(L_\kernel + 2p(L_\kernel U+C)\frac{\pers_{p-1,\epsilon}^{p-1}(D_1) + \pers_{p-1,\epsilon}^{p-1}(D_2)}{\pers_{p,\epsilon}^p(D_1)}\right)d_B(D_1,D_2).
	\end{align*}
	Combine $\pers_{p-1,\epsilon}^{p-1}(D_1) + \pers_{p-1,\epsilon}^{p-1}(D_2)\leq 2\max_{k=1,2} \pers_{p-1,\epsilon}^{p-1}(D_k)$ with the observation that the bound is symmetric so that we can have $\pers_{p,\epsilon}^p(D_2)$ in the denominator.
\end{proof}
\begin{remark}
	The result we give for $\functional$ is a special of~\cite[Theorem 3]{divolChoiceWeightFunctions2019}. To see this, notice that using the notations of that article, $Lip(\phi) = L_\kernel$, $A=p$, and $\alpha=p$, where `$p$' is from our work. In their article, $p=\infty$ and $a=1$. In particular, we see exactly that $\Vert \kernel\Vert_\infty \leq L_\kernel U + C$.
\end{remark}

For a function $f$, all points in its persistence diagram have birth value at least $\min f$ and a death value of at most $\max f$, so that $U=2(\max f -\min f)$ is sufficient.
Using Proposition~\ref{prop:truncated_persistence_continuity_bottleneck}, we can conclude from~\eqref{eq:functional_bottleneck_bound} (resp.~\eqref{eq:normalizedfunctional_bottleneck_bound}) that $\functional$ (resp. $\normalizedfunctional$) is continuous with respect to the bottleneck distance, on the space of persistence diagrams. Via stability of the diagram with respect to the input from Theorem~\ref{thm:bottleneck_stability}, it translates to continuity with respect to the input function.

In Section~\ref{sec:signatures_discrete}, we show convergence of functionals on random functions. These results are conditioned on controlling the complexity of the family $\FunctionalFamily \coloneqq (\normalizedfunctional_t)_{t\in K}$, where $\normalizedfunctional_t:\R^M\rightarrow \R$. In particular, we need the bracketing entropy of $\FunctionalFamily$ to be finite. It is a well-known result and consequence of~(\ref{eq:kernel_compact_support},~\ref{eq:kernel_lipschitz_in_time}).
\begin{proposition}
	\label{prop:bracketing_number_silhouette}
	Let $N_{[]}(\epsilon, \FunctionalFamily, \Vert\cdot\Vert)$ denote the bracketing number of $\FunctionalFamily$, with brackets $[u,l]$ of size $\Vert u-l\Vert\leq\epsilon$. Consider $\normalizedfunctional$ as in~\eqref{eq:normalizedfunctional} with $\kernel$ satisfying~(\ref{eq:kernel_compact_support}, \ref{eq:kernel_lipschitz_in_time}). Then, for any probability measure $P$ on $\R^M$ and $r\geq 1$,
	\begin{equation*}
	N_{[]}(\epsilon, \{\normalizedfunctional_t\}_{t\in \FunctionalDomain}, \Vert\cdot\Vert_{L_r(P)})\leq \frac{2^{D+1}L^D\ \mathrm{diam}(K)}{\epsilon^D},
	\end{equation*}
	where $D$ is the dimension of $\FunctionalDomain$.
	As a consequence, the bracketing entropy $J_{[]}(\infty, \FunctionalFamily, \Vert\cdot\Vert_{L_r(P)})$ is finite
	\begin{equation*}
	J_{[]}(\infty, \FunctionalFamily, \Vert\cdot\Vert_{L_r(P)}) \coloneqq \int_0^\infty \sqrt{\log N_{[]}(\epsilon, \FunctionalFamily, \Vert\cdot\Vert_{L_r(P)})} d\epsilon <\infty.
	\end{equation*}
\end{proposition}
\begin{proof}
	First, since $P$ is a probability measure, $\Vert \normalizedfunctional_t\Vert_{L_r(P)} = \left(\int \vert \normalizedfunctional_t\vert^r dP\right)^{1/r}\leq \Vert \normalizedfunctional_t\Vert_\infty \int dP = \Vert \normalizedfunctional_t\Vert_\infty$, so $N_{[]}(\epsilon, \{\normalizedfunctional_t\}_{t\in \FunctionalDomain}, \Vert\cdot\Vert_{L_r(P)})\leq N_{[]}(\epsilon, \{\normalizedfunctional_t\}_{t\in \FunctionalDomain}, \Vert\cdot\Vert_\infty)$.
	Combining~\eqref{eq:kernel_lipschitz_in_time} with the fact that $\normalizedfunctional(x)$ is a weighted average of $\kernel$, for any $x\in\R^M$ and $s,\,t \in \FunctionalDomain$, the normalized functional is $L$-Lipschitz in time
	\begin{equation*}
	\vert\normalizedfunctional_t(x) - \normalizedfunctional_s(x)\vert \leq L d(t,s).
	\end{equation*}
	Let $K$ be given by~\eqref{eq:kernel_compact_support}. Then,~\cite[Theorem 9.22]{kosorok_introduction_2008} states that
	\begin{equation*}
	N_{[]}(2\epsilon L, \{\normalizedfunctional_t\}_{t\in K}, \Vert\cdot\Vert_\infty) \leq N(\epsilon, K, d),
	\end{equation*}
	where $N(\epsilon, K, d)$ is the covering $\epsilon$-number of $(K,d)$.
	By assumption, $\FunctionalDomain$ is of finite dimension that we will denote by $D$. By compacity of $K$, it has a finite diameter, say $U$. Therefore, $N(\epsilon, K, d)\leq \max(1,\tfrac{U}{\epsilon^D}).$
	
	Let $t_0\notin K,\, t_1\in K$. We have that $\normalizedfunctional_{t_1}$ is uniformly bounded,
	\begin{equation*}
	\vert \normalizedfunctional_{t_1}(x)\vert \leq \vert \normalizedfunctional_{t_0}(x)\vert + L d(t_0,t_1) = L d(t_0,t_1),
	\end{equation*}
	so that $\normalizedfunctional_{t_0}=0\in [\normalizedfunctional_{t_1}-\epsilon L, \normalizedfunctional_{t_1}-\epsilon L]$, for $\epsilon > d(t_0, t_1)$.
	The brackets in the proof of~\cite[Theorem 9.22]{kosorok_introduction_2008} are of the form $[\normalizedfunctional_t - \epsilon L, \normalizedfunctional_t + \epsilon L]$, so that $N_{[]}(2\epsilon L, \{\normalizedfunctional_t\}_{t\in \FunctionalDomain}, \Vert\cdot\Vert_\infty) \leq N(\epsilon, K, d)$ for $\epsilon>d(t_0, t_1)$. In particular, one bracket is enough for $\epsilon>\max(U^{1/D}, d(t_0, t_1))$, while, for $\epsilon\leq \max(U^{1/D}, d(t_0, t_1))$, we have
	$N_{[]}(2\epsilon L, \{\normalizedfunctional_t\}_{t\in \FunctionalDomain}, \Vert\cdot\Vert_\infty)\leq 1+ N_{[]}(2\epsilon L, \{\normalizedfunctional_t\}_{t\in K}, \Vert\cdot\Vert_\infty)\leq 1+N(\epsilon, K,d)\leq 2N(\epsilon,K,d)$.
	
	Finally, since $L_r(P)$ is dominated by $\Vert \cdot\Vert_\infty$ for any probability measure $P$,
	\begin{align*}
	J_{[]}(\delta, \{\normalizedfunctional_t\}_{t\in \FunctionalDomain}, L_r(P)) 
	&= \int_{0}^\delta \sqrt{\log(N_{[]}(\epsilon, \{\normalizedfunctional_t\}_{t\in \FunctionalDomain}, L_r(P)))}d\epsilon\\
	&\leq \int_{0}^\delta \sqrt{\log(N_{[]}(\epsilon, \{\normalizedfunctional_t\}_{t\in \FunctionalDomain}, \Vert\cdot\Vert_\infty))}d\epsilon\\
	&\leq \int_{0}^{\min(\delta, 2L\max(U^{1/D}, d(t_0, t_1)))} \sqrt{\log(N(\tfrac{\epsilon}{2L}, K, d))}d\epsilon\\
	&\leq \int_{0}^{\min(\delta, 2L\max(U^{1/D}, d(t_0, t_1)))} \sqrt{\log\left(2^{D+1}U L^D\right) - \frac{1}{D}\log(\epsilon)}d\epsilon.
	\end{align*}
	As $\lim_{\delta\rightarrow 0} \int_\delta^1 \sqrt{-\log(\epsilon)}d\epsilon<\infty$, we conclude that $J_{[]}(\delta, \{\normalizedfunctional_t\}_{t\in \FunctionalDomain}, L_r(P))<\infty$.
\end{proof}

\section{Numerical illustration}
\label{sec:numerical_experiments}
To illustrate the signatures and their stability, we propose to estimate the signatures of processes with different periodic functions. Then, we compare the estimate to the signature of a process with a different reparametrisations.

We will consider periodic functions $\pattern_1$ and $\pattern_4$ defined by
\begin{equation*}
\pattern_\theta = \theta(\sin(6\pi t) + \vert t-\lfloor t\rfloor - \tfrac{1}{2}\vert - \tfrac{1}{2}) + 5\sin(4\pi t),\qquad \text{for }\theta\in\R.
\end{equation*}
The observed signal follows the discrete model~\eqref{eq:model_discrete_signal}, with $T=30$ and a sampling rate of $50$Hz.
The reparametrisations are generated by integrating twice a Markov chain of accelerations, with a truncated Gaussian transition kernel. The noise is a Gaussian process with covariance
\begin{equation*}
\Gamma(s,t)=\sigma^2\exp\left(-\frac{(s-t)^2}{2\tau^2}\right).
\end{equation*}
We fix the temporal scale $\tau$, but we vary $\sigma=0.1,\, 0.5,\, 2.$ to illustrate the impact of noise on the signature.

For $\normalizedfunctional$, we take the silhouette introduced in Example~\ref{example:persistence_silhouette}, where the weights are the $0.2$-truncated $1$-persistence ($\epsilon=0.2$, $p=1$) and we use the projection $\Proj_{-9,9}$ as in~\eqref{eq:projection}. We infer the signatures on $3$-second windows ($M=3\cdot 50$). We construct the $1\%$-confidence intervals by resampling $200$ times, with block lengths of 2 seconds ($L=2\cdot 50$).

In Figure~\ref{fig:different_patterns}, for the same random realization $\gamma_1$, we calculate the empirical signature $\hat{F}$ for $\pattern_1$ and $\pattern_4$, and estimate the corresponding confidence intervals for $F$. For low noise levels, the variance due to the number of observations and the variability in the endpoints is small, compared to the difference between the functionals. As the noise level increases, the observed function looses its recurrent appearance and the signatures become dominated by the noise.
\begin{figure}
	\includegraphics[width=0.99\textwidth]{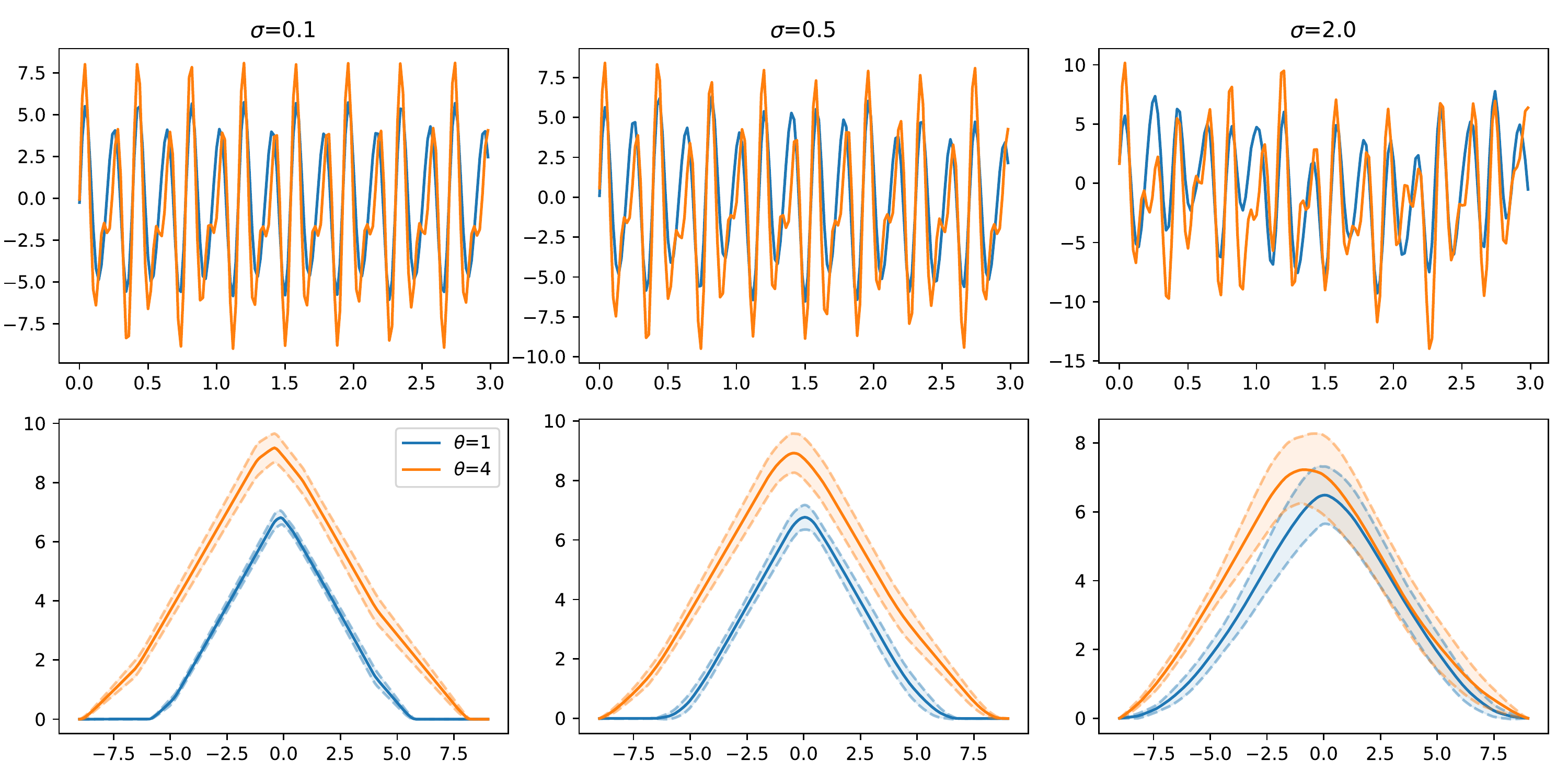}
	\caption{Signatures of $\pattern_1$ and $\pattern_4$, estimated on reparametrized signals described above. The top row shows the first 3-second window from the $30$-second signal, for both functions. The bottom row shows the estimated signatures and the confidence intervals.}
	\label{fig:different_patterns}
\end{figure}

Consider now two observations with the same periodic function $\pattern_1$, but different reparametrisations $\gamma_1,\,\gamma_2$. In Figure~\ref{fig:same_patterns}, we can see that for small values of noise, the signatures are close, what confirms their invariance to reparametrisation. It is worth noting that the signals contain different numbers of periods. For more noisy observations, the signatures lose the robustness.
\begin{figure}
	\includegraphics[width=0.99\textwidth]{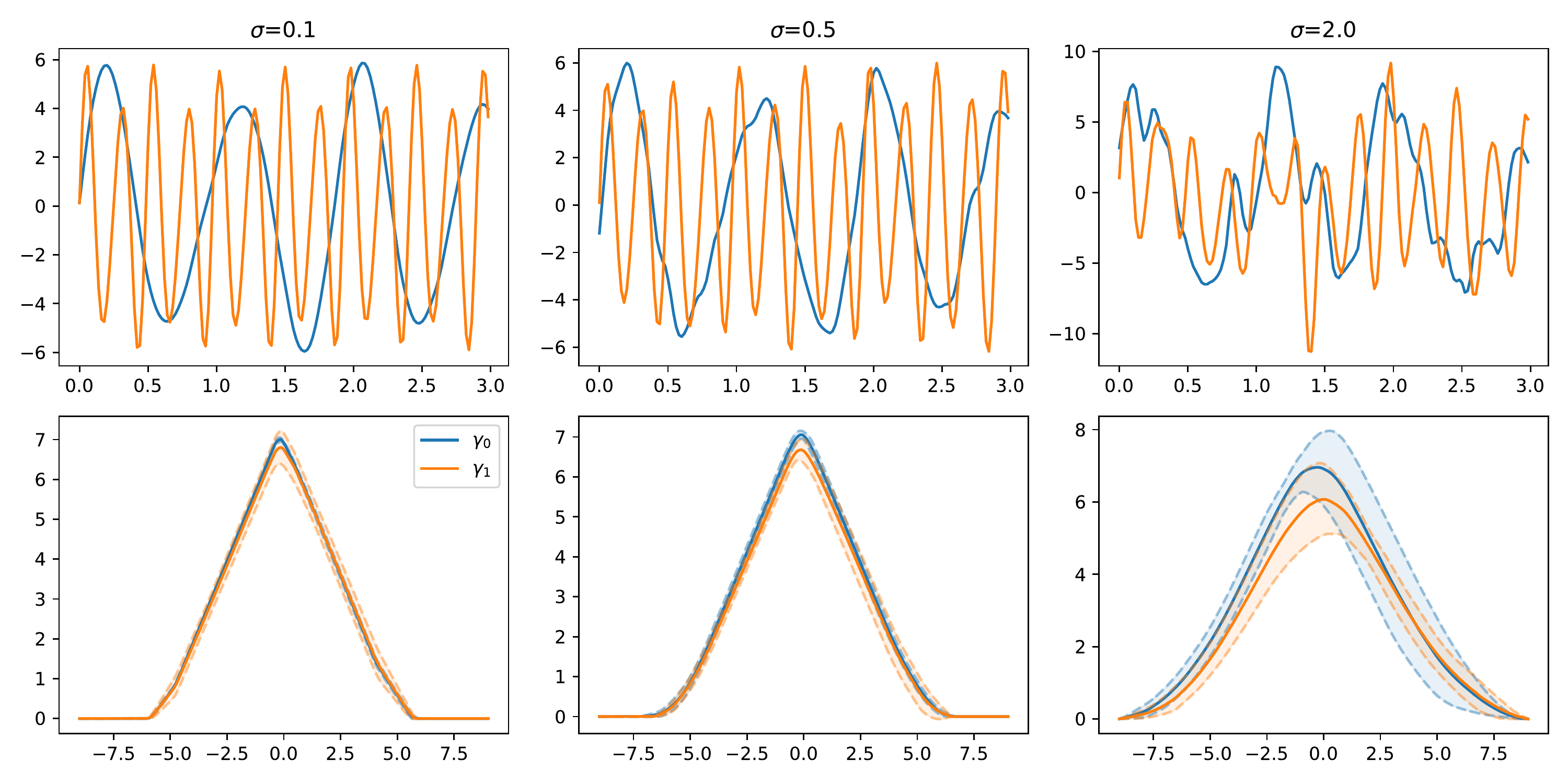}
	\caption{Signatures of $\pattern_1$, estimated on two different reparametrized observations. The top row shows the first 3-second window from the two observed signals. The bottom row shows the estimated signatures and the confidence intervals.}
	\label{fig:same_patterns}
\end{figure}

\section*{Acknowledgments}
The authors are grateful to numerous colleagues for the fruitful discussions and they wish to particularly thank Paul Doukhan, Giovanni Peccati, Alex Delalande, Quentin M\'erigot and Daniel Perez. WR thanks the ANR TopAI chair (ANR--19--CHIA--0001) for financial support.

\bibliographystyle{apalike}
\bibliography{references}

\begin{thebibliography}{}

\bibitem[Adams et~al., 2017]{adams_persistence_2017}
Adams, H., Emerson, T., Kirby, M., Neville, R., Peterson, C., Shipman, P.,
  Chepushtanova, S., Hanson, E., Motta, F., and Ziegelmeier, L. (2017).
\newblock Persistence {Images}: {A} {Stable} {Vector} {Representation} of
  {Persistent} {Homology}.
\newblock {\em The Journal of Machine Learning Research}, 18(1):218--252.

\bibitem[Az\"ais and Wschebor, 2009]{azais_level_2009}
Az\"ais, J.-M. and Wschebor, M. (2009).
\newblock {\em Level {Sets} and {Extrema} of {Random} {Processes} and
  {Fields}}.
\newblock John Wiley \& Sons, Inc., Hoboken, NJ, USA.

\bibitem[Berry et~al., 2018]{berry_functional_2018}
Berry, E., Chen, Y.-C., Cisewski-Kehe, J., and Fasy, B.~T. (2018).
\newblock Functional {Summaries} of {Persistence} {Diagrams}.
\newblock {\em arXiv:1804.01618}.

\bibitem[Bogachev, 2007]{bogachevMeasureTheory2007}
Bogachev, V.~I. (2007).
\newblock {\em Measure Theory}.
\newblock {Springer}, {Berlin ; New York}.

\bibitem[Bois et~al., 2022]{boisTopologicalDataAnalysisbased2022}
Bois, A., Tervil, B., Moreau, A., {Vienne-Jumeau}, A., Ricard, D., and Oudre,
  L. (2022).
\newblock A topological data analysis-based method for gait signals with an
  application to the study of multiple sclerosis.
\newblock {\em PLOS ONE}, 17(5).

\bibitem[Bonis et~al., 2022]{bonisTopologicalPhaseEstimation2022}
Bonis, T., Chazal, F., Michel, B., and Reise, W. (2022).
\newblock Topological phase estimation method for reparameterized periodic
  functions.

\bibitem[Bubenik, 2015]{bubenik_statistical_2015}
Bubenik, P. (2015).
\newblock Statistical {Topological} {Data} {Analysis} using {Persistence}
  {Landscapes}.
\newblock {\em Journal of Machine Learning Research}, 6:77--102.

\bibitem[B{\"u}hlmann, 1995]{buhlmannBlockwiseBootstrapGeneral1995}
B{\"u}hlmann, P. (1995).
\newblock The blockwise bootstrap for general empirical processes of stationary
  sequences.
\newblock {\em Stochastic Processes and their Applications}, 58(2):247--265.

\bibitem[B{\"u}hlmann, 2002]{buhlmannBootstrapsTimeSeries2002}
B{\"u}hlmann, P. (2002).
\newblock Bootstraps for {{Time Series}}.
\newblock {\em Statistical Science}, 17(1).

\bibitem[Carri{\`e}re et~al., 2020]{carrierePersLayNeuralNetwork2020a}
Carri{\`e}re, M., Chazal, F., Ike, Y., Lacombe, T., Royer, M., and Umeda, Y.
  (2020).
\newblock {{PersLay}}: {{A Neural Network Layer}} for {{Persistence Diagrams}}
  and {{New Graph Topological Signatures}}.
\newblock In {\em Proceedings of the {{Twenty Third International Conference}}
  on {{Artificial Intelligence}} and {{Statistics}}}, volume 108, pages
  2786--2796. {PMLR}.

\bibitem[Chazal et~al., 2016]{chazalStructureStabilityPersistence2016}
Chazal, F., de~Silva, V., Glisse, M., and Oudot, S. (2016).
\newblock {\em The {{Structure}} and {{Stability}} of {{Persistence Modules}}}.
\newblock Number 2191-8198 in {{SpringerBriefs}} in {{Mathematics}}. {Springer,
  Cham}.

\bibitem[Chazal et~al., 2014]{chazal_stochastic_2014}
Chazal, F., Fasy, B.~T., Lecci, F., Rinaldo, A., and Wasserman, L. (2014).
\newblock Stochastic {Convergence} of {Persistence} {Landscapes} and
  {Silhouettes}.
\newblock In {\em Annual {Symposium} on {Computational} {Geometry} -
  {SOCG}'14}, pages 474--483, Kyoto, Japan. ACM Press.

\bibitem[Chazal and Michel, 2021]{chazalIntroductionTopologicalData2021a}
Chazal, F. and Michel, B. (2021).
\newblock An {{Introduction}} to {{Topological Data Analysis}}: {{Fundamental}}
  and {{Practical Aspects}} for {{Data Scientists}}.
\newblock {\em Frontiers in Artificial Intelligence}, 4.

\bibitem[Chung and Lawson, 2022]{chung_persistence_2022}
Chung, Y.-M. and Lawson, A. (2022).
\newblock Persistence {Curves}: {A} canonical framework for summarizing
  persistence diagrams.
\newblock {\em Advances in Computational Mathematics}, 48(1):6.

\bibitem[{Cohen-Steiner} et~al.,
  2007]{cohen-steinerStabilityPersistenceDiagrams2007}
{Cohen-Steiner}, D., Edelsbrunner, H., and Harer, J. (2007).
\newblock Stability of {{Persistence Diagrams}}.
\newblock {\em Discrete \& Computational Geometry}, 37(1):103--120.

\bibitem[Cohen-Steiner et~al., 2010]{cohen-steiner_lipschitz_2010}
Cohen-Steiner, D., Edelsbrunner, H., Harer, J., and Mileyko, Y. (2010).
\newblock Lipschitz {Functions} {Have} {L} p -{Stable} {Persistence}.
\newblock {\em Foundations of Computational Mathematics}, 10(2):127--139.

\bibitem[Corcoran and Jones, 2017]{corcoranModellingTopologicalFeatures2017}
Corcoran, P. and Jones, C. (2017).
\newblock Modelling {{Topological Features}} of {{Swarm Behaviour}} in
  {{Space}} and {{Time With Persistence Landscapes}}.
\newblock {\em IEEE Access}, PP:1--1.

\bibitem[Dedecker, 2007]{dedecker_weak_2007}
Dedecker, J. (2007).
\newblock {\em Weak dependence: with examples and applications}.
\newblock Number 190 in Lecture notes in statistics. Springer, New York.

\bibitem[Divol and Polonik, 2019]{divolChoiceWeightFunctions2019}
Divol, V. and Polonik, W. (2019).
\newblock On the choice of weight functions for linear representations of
  persistence diagrams.
\newblock {\em Journal of Applied and Computational Topology}, (3):249--283.

\bibitem[Doukhan, 1995]{doukhanMixing1995}
Doukhan, P. (1995).
\newblock {\em Mixing}, volume~85 of {\em Lecture {{Notes}} in {{Statistics}}}.
\newblock {Springer New York, NY}.

\bibitem[Fern{\'a}ndez and Mateos,
  2022]{fernandezTopologicalBiomarkersRealtime2022}
Fern{\'a}ndez, X. and Mateos, D. (2022).
\newblock Topological biomarkers for real-time detection of epileptic seizures.

\bibitem[Gasser and Wang, 1997]{gasserAlignmentCurvesDynamic1997}
Gasser, T. and Wang, K. (1997).
\newblock Alignment of curves by dynamic time warping.
\newblock {\em The Annals of Statistics}, 25(3):1251--1276.

\bibitem[Ghil and Sciamarella, 2023]{ghilReviewArticleDynamical2023}
Ghil, M. and Sciamarella, D. (2023).
\newblock Review {{Article}}: {{Dynamical Systems}}, {{Algebraic Topology}},
  and the {{Climate Sciences}}.
\newblock Preprint, {Bifurcation, dynamical systems, chaos, phase transition,
  nonlinear waves, pattern formation/Climate, atmosphere, ocean, hydrology,
  cryosphere, biosphere/Theory}.

\bibitem[Gidea and Katz, 2018]{gideaTopologicalDataAnalysis2018a}
Gidea, M. and Katz, Y. (2018).
\newblock Topological {{Data Analysis}} of {{Financial Time Series}}:
  {{Landscapes}} of {{Crashes}}.
\newblock {\em Physica A: Statistical Mechanics and its Applications},
  491:820--834.

\bibitem[Goldberger et~al.,
  2000]{goldbergerPhysioBankPhysioToolkitPhysioNet2000}
Goldberger, A.~L., Amaral, L.~A., Glass, L., Hausdorff, J.~M., Ivanov, P.~C.,
  Mark, R.~G., Mietus, J.~E., Moody, G.~B., Peng, C.~K., and Stanley, H.~E.
  (2000).
\newblock {{PhysioBank}}, {{PhysioToolkit}}, and {{PhysioNet}}: Components of a
  new research resource for complex physiologic signals.
\newblock {\em Circulation}, 101(23):E215--220.

\bibitem[{Herbert Edelsbrunner} et~al.,
  2023]{herbertedelsbrunnerWindowPersistence1D2023}
{Herbert Edelsbrunner}, {Ranita Biswas}, {Sebastiano Cultrera di Montesano},
  and {Morteza Saghafian} (2023).
\newblock Window to the {{Persistence}} of {{1D Maps}}: {{Geometric
  Characterization}} of {{Critical Point Pairs}}.
\newblock {\em Journal of Applied Computational Topology}, 48(2):47.

\bibitem[Hu and Le, 2013]{hu_multiparameter_2013}
Hu, Y. and Le, K. (2013).
\newblock A multiparameter {Garsia}-{Rodemich}-{Rumsey} inequality and some
  applications.
\newblock {\em Stochastic Processes and their Applications}, 123(9):3359--3377.

\bibitem[Khasawneh and Munch, 2016]{khasawnehChatterDetectionTurning2016}
Khasawneh, F.~A. and Munch, E. (2016).
\newblock Chatter detection in turning using persistent homology.
\newblock {\em Mechanical Systems and Signal Processing}, 70--71:527--541.

\bibitem[Khorram et~al., 2019]{khorramTrainableTimeWarping2019}
Khorram, S., McInnis, M.~G., and Provost, E.~M. (2019).
\newblock Trainable {{Time Warping}}: {{Aligning Time-Series}} in the
  {{Continuous-Time Domain}}.
\newblock In {\em {{ICASSP}} 2019 - 2019 {{IEEE International Conference}} on
  {{Acoustics}}, {{Speech}} and {{Signal Processing}} ({{ICASSP}})}, pages
  3502--3506.
\newblock Comment: ICASSP 2019.

\bibitem[Kosorok, 2008]{kosorok_introduction_2008}
Kosorok, M.~R. (2008).
\newblock {\em Introduction to {Empirical} {Processes} and {Semiparametric}
  {Inference}}.
\newblock Springer {Series} in {Statistics}. Springer New York, New York, NY.

\bibitem[Krebs, 2021]{krebsLimitTheoremsPersistent2021}
Krebs, J. (2021).
\newblock On limit theorems for persistent {{Betti}} numbers from dependent
  data.
\newblock {\em Stochastic Processes and their Applications}, 139:139--174.

\bibitem[Marron et~al., 2015]{marronFunctionalDataAnalysis2015}
Marron, J.~S., Ramsay, J.~O., Sangalli, L.~M., and Srivastava, A. (2015).
\newblock Functional {{Data Analysis}} of {{Amplitude}} and {{Phase
  Variation}}.
\newblock {\em Statistical Science}, 30(4):468--484.

\bibitem[Panaretos and Zemel,
  2020]{panaretosInvitationStatisticsWasserstein2020}
Panaretos, V.~M. and Zemel, Y. (2020).
\newblock {\em An {{Invitation}} to {{Statistics}} in {{Wasserstein Space}}}.
\newblock {{SpringerBriefs}} in {{Probability}} and {{Mathematical
  Statistics}}. {Springer International Publishing}, {Cham}.

\bibitem[Perea, 2019]{pereaTopologicalTimeSeries2019}
Perea, J.~A. (2019).
\newblock Topological {{Time Series Analysis}}.
\newblock {\em Notices of the American Mathematical Society}, 66(05):1.

\bibitem[Perez, 2022]{perez_c0-persistent_2022}
Perez, D. (2022).
\newblock On {C0}-persistent homology and trees.
\newblock page~41.

\bibitem[Perng et~al., 2000]{perngLandmarksNewModel2000}
Perng, C.-S., Wang, H., Zhang, S., and Parker, D. (2000).
\newblock Landmarks: A new model for similarity-based pattern querying in time
  series databases.
\newblock In {\em Proceedings of 16th {{International Conference}} on {{Data
  Engineering}} ({{Cat}}. {{No}}.{{00CB37073}})}, pages 33--42, {San Diego, CA,
  USA}. {IEEE Comput. Soc}.

\bibitem[Plonka and Zheng, 2016]{plonka_relation_2016}
Plonka, G. and Zheng, Y. (2016).
\newblock Relation between total variation and persistence distance and its
  application in signal processing.
\newblock {\em Advances in Computational Mathematics}, 42(3):651--674.

\bibitem[Radulovi{\'c}, 1996]{radulovicBootstrapEmpiricalProcesses1996}
Radulovi{\'c}, D. (1996).
\newblock The bootstrap for empirical processes based on stationary
  observations.
\newblock {\em Stochastic Processes and their Applications}, 65(2):259--279.

\bibitem[Ramsay and Silverman, 2002]{ramsayAppliedFunctionalData2002}
Ramsay, J.~O. and Silverman, B.~W. (2002).
\newblock {\em Applied Functional Data Analysis: Methods and Case Studies}.
\newblock Springer Series in Statistics. {Springer}, {New York}.

\bibitem[Shevchenko, 2017]{mathoverflow_kolmogorov_279085}
Shevchenko, G. (2017).
\newblock Kolmogorov continuity theorem and holder norm.
\newblock MathOverflow.
\newblock https://mathoverflow.net/q/279085 (version: 2017-08-19).

\bibitem[Steinwart, 2022]{steinwartMathematicsGaussianProcesses2022}
Steinwart, I. (2022).
\newblock Mathematics of {{Gaussian Processes}} for {{Machine Learning}}.

\bibitem[Su et~al., 2014]{suStatisticalAnalysisTrajectories2014}
Su, J., Kurtek, S., Klassen, E., and Srivastava, A. (2014).
\newblock Statistical analysis of trajectories on {{Riemannian}} manifolds:
  {{Bird}} migration, hurricane tracking and video surveillance.
\newblock {\em The Annals of Applied Statistics}, 8(1):530--552.

\bibitem[Tang and Muller, 2008]{tangPairwiseCurveSynchronization2008}
Tang, R. and Muller, H.-G. (2008).
\newblock Pairwise curve synchronization for functional data.
\newblock {\em Biometrika}, 95(4):875--889.

\end{thebibliography}

\appendix
\addtocontents{toc}{\protect\setcounter{tocdepth}{0}}

\section{Measurability of functionals}
For our analysis of such signals, in section~\ref{sec:persistence_results}, we will introduce functionals of the form $\normalizedfunctional_t:C([0,T],\R)\rightarrow \R$, $t\in \T$, where the index set $\T$ is a (compact) metric space. Then, we will apply these functions pathwise and study the random variable $\normalizedfunctional(S)$, where $\normalizedfunctional$ is seen as a map $C([0,T],\R)$ to $\R^\T$.
Since $\normalizedfunctional_t$ is applied pathwise, it is not obvious under what conditions $\normalizedfunctional(S)$ is a random variable.
Such considerations could be circumvented by using outer probabilities~\cite{radulovicBootstrapEmpiricalProcesses1996, kosorok_introduction_2008}, but we address them in Proposition~\ref{prop:measurability}.

As a stochastic process, $S: (\SampleSpace, \sAlgebra)\rightarrow (C([0,T], \sigma(\R^{[0,T]})))$ is a random variable on the measured space $(\SampleSpace, \sAlgebra, \eta)$, where $\sigma(\R^{[0,T]})$ is the $\sigma$-algebra generated by the product topology on $\R^{[0,T]}$ and $\eta$ is the law of $S$. In our model, $\eta$ is determined by $\pattern$, $\mu$ and $\nu$.
\begin{proposition}
	\label{prop:measurability}
	Let ${\gamma:(\SampleSpace_r,\sAlgebra_r)\rightarrow (C([0,T],\R),\sigma(\Vert\cdot\Vert_\infty))}$ and ${W:(\SampleSpace_n,\sAlgebra_n)\rightarrow (C([0,T],\R), \sigma(\R^{[0,T]}))}$ be independent random variables and $S=\pattern\circ\gamma + W$ as in~\eqref{eq:model_continuous_signal}. If $f:C([0,T], \R)\rightarrow C(\T,\R)$ is continuous and $C(\T,\R)$ is $\Vert\cdot\Vert_\infty$-separable, then
	$f(S)$
	is $(C(\T,\R), \Vert\cdot \Vert_\infty)$-measurable.
\end{proposition}
For the proof, we will need the following lemma.
\begin{lemma}[Pettis' measurability theorem]
	\label{lemma:petts_measurability}
	Consider $h:\SampleSpace\rightarrow E$, where $(E, d_E)$ is a Banach space. If $E$ is separable as a metric space and $h$ is weakly-measurable, then $h$ is measurable with respect to the Borel $\sigma$-algebra induced by $d_E$.
\end{lemma}
\begin{proof}[Proof (proposition~\ref{prop:measurability})]
	First, assume that $S$ is weakly-measurable on $E = C([0,T],\R)$ and that $(C([0,T],\R), \Vert\cdot\Vert_\infty)$ is separable. Using lemma~\ref{lemma:petts_measurability}, we get that $S$ is $\sigma(\Vert\cdot\Vert_\infty)$-measurable.
	Because $f:C([0,T], \R)\rightarrow C(\T,\R)$ is continuous, it is measurable for the two $\sigma$-algebra on the domain and co-domain. This allows us to conclude that $f(S)$ is $(C(\T,\R),\sigma(\Vert\cdot\Vert_\infty))$-measurable.
	
	Let us now verify the assumptions of Lemma~\ref{lemma:petts_measurability}. By continuity of $\pattern$, the composition ${\pattern\circ\gamma: (\SampleSpace_r, \sAlgebra_r)\rightarrow (C([0,T],\R), \sigma(\R^{[0,T]})}$ is measurable. As a sum of two (independent) random variables, $S=\pattern\circ\gamma + W$ is $(C([0,T],\R), \sigma(\R^{[0,T]})$-measurable for $(\SampleSpace, \sAlgebra)$, where $\SampleSpace=\SampleSpace_r\times\SampleSpace_n$ and $\sAlgebra=\sAlgebra_r\otimes \sAlgebra_n$. The product $\sigma$-algebra $\sigma(\R^{[0,T]})$ coincides with that of weak measurability on $\R^{[0,T]}$.
	The space $C([0,T],\R)$ with the topology induced by $\Vert f\Vert_\infty \coloneqq \sup_{x\in[0,T]}\vert f(x)\vert$ is a Banach, separable space.
	Any subspace of a separable metric space is separable, so $S(\SampleSpace)$ is also separable.
\end{proof}
\begin{remark}
	Lemma~\ref{lemma:petts_measurability} and its application to prove the measurability of the process were taken from the course~\cite{steinwartMathematicsGaussianProcesses2022}.
\end{remark}

\section{Invariance of the signature to reparametrisation}
\label{appendix:disintegration}
Consider $(C([0,T],\R),\Vert\cdot\Vert_\infty)$ with the Borel $\sigma$-algebra.
We assume that $\mu_1,\mu_2$ are Borel measures on the restriction of that $\sigma$-algebra to a closed subspace $\Gamma\subset C([0,T])$.
We denote by ${\delta_t: \gamma\mapsto\gamma(t)}$ the evaluation map, we let $\pushforward{\delta_t}\mu_1=\mu_1\circ(\delta_t)^{-1}$ be the measure which characterizes the marginal distribution of $\gamma_1(t)$ and we proceed similarly for $\mu_2$. Note that the evaluation is measurable, as it corresponds to weak-measurability.
Similarly, we denote $\delta_{0,T}: \gamma\mapsto(\gamma(0),\gamma(T))\in\R^2$.
\begin{proposition}
	\label{prop:invariance}
	If the marginals $\pushforward{\delta_{0,T}}\mu_1$ and $\pushforward{\delta_{0,T}}\mu_2$ are equal, then
	$$F(\pattern\circ\gamma_1)=F(\pattern\circ\gamma_2).$$
\end{proposition}
\begin{proof}
	We first need to show that we can condition on $(\gamma(0),\gamma(T))$. The space of continuous functions $C([0,T],\R)$ is Polish, and so is $\Gamma$, because it is a closed subspace. Let $\mathcal{A} = \sigma(\delta_{0,T})$ be the $\sigma$-algebra generated by the evaluations. By~\cite[Corollary 10.4.6]{bogachevMeasureTheory2007}, there is a regular conditional measure $((\mu_1)_{x}(d\gamma))_{x\in \R^2}$.
	
	Lemma~\ref{lemma:invariance_to_reparametrisation} implies that $\gamma\mapsto \normalizedfunctional_t(\pattern\circ\gamma)$ is constant on $\delta_{0,T}^{-1}(x)$, for any $x=(s,r)\in\R^2$. For any $t\in\T$, using the regular conditional measure property~\cite[Definition 10.4.1]{bogachevMeasureTheory2007},
	\begin{align*}
	F_t(\pattern\circ\gamma_1)
	&= \int_{\Gamma}\normalizedfunctional_t(\pattern\circ\gamma)\mu_1(d\gamma)\\
	&= \int_{\R^2}\int_{\delta_{0,T}^{-1}(x)}\normalizedfunctional_t(\pattern\circ\gamma) (\mu_1)_{x}(d\gamma)\pushforward{\delta_{0,T}}\mu_1(dx)\\
	&= \int_{\R^2}\int_{\delta_{0,T}^{-1}(x)}\normalizedfunctional_t(\pattern\circ\gamma) (\mu_2)_{x}(d\gamma)\pushforward{\delta_{0,T}}\mu_2(dx)\\
	&= F_t(\pattern\circ\gamma_2).
	\end{align*}
\end{proof}
Since $\gamma_1,\,\gamma_2$ are reparametrisations, we require $\Gamma$ to be included in the space of injective functions. An example is given in~\eqref{eq:gamma_closed_example}.

While it is disappointing to require equality of the marginals $\pushforward{\delta_{0,T}}\mu_1$ and $\pushforward{\delta_{0,T}}\mu_2$ in Proposition~\ref{prop:invariance}, removing this assumption poses a difficulty which we now discuss. Consider $\gamma_1$ and $\gamma_2$ fixed, assume that $R \coloneqq R_1 < R_2$ and let $T_1 = \gamma_2^{-1}(R)$. For continuous functions on an interval, we can only control the stability of the persistence diagram in the bottleneck distance $d_B$ (see Theorem~\ref{thm:bottleneck_stability}).

As an example, consider the case when $R_2 = R + 1$. Then, the distance between the persistence diagrams $d_B(D(\pattern\circ\gamma_1), D(\pattern\circ\gamma_2))$ is of the order of the amplitude $A_\pattern\coloneqq \max\pattern-\min\pattern$, as the multiplicity of the point $(\min\pattern, \max\pattern)$ differs by at least one between both diagrams. The term $\tfrac{\pers_{p, \epsilon}^p(\pattern\circ\gamma_1) + \pers_{p, \epsilon}^p(\pattern\circ\gamma_2)}{\pers_{p, \epsilon}^p(\pattern\circ\gamma_1)}$ is roughly constant ($1\leq\tfrac{R_1 + R_2}{R_1}\leq 3$). Therefore, the fact that the difference between $\normalizedfunctional(\pattern\circ\gamma_1)$ and $\normalizedfunctional(\pattern\circ\gamma_2)$ will be small is not reflected by Proposition~\ref{prop:functional_continuity} which gives a trivial bound.
Instead, let $D_1 = D(\restr{(\phi\circ\gamma_2)}{[0,T_1]})$, $D_2 = D(\restr{(\phi\circ\gamma_2)}{[T_1,T]})$ and consider
\begin{equation}
\label{eq:diagram_cutting_decomposition}
\Vert \normalizedfunctional(\phi\circ\gamma_1) -\normalizedfunctional(\phi\circ\gamma_2) \Vert_\infty
\leq \Vert \normalizedfunctional(D_1) -\normalizedfunctional(D_1\sqcup D_2)\Vert_\infty + 
\Vert \normalizedfunctional(D_1\sqcup D_2) - \normalizedfunctional(\phi\circ\gamma_2)\Vert_\infty.
\end{equation}
Conveniently, a normalized functional of a union of diagrams is a weighted average of the normalized functionals of the individual diagrams
\begin{equation*}
\normalizedfunctional(D_1\sqcup D_2)(t) = \normalizedfunctional(D_1)(t)\frac{\pers_{p,\epsilon}^p(D_1)}{\pers_{p,\epsilon}^p(D_1\sqcup D_2)} + \normalizedfunctional(D_2)(t)\frac{\pers_{p,\epsilon}^p(D_2)}{\pers_{p,\epsilon}^p(D_1\sqcup D_2)},
\end{equation*}
so that
\begin{align*}
\vert \normalizedfunctional(D_1)(t) -\normalizedfunctional(D_1\sqcup D_2)(t)\vert
&= \vert \normalizedfunctional(D_1)(t)\left(\frac{\pers_{p,\epsilon}^p(D_1)}{\pers_{p,\epsilon}^p(D_1\sqcup D_2)} -1\right) + \normalizedfunctional(D_2)(t)\frac{\pers_{p,\epsilon}^p(D_2)}{\pers_{p,\epsilon}^p(D_1\sqcup D_2)}\vert \\
&=\vert \normalizedfunctional(D_1)(t)-\normalizedfunctional(D_2)(t)\vert \frac{\pers_{p,\epsilon}^p(D_2)}{\pers_{p,\epsilon}^p(D_1\sqcup D_2)}\\
&\leq (L_\kernel A_\pattern +C)\frac{\pers_{p,\epsilon}^p(D_2)}{\pers_{p,\epsilon}^p(D_1\sqcup D_2)}.
\end{align*}
We claim that if $\pattern$ is regular enough and $R_2-R$ is small, then so is $\pers_{p, \epsilon}^p(D_2)$. For example, with~\cite[Corollary 4.6]{herbertedelsbrunnerWindowPersistence1D2023}, if $\restr{\pattern}{[R,R_2]}$ is continuous and has finitely many critical points, then the total variation of $\pattern_{[R_2,R]}$ is equal to $\pers_{1, 0}(D_2)$.

Thanks to Proposition~\ref{prop:functional_continuity}, the second term in~\eqref{eq:diagram_cutting_decomposition} is the error made when approximating the diagram of $\restr{\pattern}{[0,R_2]}$ by the union of diagrams of $\restr{\pattern}{[0,R]}$ and $\restr{\pattern}{[R,R_2]}$. For a particularly good cutting point $R$, that is, when $R$ is a global maximum, $D_1\sqcup D_2=D(\phi\circ\gamma_2)$. However, in general, the support of the union of diagrams differs from the diagram of the whole interval. To show stability, we miss the study of $d_B(D_1\sqcup D_2, D(\phi\circ\gamma_2))$. A possible avenue is given by the tools introduced in~\cite{herbertedelsbrunnerWindowPersistence1D2023}.

\section{Proof of Theorem~\ref{thm:signature_stability_to_reparametrisation}}
\label{app:proof_stability}
We start by treating $S$ path--wise. Using Proposition~\ref{prop:functional_continuity} and the bottleneck stability of persistence diagrams,
\begin{align}
\Vert \normalizedfunctional(\pattern\circ\gamma_1 + W) - \normalizedfunctional(\pattern\circ\gamma_2 + W)\Vert
&= \Vert \normalizedfunctional(\pattern + {W}_{\gamma_1^{-1}}) - \normalizedfunctional(\pattern+ {W}_{\gamma_1^{-1}})\Vert\nonumber\\
&\leq L_\kernel\left(1 + 4pU\max_{k=1,2}\frac{\pers_{p-1,\epsilon}^{p-1}(\pattern + {W}_{\gamma_k^{-1}} )}{\pers_{p,\epsilon}^p(\pattern+{W}_{\gamma_k^{-1}})}\right) \Vert {W}_{\gamma_1^{-1}} - {W}_{\gamma_2^{-1}}\Vert_\infty,
\label{eq:functional_difference_same_Rs}
\end{align}
where $L_\kernel$ is a regularity constant of the kernel and $U$ is an upper-bound on the persistence of any point in both diagrams. The persistence of any point in the diagram $D(h)$ of a function $h$ is bounded by $A_h$. Hence, the persistence of a point in $D(\pattern + W)$ is bounded by $U=A_{\pattern + W}\leq A_\pattern + A_W\leq A_\pattern + (A_\pattern - \epsilon - q)\leq 2A_\pattern$.

Next, we obtain an upper--bound of $\max_{k=1,2}\frac{\pers_{p-1,\epsilon}^{p-1}(\pattern + {W}_{\gamma_k^{-1}} )}{\pers_{p,\epsilon}^p(\pattern+{W}_{\gamma_k^{-1}})}$. By Proposition~\ref{proposition:kolmogorov_implies_holder}, we can assume that $W$ has $\alpha$-H\"older paths with a (random) constant $\Lambda_W$, for $\alpha\coloneqq
\tfrac{\min(1,\KolmogR-1)}{\KolmogP}$.
This implies that $\tfrac{1}{\alpha} + 1 < p$ and we use the continuity of truncated persistence from Proposition~\ref{prop:continuity_truncated_persistence} to obtain
\begin{equation}
\label{eq_in_proof:upper_bound_persistence}
\pers_{p-1,\epsilon}^{p-1}(\pattern + {W}_{\gamma_k^{-1}} ) \leq \pers_{p-1,\epsilon}^{p-1}(\restr{\pattern}{[0,T]}) + (p-1)\Vert W\Vert_\infty(\pers_{p-2,\epsilon}^{p-2}(\restr{\pattern}{[0,T]}) + \pers_{p-2,\epsilon}^{p-2}(W_{\gamma_k^{-1}})).
\end{equation}
For any $x\in[0,1]$ and $p\geq 0$, the function $p\mapsto x^p$ is decreasing, so that
\begin{align*}
\pers_{p-1,\epsilon}^{p-1}(\restr{\pattern}{[0,T]})
&= (A_\pattern - \epsilon)^{p-1}\sum_{(b,d)\in D} \max\left(\tfrac{d-b-\epsilon}{A_\pattern-\epsilon},0\right)^{p-1}\\
&\leq (A_\pattern - \epsilon)^{p-1}\sum_{(b,d)\in D} \max\left(\tfrac{d-b-\epsilon}{A_\pattern-\epsilon},0\right)^{p-2}\\
&= (A_\pattern - \epsilon)\pers_{p-2,\epsilon}^{p-2}(\pattern).
\end{align*}
Since $\Vert W\Vert_\infty< (A_\pattern - \epsilon)/2$ and the persistence does not depend on the parametrisation $\pers_{p-2,\epsilon}^{p-2}(W_{\gamma_k^{-1}}) = \pers_{p-2,\epsilon}^{p-2}(W)$, equation~\eqref{eq_in_proof:upper_bound_persistence} becomes
\begin{align*}
\pers_{p-1,\epsilon}^{p-1}(\pattern + {W}_{\gamma_k^{-1}} )
&\leq
(A_\pattern - \epsilon)\pers_{p-2,\epsilon}^{p-2}(\pattern)\left(1+ \tfrac{p-1}{2}\left(1+\tfrac{\pers_{p-2,\epsilon}^{p-2}(W)}{\pers_{p-2,\epsilon}^{p-2}(\pattern)}\right)\right)\\
&\leq
p(A_\pattern - \epsilon)\pers_{p-2,\epsilon}^{p-2}(\pattern)\left(1+ \tfrac{1}{2}\tfrac{\pers_{p-2,\epsilon}^{p-2}(W)}{\pers_{p-2,\epsilon}^{p-2}(\pattern)}\right).
\end{align*}
An upper--bound for the persistence of $W$ is given in Proposition~\ref{prop:upper_bound_p_persistence}
\begin{equation*}
\pers_{p,\epsilon}^p(W)\leq (A_W-\epsilon)^p\left(1+pT\left(\tfrac{2\Lambda_W}{\epsilon}\right)^{1/\alpha}\right),
\end{equation*}
where $\Lambda_W$ is the path--wise H\"older constant of $W$.
The amplitude $A_\pattern$ upper--bounds the persistence of a point and it is also realized as the persistence of a pair of a global minimum and a global maximum, so $\pers_{p-2,\epsilon}^{p-2}(\restr{\pattern}{[0,R]}) \geq (R-2)(A_\pattern-\epsilon)^{p-2}$ and hence
\begin{equation*}
\frac{\pers_{p,\epsilon}^p(W)}{\pers_{p-2,\epsilon}^{p-2}(\pattern)}
\leq \left(\tfrac{A_W-\epsilon}{A_\pattern -\epsilon}\right)^{p-2} (A_W-\epsilon)^2\tfrac{T}{R-2}\left(1+p\left(\tfrac{2\Lambda_W}{\epsilon}\right)^{1/\alpha}\right).
\end{equation*}
Putting the above together, with $p\geq2$,
\begin{align*}
\pers_{p-1,\epsilon}^{p-1}(\pattern + {W}_{\gamma_k^{-1}} )
&\leq
p(A_\pattern - \epsilon)\pers_{p-2,\epsilon}^{p-2}(\pattern)\left(1+ \left(\tfrac{A_W-\epsilon}{A_\pattern -\epsilon}\right)^{p-2} (A_W-\epsilon)^2\tfrac{T}{R-2}\max\left(1,p\left(\tfrac{2\Lambda_W}{\epsilon}\right)^{1/\alpha}\right)\right).
\end{align*}
We have therefore an upper--bound for the numerator. To lower--bound the denominator, we use Proposition~\ref{prop:lower_bound_p_persistence}:
\begin{align*}
\pers_{p,\epsilon}^{p}(\pattern + {W}_{\gamma_k^{-1}})
&\geq \pers_{p,\epsilon + A_W}^{p}(\pattern)\\
&\geq (R-2)(A_\pattern - (A_W+\epsilon))^p \\
&\geq (R-2)(A_\pattern - (A_\pattern - \epsilon + q + \epsilon))^p = (R-2)q^p.
\end{align*}

We conclude that we have an upper--bound $C_{\Lambda_W}$ on $\frac{\pers_{p-1,\epsilon}^{p-1}(\pattern + {W}_{\gamma_k^{-1}} )}{\pers_{p,\epsilon}^p(\pattern+{W}_{\gamma_k^{-1}})}$, that is
\begin{equation*}
C_{\Lambda_W} \coloneqq
L_\kernel	\left(1 + 8p^2A_\pattern
\frac{(A_\pattern - \epsilon)\pers_{p-2,\epsilon}^{p-2}(\pattern)\left(1+ \left(\tfrac{A_W-\epsilon}{A_\pattern -\epsilon}\right)^{p-2} (A_W-\epsilon)^2\tfrac{T}{R-2}\max\left(1,p\left(\tfrac{2\Lambda_W}{\epsilon}\right)^{1/\alpha}\right)\right)}{(R-2)q^p}
\right).
\end{equation*}
As $A_W\leq A_\pattern - \epsilon-q$, the only remaining stochastic term in $C_{\Lambda_W}$ is $\Lambda_W^{1/\alpha}$. Also, the bound only depends on $R$ (which is fixed), but not on $\gamma$ itself.

Let $\pi:\sAlgebra_{r,1}\times\sAlgebra_{r,2}\rightarrow \R$ be a coupling of $\mu_1$ and $\mu_2$. Specifically, $\pi$ is a measure on the product space $(\gammaSpace\times\gammaSpace,\sAlgebra_{r,1}\otimes \sAlgebra_{r,2}),$ such that
$\pi(A,\gammaSpace) = \mu_1(A)$ and $\pi(\gammaSpace, A)=\mu_2(A)$, for all $A\in\sAlgebra$. Then, $\pi\otimes\nu:((A_1,B_1),(A_2,B_2))\mapsto \pi(A_1,A_2)\nu(B_1\cap B_2)$ is a coupling of $\mu_1\otimes\nu$ and $\mu_2\otimes\nu$.
Using the coupling and~\eqref{eq:functional_difference_same_Rs},
\begin{align*}
\vert \E[\normalizedfunctional(\pattern\circ\gamma_1 + W)\mid W] - \E[\normalizedfunctional(\pattern\circ\gamma_2+W)\mid W]\vert
&= \left\vert\E_{(\gamma_1,\gamma_2)\sim\pi}[\normalizedfunctional(\pattern\circ\gamma_1 + W) - \normalizedfunctional(\pattern\circ\gamma_2 + W)\mid W]\right\vert\\
&\leq \E_{(\gamma_1,\gamma_2)\sim\pi} \left[\vert\normalizedfunctional(\pattern\circ\gamma_1 + W) - \normalizedfunctional(\pattern\circ\gamma_2 + W)\vert \mid W\right]\\
& \leq C_{\Lambda_W}
\E[\Vert {W}_{\gamma_1^{-1}} - {W}_{\gamma_2^{-1}}\Vert_\infty \mid W],\\
&\leq C_{\Lambda_W}\Lambda_W\E[\Vert\gamma_1^{-1} - \gamma_2^{-1}\Vert_\infty^\alpha].
\end{align*}
We have thus completely separated the bound into a product, with terms depending on $\nu$ and $(\mu_1,\mu_2)$.

On one hand, it remains to take the expectation with respect to $W$. We bound the moments of $\Lambda_W$ using Theorem~\ref{thm:holder_constant_moments}, obtaining
\begin{align*}
\E[\Lambda_W]&\leq 16\tfrac{\alpha+1}{\alpha}(K_{\KolmogP,\KolmogR})^{1/\KolmogP}\\
\E[\Lambda_W^{1+1/\alpha}] &\leq 6^{\KolmogP+2} K^{(1/\KolmogP + 1/(\KolmogR-1))}_{\KolmogP,\KolmogR}.
\end{align*}

On the other hand, by Jensens' inequality,
${\E[\Vert\gamma_1^{-1} - \gamma_2^{-1}\Vert_\infty^\alpha]\leq \E[\Vert\gamma_1^{-1} - \gamma_2^{-1}\Vert_\infty]^\alpha}$. Using the lower--bound on the modulus of continuity,
\begin{equation*}
\sup_{r\in[0,R]}\vert\gamma_1^{-1}(r) - \gamma_2^{-1}(r)\vert
= \sup_{t\in[0,T]}\vert t - \gamma_2^{-1}(\gamma_1(t))\vert
\leq \sup_{t\in[0,T]}\frac{1}{v_{\min}}\vert\gamma_2(t)-\gamma_1(t)\vert.
\end{equation*}
Taking the infimum over couplings, we obtain the $1-$Wasserstein distance $W_1(\mu_1,\mu_2)$.

\section{Proof of Proposition~\ref{prop:functional_difference_via_bias}}
\label{app:proof_functional_difference_via_bias}
We start by proving a lemma.
\begin{lemma}[Perturbed, path--wise version]
	\label{lemma:perturbed_pathwise_version}
	Consider a continuous perturbation $W\in C^{\alpha}_\Lambda([0,T],\R)$ and set $\delta\coloneqq\Vert W\Vert_\infty$. If $2\delta\leq \max \pattern - \min\pattern$, then
	\begin{equation*}
	\Vert \normalizedfunctional(\pattern + W) - \normalizedfunctional(\pattern)\Vert_\infty \leq
	L_k(P_1 \delta + P_2\delta^2 + P_3\delta^3)\eqqcolon L_k P(\delta),
	\end{equation*}
	where
	\begin{align*}
	P_1 =& 1+ 4A_\pattern C_T C_{p-1,p}^\epsilon(\pattern),\\
	P_2 =& 8C_TC_{p-1,p}^\epsilon(\pattern) + 4pA_\pattern (C_T C_{p-2, p}^\epsilon(\pattern) + \tfrac{C_{p-3,\Lambda,\alpha, T}}{\pers_{p,\epsilon}^p(\pattern)}),\\
	P_3 =& 4p\left(C_T C_{p-2, p}^\epsilon(\pattern) + \tfrac{C_{p-3,\Lambda,\alpha,T}}{\pers_{p,\epsilon}^p(\pattern)}\right),
	\end{align*}
	and
	\begin{equation*}
	C_T = \frac{\lceil T\rceil}{\lfloor T\rfloor -2},\qquad
	C_{p, p'}^\epsilon(\pattern) = \frac{\pers_{p,\epsilon}^p(\pattern)}{\pers_{p',\epsilon}^{p'}(\pattern)},\qquad
	A_\pattern = \Vert \pattern\Vert_\infty.
	\end{equation*}
\end{lemma}
\begin{proof}
	By the diagram stability theorem, $d_B(D(\pattern + W), D(\pattern))\leq \Vert W\Vert_\infty\leq \delta$. The persistence of a point in $D(\pattern)$ and $D(\pattern+W)$ is bounded by $2A_\pattern$ and $2A_{\pattern + W}\leq 2(A_\pattern + \delta)$ respectively. Using Proposition~\ref{prop:continuity_truncated_persistence}, we also bound $\pers_{p-1,\epsilon}^{p}(\pattern + W)\leq \pers_{p-1,\epsilon}^{p-1}(\pattern) + p\delta(\pers_{p-2,\epsilon}^{p-2}(\pattern) + \pers_{p-2,\epsilon}^{p-2}(W))$. Using the uniform bound on persistence from Proposition~\ref{prop:upper_bound_p_persistence}, $\pers_{p-2,\epsilon}^{p-2}(W)\leq C_{p-3,\Lambda, \alpha, T}$.
	Finally, putting these together with Proposition~\ref{prop:functional_continuity}, we obtain:
	\begin{align*}
	\Vert \normalizedfunctional(\pattern) - \normalizedfunctional(\pattern + W)\Vert
	\leq& L_\kernel \left(1 + 2pU\frac{\pers_{p-1,\epsilon}^{p-1}(\pattern) + \pers_{p-1,\epsilon}^{p-1}(\pattern+W)}{\pers_{p,\epsilon}^p(\pattern)}\right) d_B(D(\pattern), D(\pattern+W))\\
	\leq& \delta L_\kernel \left(1+4p(\Vert\pattern\Vert_\infty + \delta)
	\tfrac{2\lceil T\rceil\pers_{p-1,\epsilon}^{p-1}(\restr{\pattern}{[c, c+1]}) + p\delta(\pers_{p-2,\epsilon}^{p-2}(\pattern) + C_{p-3,\Lambda, \alpha, T})}{(\lfloor T\rfloor-2)\pers_{p,\epsilon}^p(\pattern)}\right)\\
	\leq& L_\kernel\left(1+ 4A_\pattern C_T C_{p-1,p}^\epsilon(\pattern)\right) \delta\ +\\
	&L_\kernel \left(8C_TC_{p-1,p}^\epsilon(\pattern) + 4pA_\pattern (C_T C_{p-2, p}^\epsilon(\pattern) + \tfrac{C_{p-3,\Lambda,\alpha, T}}{\pers_{p,\epsilon}^p(\pattern)})\right)\delta^2\ +\\
	&4L_\kernel p\left(C_T C_{p-2, p}^\epsilon(\pattern) + \tfrac{C_{p-3,\Lambda,\alpha,T}}{\pers_{p,\epsilon}^p(\pattern)}\right)\delta^3.
	\end{align*}
\end{proof}

\begin{proof}[Proof of Proposition~\ref{prop:functional_difference_via_bias}]
	Combining lemma~\ref{lemma:perturbed_pathwise_version} and theorem~\ref{thm:convergence_to_limit},
	\begin{equation*}
	\begin{array}{r c r l}
	\Vert \normalizedfunctional(\pattern\circ\gamma_1 + W_1) - \normalizedfunctional(\pattern\circ\gamma_2 + W_2)\Vert
	&\leq&& \Vert \normalizedfunctional(\pattern + (W_1)_{\gamma_1^{-1}}) - \normalizedfunctional(\restr{\pattern}{[0,R_1]})\Vert \\
	&&+& \Vert \normalizedfunctional(\restr{\pattern}{[0,R_1]}) - \normalizedfunctional(\restr{\pattern}{[0,R_2]})\Vert \\
	&&+& \Vert \normalizedfunctional(\restr{\pattern}{[0,R_2]}) - \normalizedfunctional(\pattern + (W_2)_{\gamma_2^{-2}})\Vert \\
	&\leq&& L_\kernel(P(\delta_1) + P(\delta_2) + 2\tfrac{4}{\min(R_1,R_2)}\normalizedfunctional(\restr{\pattern}{[c,c+1]})) \\
	&\leq&& L_\kernel\left(P(\delta_1) + P(\delta_2) + \tfrac{8}{\min(R_1,R_2)-2}\tfrac{A_\pattern}{2}\right)\\
	&\leq&& L_\kernel \left( P(\max(\delta_1,\delta_2))+ \tfrac{4A_\pattern}{\min(R_1,R_2)-2}\right).
	\end{array}
	\end{equation*}
\end{proof}

\section{Exponential mixing of the reparametrisation process}
\label{appendix:mixing_proof}
\begin{proposition}
	\label{prop:gamma_beta_mixing}
	Consider $(\gamma_n)_{n=1}^N$ as in~\eqref{eq:gamma_Markov_chain} with $(V_n)_{n=1}^N$ as in Model 1 or 2. Then, $\beta_{\Frac(\gamma)}(k)\rightarrow 0$ exponentially fast.
\end{proposition}
The proof of this proposition relies on the continuity of the transition kernel with respect to the Lebesgue measure and the use of the following sufficient condition.
\begin{theorem}[{\cite[Section 2.4, Theorem 1]{doukhanMixing1995}}]
	\label{theorem:mixing}
	Let $(Z_n)_n$ be a stationary Markov chain and $\nu$ a non-negative and non-zero measure. If there exists $r\in\N^*$ such that
	\begin{equation}
	\label{eq:mixing_condition}
	P(Z_r\in A\mid Z_0=z)\geq \nu(A),\qquad\text{ for any } z,\text{ and } A \text{ any P-measurable set,}
	\end{equation}
	then $(Z_n)_n$ is $\beta$-mixing and the coefficients decay exponentially fast.
\end{theorem}
Condition~\eqref{eq:mixing_condition} is called a Doeblin condition. It consists in providing a non-trivial lower--bound on the family of measures $(P^r(z,A))_{z}$. 
We first treat the case where $(V_n)_{n\in\N}$ are all \iid. The case where $(V_n)_{n\in\N}$ is a Markov Chain is similar, but technically more difficult.

\subsection{Model 1}
\label{sec:mixing_proof_iid}
Recall that $\gamma_n = \gamma_{n-1} + V_{n-1}$. In Model 1, $V_n$ is independent from $(V_k)_{k<n}$ and $\gamma_0$, so $(\gamma_n)_{n\in\N}$ is a Markov chain. We will now verify~\eqref{eq:mixing_condition}. Let $r\coloneqq \lceil 2/ (b-a)\rceil$ and $\epsilon=\lfloor \tfrac{b-a}{r}\rfloor$.
	\begin{lemma}
		\label{lemma:convolution_two_measures}
		Consider two measures $\mu_1,\,\mu_2$ such that $\mu_k(A)\geq c_k\Lebesgue(A),$ for $ A\in \Borel([a_k,b_k])$. Then, for any $0<\epsilon<\min(b_1-a_1,b_2-a_2)$, we have that
		$(\mu_1\star\mu_2)(A) \geq c_1c_2\epsilon\Lebesgue(A),$ for any $A\in\Borel([a_1+a_2 +\epsilon, b_1+b_2-\epsilon]).$
	\end{lemma}
We now apply this Lemma~\ref{lemma:convolution_two_measures} inductively to $\mu_1$ and $\mu_2$ the measures of $\sum_{n=1}^{r_1} V_{r_1}$ and $V_{r_1+1}$ respectively, for $1\leq r_1\leq r-1$. We conclude that $P(\sum_{n=1}^r V_n\in A)\geq c\Lebesgue(A)$ for all $A\in\Borel(B)$, where $B\coloneqq[r(a+\epsilon) - \epsilon, r(b-\epsilon) + \epsilon]$ and $c\coloneqq c_1c_2\epsilon^{r-1}$. Thanks to our choice of $r$ and $\epsilon$, $B$ is an interval of length at least 1.

Let $x_0\in\lbrack0,1\lbrack$ and $A\in\mathcal{B}([0,1])$. We write $\Frac^{-1}(A) = \cup_{k\in\Z} A + k$, where $A+k=\{a+k\mid a\in A\}$. Then,
\begin{align*}
	P(\Frac(\gamma_r)\in A\mid \gamma_0=x_0)
	&= P\left(x_0 + \sum_{n=0}^r V_n\in \Frac^{-1}(A)\right)\\
	&= P\left(\sum_{n=0}^r V_n \in \bigcup_{k\in\Z} (A+k)-x_0\right)\\
	&\geq P\left(\sum_{n=0}^r V_n \in \bigcup_{k\in\Z} (A+k -x_0) \cap B\right) \\
	&\geq c\Lebesgue\left(\bigcup_{k\in\Z} (A+k-x_0) \cap B\right)\\
	&= c\sum_k\Lebesgue(A+k-x_0\cap B),
\end{align*}
where the last equality follows from the fact that $\Lebesgue(A\cap (A+1))=0$, because $A\subset [0,1]$.
Notice that for any set $A + z \cap B = (A\cap (B-z)) + z$ and that $\Lebesgue(A+z) = \Lebesgue(A)$, for any $z\in\R$. Hence, for any $k\in\Z$,
\begin{equation*}
\Lebesgue\left(A+k-x_0 \cap B\right) =  \Lebesgue(k-x_0 + (A\cap (B-k+x_0))) = \Lebesgue(A\cap (B - k - x_0)).
\end{equation*}
Recall that $B$ is an interval of length greater than 1, so $(B-k-x_0)_{k\in\Z}$ is a cover of $\R$. Hence,
\begin{align*}
P(\Frac(\gamma_r)\in A\mid \gamma_0=x_0)
&\geq c\sum_k\Lebesgue(A \cap (B - k - x_0))\\
&\geq c \Lebesgue\left(A\cap \bigcup_k (B - k - x_0)\right)\\
&= c\Lebesgue(A).
\end{align*}
We can therefore set $\mu \coloneqq c\Lebesgue$. The measure does not depend on $x_0$ and it has total mass $c>0$.

We now show that $(\Frac(\gamma_n))_{n\in\N}$ is strictly stationary: for any $K\in\N^*$, $\tau\in\N$ and $n_1,\ldots n_K$, the vectors
$(\Frac(\gamma_{n_1}),\ldots, \Frac(\gamma_{n_K}))\sim(\Frac(\gamma_{n_1+\tau}),\ldots, \Frac(\gamma_{n_K+\tau}))$, where $X\sim Y$ is a shorthand notation for ``$X$ and $Y$ have the same distribution". It is enough to show that for any $K\geq 1$, $(\Frac(\gamma_0),\ldots, \Frac(\gamma_K))\sim(\Frac(\gamma_n),\ldots \Frac(\gamma_{n+K}))$, for any $n\geq 0$.
We write $(\Frac(\gamma_n),\ldots \Frac(\gamma_{n+K})) = \Frac(\Frac(\gamma_0 + \sum_{r=0}^{n-1} V_r) + \Frac(0,V_{n}, \ldots,\sum_{r=n}^{n+K-1} V_r))$ and we analyze the two terms separately. Here, $\Frac$ is applied component--wise.
First, because $(V_n)_{n\in\N}$ are \iid, $\left(\sum_{r=0}^{k} V_r\right) \sim \left(\sum_{r=n}^{n+k} V_r\right)$, for any $n,k \in\N$. Therefore, $(0, V_0,\ldots, \sum_{r=0}^{n-1} V_r)\sim(0,V_{n}, \ldots,\sum_{r=n}^{n+K-1} V_r)$. It also remains true when we apply $\Frac$ component--wise, because it is a measurable mapping $\R^{K+1}\rightarrow\R^{K+1}$.
Second, we claim the following lemma on the sum of two random variables, one of which is uniform.
\begin{lemma}
	\label{lemma:convolution}
	If $U\sim\Uniform([0,1])$ and $Z$ is a real--valued random variable independent of $U$, then ${\Frac(U+Z)\sim\Frac(U)\sim U}$.
\end{lemma}
Before showing Lemma~\ref{lemma:convolution}, we conclude the proof by applying it to $U=\gamma_0$ and $Z = \sum_{r=0}^{n-1} V_r$. Indeed, $\gamma_0$ is independent from $(V_r)_{r=0}^{n-1}$, so we obtain that
$\Frac(\gamma_0)\sim\Frac(\gamma_0 + \sum_{r=0}^{n-1} V_r).$
Finally, combining the above with $\Frac((0, V_0,\ldots, \sum_{r=0}^{n-1} V_r))\sim\Frac((0,V_{n}, \ldots,\sum_{r=n}^{n+K-1} V_r))$, we have that $\Frac(\gamma_0,\ldots, \gamma_K)\sim \Frac(\gamma_{n}, \ldots, \gamma_{n+K})$.
\begin{proof}[Proof of Lemma~\ref{lemma:convolution}]
	First, it is clear that for $s\leq 0$, $P(\Frac(U+Z)<s)=0$ and that for $s>1$, $1\geq P(\Frac(U+Z)<s)\geq P(\Frac(U+Z)\leq 1) = 1$. 
	For $0<s<1$,
	\begin{equation}
	\label{eq:convolution}
	P(\Frac(U+Z)\leq s) = P\left(U+Z \in \bigcup_{k\in\Z}[k, k+s]\right) = \sum_{k\in\Z} P(U+Z\in [k, k+s]).
	\end{equation}
	Because $U$ and $Z$ are independent, $P(U+Z\in[k,k+s])=(\mu_U\star\mu_Z)([k,k+s])$, where $\mu_U$ and $\mu_Z$ are the probability measures of $U$ and $Z$ respectively and $\star$ denotes their convolution. Note that since $\Lebesgue$ is translation--invariant,
	\begin{align*}
	(\mu_U\star\mu_Z)([k,k+s])
	&= \int_\R \int_0^1 \indicator{[k,k+s]}(z+u)du d\mu_Z(z)\\
	&= \int_\R \Lebesgue([0,1]\cap[k-z,k+s-z]) d\mu_Z(z) \\
	&=\int_\R \Lebesgue([-k, -k+1]\cap[-z, -z+s]) d\mu_Z(z)\\
	&=\int_\R \Lebesgue(\lbrack-k, -k+1\lbrack\cap[-z, -z+s]) d\mu_Z(z)
	\end{align*}
	Going back to~\eqref{eq:convolution},
	\begin{align*}
	P(\Frac(U+Z)\leq s)
	&= \sum_{k\in\Z} \int_\R \Lebesgue(\lbrack-k, -k+1\lbrack\cap[-z, -z+s]) d\mu_Z(z)&\\
	&= \int_\R \sum_{k\in\Z} \Lebesgue(\lbrack-k, -k+1\lbrack\cap[-z, -z+s]) d\mu_Z(z)&\\
	&= \int_\R \Lebesgue([-z,-z+s])d\mu_Z(z)&\\
	&= \Lebesgue([0,s])\int_\R d\mu_Z(z).&\\
	&=s.
	\end{align*}
	Therefore, the distribution function of $\Frac(U+Z)$ is uniform on $[0,1]$ and therefore also equal to that of $\Frac(U)$.
\end{proof}

\subsection{Model 2}
The process $(\Frac \gamma_n)_{n\in\N}$ is defined in~\eqref{eq:gamma_Markov_chain}, via the Markov chain $(V_n)_{n\in\N}$. Recall that this Markov chain has a transition probability kernel $P$, with support included in $I=[\vmin, \vmax]$. Therefore, $(\Frac(\gamma_n))_{n\in\N}$ is not itself a Markov Chain (of order 1). However, the process $((\gamma_n, V_n))_{n\in\N}$ is a Markov Chain. We characterize its distribution and we verify that it satisfies the Doeblin condition~\eqref{eq:mixing_condition}, which takes the remaining of this Section.

Consider now $(\R,\Borel(\R))$ and let $(x,A)\mapsto \indicator{A}(x)$, which is also a transition probability kernel. We define a product kernel on $R\coloneqq\R\times I$, where $I=[\vmin,\vmax]$. It is characterised by the following measure on rectangles
\begin{equation*}
((y,v), (A\times B))\mapsto \indicator{A}(y) \TPK(v, B).
\end{equation*}
More generally, it extends to any set $A\in \Borel(R)$ as $((y,v), (A\times B))\mapsto \TPK(v, A_y)$, where 
\begin{equation}
\label{eq:set_projection}
A_y=\{v\in I\mid (y,v)\in A\}
\end{equation}
is the projection of $A\cap \{x=y\}$ onto the second coordinate.
We define the map $T$
\begin{equation*}
\begin{array}{rrcl}
T:&\R^2&\rightarrow&\R^2\\
&(x,v)&\mapsto&(x+hv, v),
\end{array}
\end{equation*}
and we let $\BiTPK$ be the pull-back of the product kernel $\TPK$ by this map. Explicitly, for $A\in\Borel(R)$,
\begin{equation}
\label{eq:bitpk_definition}
\BiTPK((u,v), A) = \TPK(v, A_{u+hv}).
\end{equation}
In what follows, we show~\eqref{eq:mixing_condition} for the Markov chain $((\Frac\gamma_n, V_n))_{n\in\N}$, which has transition probability kernel $\Frac_\star\BiTPK$.
\begin{figure}
	\noindent\includestandalone[width=1.0\textwidth]{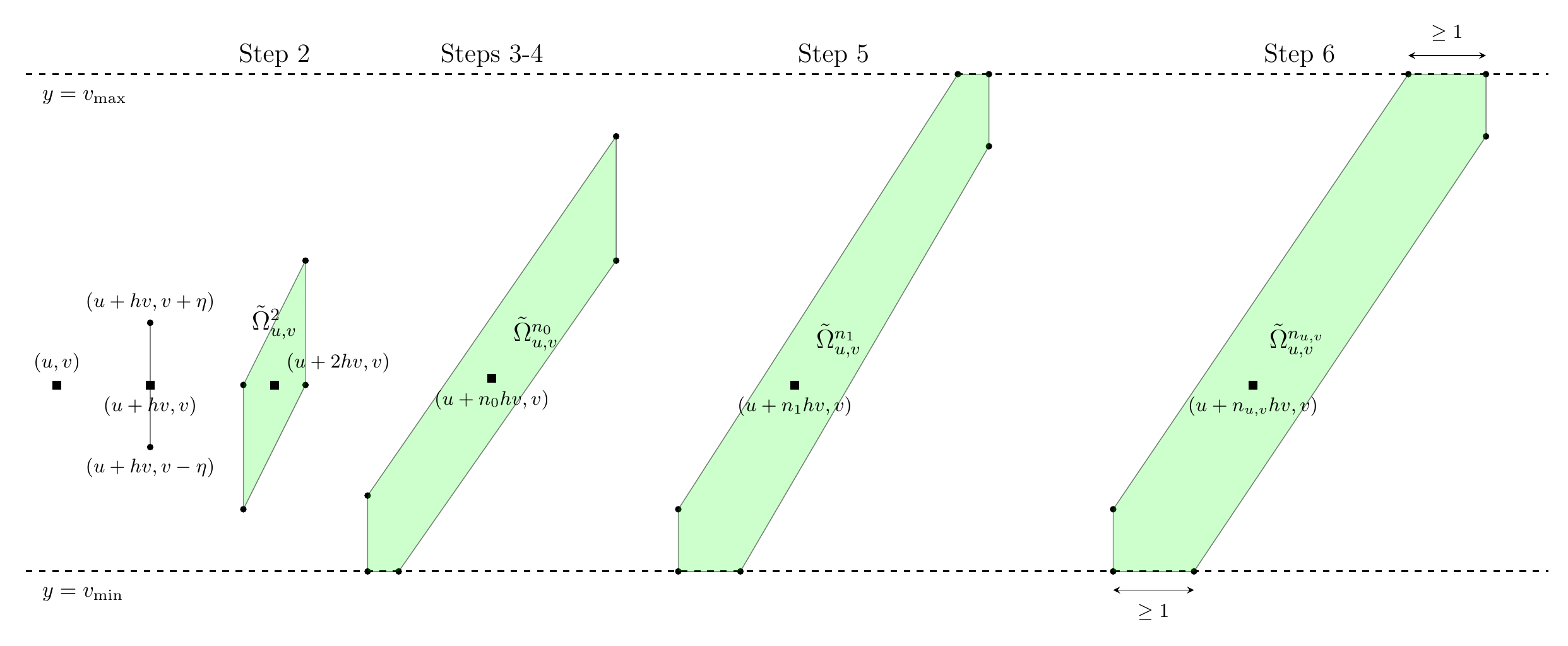}
	\caption{A schematic illustration of the form of a densitys' support. The density lower--bounds $\BiTPK^n((u,v), \cdot))$.}
	\label{fig:mixing_proof_overview}
\end{figure}

Figure~\ref{fig:mixing_proof_overview} illustrates the proof. For $(u,v)\in R$, we show that $\BiTPK^n((u,v), \cdot))$ is lower--bounded by a uniform measure of which we carefully characterise the support, $\Omega^n_{u,v}$. In Steps 1-6, we show that for a certain $n_{u,v}\in\N$, the support of this uniform measure, $\Omega^{n_{u,v}}_{u,v}$, is large enough. In~\nameref{mixing_proof:step_uniform_lower_bound}, we show that $n_{u,v}\leq N\in\N$, for all $(u,v)\in R$. We conclude in~\nameref{mixing_proof:step_conclusion} by showing~\eqref{eq:mixing_condition}. Compared with the \iid~ case treated in Section~\ref{sec:mixing_proof_iid},~\nameref{mixing_proof:step_n2} is the analogue of Lemma~\ref{lemma:convolution_two_measures}, except that the iteration requires the additional Steps 2-5.

\stepparagraph{lower--bound for $\BiTPK^2((u,v),\cdot)$}
\label{mixing_proof:step_n2}
For $(u,v)\in R$ and $(z_1,z_2)\in R$, according to~\eqref{eq:set_projection},
\begin{align*}
([0,z_1]\times[\vmin, z_2])_{u+h(v+y)}
= \begin{cases}
[\vmin,z_2], & \text{ if } u+h(v+y)\in [0,z_1],\\
\emptyset, &\text{ otherwise.}
\end{cases}
\end{align*}
In~\eqref{eq:bitpk_definition}, we observe that integrating with respect to $\BiTPK^2$ amounts to integrating $\TPK$ along a vertical strip, so marginalizing with respect to $(\gamma_1, V_1)$,
\begin{align*}
\BiTPK^2((u,v), \rbrack -\infty, z_1\rbrack\times \lbrack \vmin, z_2\rbrack)
&= \int_R \BiTPK((u,v), dxdy))\BiTPK((x,y), \rbrack -\infty, z_1\rbrack\times \lbrack \vmin, z_2\rbrack)\\
&=\int_I \TPK(v, dy)\TPK(y, (\rbrack -\infty, z_1\rbrack\times \lbrack \vmin, z_2\rbrack)_{u+h(v+y)})\\
&= \int_{\vmin}^{\max((z_1-u)/h - v, \vmax)} \TPK(v,dy)\TPK(y, [\vmin, z_2])
\end{align*}
Differentiating the above expression with respect to $z_1$ and then $z_2$, 
for $z_1\leq u + h(v+\vmax)$, we get
\begin{align*}
\frac{\partial \BiTPK^2((u,v), \rbrack -\infty, z_1\rbrack\times \lbrack \vmin, z_2\rbrack)}{\partial z_1}
&=
\density{v}\left(\tfrac{z_1-u}{h} - v \right) P\left(\tfrac{z_1-u}{h} - v, \rbrack \vmin, z_2\rbrack\right)\\
\density{(u,v)}^{\star2}(z_1,z_2) = \frac{\partial^2 \BiTPK^2((u,v), \rbrack -\infty, z_1\rbrack\times \lbrack 0, z_2\rbrack)}{\partial z_1\partial z_2}
&=
\density{v}\left(\tfrac{z_1-u}{h} - v \right) \density{\tfrac{z_1-u}{h} - v}(z_2).
\end{align*}
As $\density{v}(y)\geq \mu_0 \indicator{[v-\eta, v+\eta]}(y)$, we have
$\density{(u,v)}^{\star2}(z_1,z_2)\geq \mu_0^2$, if
\begin{equation*}
\left\{\begin{aligned}
\frac{z_1-u}{h} -v &\in [\max(\vmin,v-\eta), \min(\vmax, v+\eta)],\\
z_2&\in\left[\max\left(\vmin,\frac{z_1-u}{h} -v - \eta\right), \min\left(\vmax, \frac{z_1-u}{h} -v + \eta\right) \right].
\end{aligned}\right.
\end{equation*}
The above is equivalent to
\begin{equation}
\label{eq:parallelogram_n2}
\left\{\begin{aligned}
z_1 &= u+2hv +k h\eta \\
z_2 &= v + (k+l)\eta,
\end{aligned}\right.\qquad \text{ for some } l\in[-1,1], k\in [-1,1]\cap \left[\tfrac{\vmin-v}{\eta}, \tfrac{\vmax-v}{\eta}\right].
\end{equation}
So, $\BiTPK^2((u,v),\cdot)$ has a density $\density{(u,v)}^{\star 2}$ with respect to the Borel measure on $\R^2$. That density is lower--bounded: for
$(z_1,z_2)\in R\cap\Omega^2_{(u,v)}$, we have $\density{(u,v)}^{\star2}\left(z_1,z_2\right)\geq \mu_0^2$, where
\begin{equation}
\label{eq:omega2_definition}
\Omega^2_{(u,v)}\coloneqq \{(u+2hv, v)+k(h\eta, \eta) + l(0,\eta)\mid k,l\in[-1,1]\}.
\end{equation}
When $\tfrac{\vmin -v}{\eta}<-1$ and $1<\tfrac{\vmax-v}{\eta}$, then $\Omega^2_{(u,v)}\subset R$ and we carry on with the induction to~\nameref{mixing_proof:step_induction}. Otherwise, we go directly to~\nameref{mixing_proof:step_first_boundary} as $\Omega^{n+1}\cap R^c\neq \emptyset$.

\stepparagraph{Lower--bound for $n\geq 3$, while $\Omega^{n}_{(u,v)}\cap R^c = \emptyset$}
\label{mixing_proof:step_induction}
We start by defining the parallelograms $\Omega^{n}_{(u,v)}$ and showing some properties of the vectors that generate them.
Then, by induction that for $n\geq2$, we will show the following statement:
\begin{equation}
\label{eq:bitpk_n_lower_bound}
\parbox{0.8\textwidth}{
	\raggedright
	For $0<\epsilon<\min\left(\tfrac{1}{4}, \tfrac{\eta}{2}(\vmax-\vmin)\right)$ and $\eta<(\vmax-\vmin)/2$,
	$\BiTPK^{n}$~has~a~density~$\density{}^{\star n}$ lower--bounded by $\left(\tfrac{\eta\epsilon}{2}\right)^{n-2}\mu_0^n$ on $\Omega^n_{(u,v)}$.
}
\end{equation}
Our induction is valid while $\Omega^{n}_{u,v}\subset R$ and \nameref{mixing_proof:step_first_boundary} shows how to modify it when it ceases to be the case. Our arguments become progressively more geometric, for what we find the illustration of the proof in Figure~\ref{fig:omega_n} helpful.
\begin{figure}
	\centering
	\noindent\includestandalone[width=1.0\textwidth]{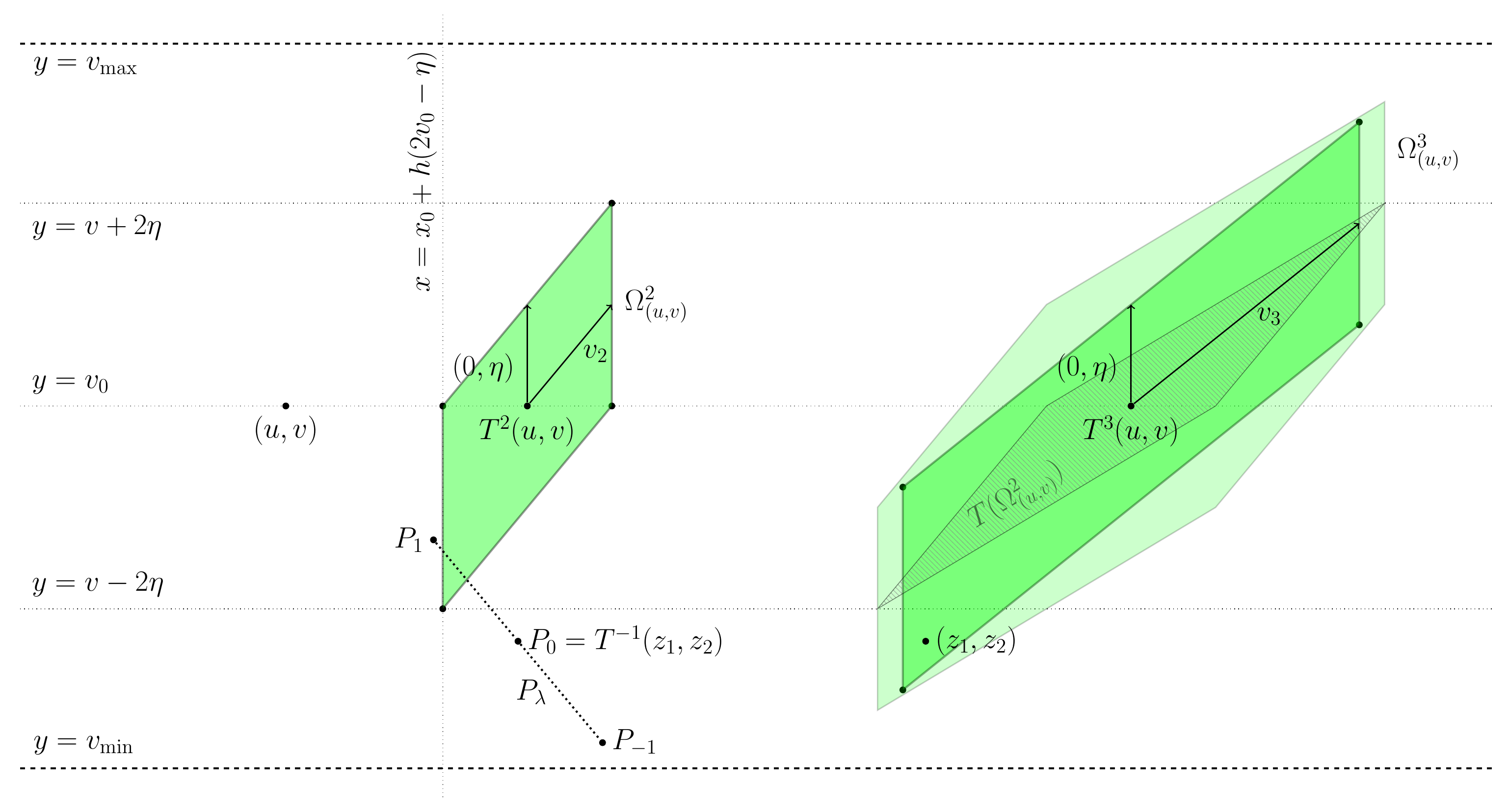}
	\caption{Illustration of $\Omega^n_{(u,v)}$ for $n=2,\,3$ and of the segment $P_\lambda$. Our argument consists in showing that for any $(z_1,z_2)\in\Omega^{n+1}_{(u,v)}$, the length of the intersection of $P_\lambda$ with $\Omega^n_{(u,v)}$ is at least $\eta\epsilon/2$. While the dark green region is $\Omega^3$, the lighter colour shows a larger region where the lower--bound is valid.}
	\label{fig:omega_n}
\end{figure}

To define $\Omega^n_{(u,v)}$, let $v_2\coloneqq T\RowVec{0, \eta} = \RowVec{h\eta,\eta}$ and for $n\geq 3$,
\begin{equation}
\label{eq:vn_definition}
v_n = (1-\epsilon)\left(T\RowVec{0,\eta} + T(v_{n-1})\right)\in\R^2.
\end{equation}
For $n\geq 3$, we define 
\begin{equation}
\label{eq:omegan_definition}
\Omega^{n}_{(u,v)}\coloneqq \left\{T^n\RowVec{u,v} + l\RowVec{0,\eta} + k v_n\mid l,k\in [-1,1]\right\}.
\end{equation}
Notice that if we take $n=2$ in~\eqref{eq:omegan_definition}, we get $\Omega^2_{(u,v)}$ from as defined in~\eqref{eq:omega2_definition}.

While one can obtain an explicit expression of $v_n$, it is of little pratical interest: we only need to ensure that the horizontal component of $v_n$ remains sufficiently large. This is detailed in the proof of Lemma~\ref{lemma:length of segment}.

Since we have shown the statement~\eqref{eq:bitpk_n_lower_bound} for $n=2$, we proceed with the induction step. For $\RowVec{z_1,z_2}\in\Omega^{n+1}_{(u,v)}\cap R$, we calculate
\begin{equation*}
\begin{aligned}
\BiTPK^{n+1}((u,v), \rbrack -\infty, z_1\rbrack\times \lbrack \vmin, z_2\rbrack)
&= \int_{R\cap \{x+yh\leq z_1\}} \BiTPK^{n}((u,v),dxdy) P(y, [\vmin, z_2])\\
&= \int_{R\cap \{x+yh\leq z_1\}}\density{(u,v)}^{\star {n}}(x, y)P(y, [\vmin, z_2])dxdy.
\end{aligned}
\end{equation*}
We can rewrite $R\cap \{x+yh\leq z_1\} = \{(x,y)\in\R\mid y\in I,\, x\leq z_1-yh\}$.
Differentiating with respect to $z_1$, we obtain
\begin{equation*}
\frac{\partial \BiTPK^{n+1}((u,v), \rbrack -\infty, z_1\rbrack\times \lbrack \vmin, z_2\rbrack)}{\partial z_1} =
\int_{y}\density{(u,v)}^{\star {n}}(z_1-yh, y)P(y, [\vmin, z_2])dxdy,
\end{equation*}
where $\density{v}$ is defined in~\eqref{eq:borel_lower_bound_density}.
For $z_2\geq \vmin$, we get
\begin{equation}
\label{eq:density_n}
\density{(u,v)}^{\star n+1}(z_1,z_2) = \frac{\partial^2 \BiTPK^{n+1}((u,v), \rbrack -\infty, z_1\rbrack\times \lbrack \vmin, z_2\rbrack)}{\partial z_1\partial z_2}
= \int_I \density{(u,v)}^{\star {n}}(z_1-yh, y) \density{y}(z_2)dy.
\end{equation}
The expression in~\eqref{eq:density_n} is similar to that from~\nameref{mixing_proof:step_n2}, except that it is integrated over $I$. We can lower--bound the integrand: $\density{(u,v)}^{\star {n}}$ is lower--bounded by $\left(\tfrac{\eta\epsilon}{2}\right)^{n-2}\mu_0^{n}$ on $\Omega^{n}_{(u,v)}$ for $n\geq 2$ and $\density{y}$ by $\mu_0$ on $[y-\eta, y+\eta]\cap I$. To lower-bound $\density{(u,v)}^{\star n+1}$, it remains to lower--bound the length of the integration domain.
For the calculations, we take the following parametrisation of $[z_2-\eta,z_2+\eta]\cap I$,
\begin{equation}
\label{eq:lambda_segment}
\left\{\lambda\in[-1,1]\mid P_\lambda\coloneqq P_0 + \lambda\eta (-h, 1)\in\Omega^{n}_{(u,v)}\right\},\qquad \text{where } P_0 = T^{-1}(z_1,z_2).
\end{equation}
\begin{lemma}
	\label{lemma:length of segment}
	The length of the segment~\eqref{eq:lambda_segment} is at least $\tfrac{\epsilon}{2}$.
\end{lemma}
For the sake of readablity, we differ the proof of Lemma~\ref{lemma:length of segment} to Section~\ref{appendix:lemma_segment_length}. Finally, going back to~\eqref{eq:density_n}, we have the desired lower--bound
\begin{equation*}
\density{(u,v)}^{\star (n+1)}(z_1, z_2) \geq \left(\tfrac{\eta\epsilon}{2}\right)^{n-2}\mu_0^n \times\mu_0\times \left(\tfrac{\eta\epsilon}{2}\right) = \left(\tfrac{\eta\epsilon}{2}\right)^{(n+1)-2}\mu_0^{n+1},
\qquad\text{ for }(z_1,\,z_2)\in\Omega^{n+1}_{u,v}.
\end{equation*}
In addition, for $\epsilon<1\left/\left(1+\tfrac{3(\vmax-\vmin)}{2\eta}\right)\right.$,
$$\ScalarProdR{\RowVec{0,1}}{(v_{n+1}-v_n)} = (1-\epsilon)\eta - \epsilon\ScalarProdR{\RowVec{0,1}}{v_n}>(1-\epsilon)\eta - \epsilon(\vmax-\vmin)> \epsilon\tfrac{\vmax-\vmin}{2}.$$
The height of $\Omega^n_{u,v}$ grows with $n$, by at least a constant, positive term. Hence, it eventually reaches ${\vmax-\vmin}$, in which case $\Omega^n\cap R^c\neq \emptyset$.

\stepparagraph{First non--empty intersection with the boundary}
\label{mixing_proof:step_first_boundary}
Let $n_0\coloneqq\min\{n\in\N\mid \mu(\Omega^n\cap R^c)>0\}$. Without loss of generality, $\Omega^{\nzero}$ extends beyond $\vmin$. We will now construct a region $\tilde{\Omega}^{\nzero+1}\subset R$ such that $\density{(u,v)}^{\star (\nzero+1)}\geq \left(\tfrac{\eta\epsilon}{2}\right)^{\nzero-2}\mu_0^{\nzero}$ on $\tilde{\Omega}^{\nzero+1}$ and for which $\tilde{\Omega}^{\nzero+1}\cap (\R\times\{\vmin\})$ is lower--bounded. Since we can choose $\eta$ arbitrarily small, we can treat the lower and upper boundaries independently, so we focus on the construction of $\tilde{\Omega}^{\nzero+1}_{u,v}$ on the boundary $\R\times\{\vmin\}$ first.

For $P\in\R\times\{\vmin\}$, we consider $P_\lambda$ as in~\eqref{eq:lambda_segment}, under the constraint that the integration segment lies within $R$, that is, $\{\lambda\in[0,1]\mid P_\lambda\in \Omega^{\nzero}\}$. We denote the length of this segment by $L(P)$
\begin{equation*}
L(P)\coloneqq\left\vert\left\{\lambda\in[-1,1]\mid P + \eta\lambda\RowVec{-h,1}\in\Omega^{\nzero}\cap R\right\}\right\vert,
\end{equation*}
and we let $A,B$ be the endpoints of $\Omega^{\nzero}\cap (\R\times\{\vmin\})$.
We rely on the following claim, whose proof is in Section~\ref{proof:claim_DE_geq_AB}. Figure~\ref{fig:boundary_scenarios} illustrates the situation.
\begin{equation}
\label{claim:DE_geq_AB}
\parbox{0.8\textwidth}{
	\raggedright The set $\{P\in\Omega^{\nzero}\cap \R\times\{\vmin\}\mid L(P)\geq \tfrac{\epsilon}{2}\}$ is not empty. If we denote by $D=A+\RowVec{x_D,0}$ and $E=A+\RowVec{x_E,0}$ its right- and left-endpoints, then for some $c_0>0$ independent of $(u,v)$, we have
	$$x_B + c_0\leq x_D - x_E.$$}
\end{equation}
\begin{figure}
	\centering
	\noindent\includestandalone[width=1.0\textwidth]{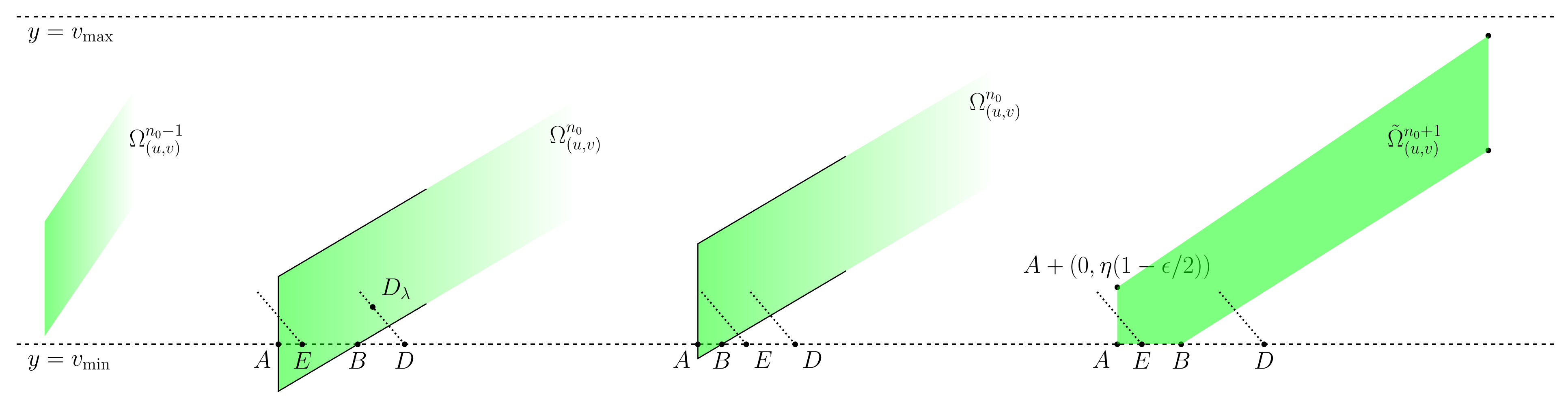}
	\caption{Illustration of~\nameref{mixing_proof:step_first_boundary} and the proof of~\eqref{claim:DE_geq_AB}. The leftmost polygon represents $\Omega^{\nzero-1}$, at the iteration before the first non-trivial intersection occurs. The two middle parallelograms illustrate the two cases, $x_E\leq x_B$ and $x_B\leq x_E$ respectively, from the proof of~\eqref{claim:DE_geq_AB}. On the right, the bottom part of the polygon $\tilde{\Omega}^{\nzero+1}$ as constructed in~\nameref{mixing_proof:step_first_boundary}. The dashed lines represent the integration segments, whose length is measured by $L$.}
	\label{fig:boundary_scenarios}
\end{figure}
In particular, $L(P)\geq \tfrac{\epsilon}{2}$ implies that $\density{(u,v)}^{\star \nzero+1}(P)\geq \left(\tfrac{\eta\epsilon}{2}\right)^{\nzero-1}\mu_0^{\nzero+1}$.
By convexity of $\Omega^{\nzero}\cap R$, we have $L(P)\geq \tfrac{\epsilon}{2}$ for $P$ on the segment $T(E)T(D)$, so we can include that segment in $\tilde{\Omega}^{\nzero+1}$. As $L(E)=L(P)$, where $P= E + (1-\epsilon/2)k\eta\RowVec{-h,1},$ for $k\in[0,1]$, we have the same lower--bound on the density holds on $T(P)$. So, we can include a segment of height $\eta(1-\epsilon/2)$ above $T(E)$ in $\tilde{\Omega}^{\nzero+1}$.
Therefore, we can define $\tilde{\Omega}^{\nzero+1}$ as the polygon with vertices $T(E)+\RowVec{0,\eta(1-\epsilon/2)}$, $T(E)$, $T(D)$, $T^{\nzero+1}\RowVec{u,v}-\RowVec{0,\eta} + v_{\nzero+1}$ and $T^{\nzero+1}\RowVec{u,v}+\RowVec{0,\eta} + v_{\nzero+1}$.

We have obtained a convex pentagon $\tilde{\Omega}^{\nzero+1}$ on which $\BiTPK^{\nzero+1}((u,v),\cdot)$ is lower--bounded by a measure with density lower--bounded by $\left(\tfrac{\eta\epsilon}{2}\right)^{\nzero-1}\mu_0^{\nzero+1}$. Because $T$ preserves lengths on horizontal cross-sections,~\eqref{claim:DE_geq_AB} implies that the length of $T(D)T(E)$ is equal to that of $ED$, which is longer  by $c_0=\eta h/4$ than the intersection at $\nzero$.

\stepparagraph{Induction for $n > \nzero+1$}
\label{mixing_proof:step_boundary_iteration}
Assume that $\tilde{\Omega}^{\nzero+1}\cap(\R\times\{\vmax\})=\emptyset$. As a consequence of calculations for~\nameref{mixing_proof:step_induction}, $\tilde{\Omega}^{n}_{u,v}$ is growing upwards. Indeed, the calculations rely on Assumption~\eqref{eq:borel_lower_bound_density} and the fact that $v_n$ has a horizontal component whose length we control. Therefore, they adapt to $\tilde{\Omega}^{n}_{u,v}$, with $v_n$ being the vector from $T(D)$ to ${T^{\nzero+1}\RowVec{u,v}-\RowVec{0,\eta} + v_{\nzero+1}}$.

In addition,~\eqref{claim:DE_geq_AB} still holds. Indeed, redefine $A,\,B,\,D$ and $E$, except with $\nzero$ replaced by $\nzero+1$ in the expression of $L$. We notice that $A,\,B$ coincide with $T(E)$ and $T(D)$ from the previous iteration. Because $AB$ is now of length at least $c_0=h\eta/4$, the proof is easier as we fall in the first case. We define $\tilde{\Omega}^{\nzero+2}$ as in~\nameref{mixing_proof:step_first_boundary}. 

We can now iterate this procedure, obtaining a lower--bound of $\density{u,v}^{\star n}$ by a uniform constant, on a convex and polygonal domain $\tilde{\Omega}^{n}$. Crucially, both the height of $\tilde{\Omega}^n$ and the length of its intersection with $\R\times\{\vmin\}$ grow, by uniformly lower--bounded amounts.

\stepparagraph{Intersection with both boundaries}
\label{mixing_proof:step_boundaries}
For some $\nOne\in\N$, the intersection $\tilde{\Omega}^{\nOne}_{u,v}\cap (\R\cap\{\vmax\})$ is not trivial. By a procedure analogue to that in~\nameref{mixing_proof:step_first_boundary}, we can define $\tilde{\Omega}^{\nOne+1}$, which non--trivially intersects both boundaries. Using the procedure from~\nameref{mixing_proof:step_boundary_iteration}, it is clear that the intersection will not only remain non--trivial with $n$, but also increase.

\stepparagraph{Cross-sections with length at least $1$}
\label{mixing_proof:step_cross_sections}
By definition, $\tilde{\Omega}^n$ is delimited by a convex, polygonal domain.
The length of any horizontal cross-section of $\tilde{\Omega^n}$ is lower--bounded by the minimum of the lengths of the intersections with the lower and upper boundaries\footnote{To see this, consider the parallelogram on the 4 vertices of $\tilde{\Omega}^n$ which belong to the boundary. That parallelogram is included in $\tilde{\Omega^n}$ by convexity, so the lengths of the horizontal sections between the length of both bases.}. Recall that by~\nameref{mixing_proof:step_boundary_iteration}, these two are increasing, and this by at least $h\eta/4$ at each iteration. Hence, for some $n=n_{u,v}$, all horizontal sections of $\tilde{\Omega}^{n_{u,v}}_{u,v}$ are of length at least 1.

By construction of $\tilde{\Omega}^{n}_{(u,v)}$, we have obtained a region such that for any $n\geq n_{u,v}$,,
\begin{enumerate}
	\item $\BiTPK^{n}((u,v), \cdot)$ is lower--bounded by $\left(\frac{\eta\epsilon}{2}\right)^{n-2}\mu_0^{n}\mu$ on $\tilde{\Omega}^n_{u,v}$, ($\mu$ being the Lebesgue measure)
	\item $\{\tilde{\Omega}^n_{(u,v)} + (k,0)\}_{k\in \mathbb{Z}}$ is a cover of $R$.
\end{enumerate}

\stepparagraph{Uniform lower--bound}
\label{mixing_proof:step_uniform_lower_bound}
We now show that we can choose a uniform $N\in\N$, such that $n_{u,v}\leq N$ for all $(u,v)\in R$. Fix $(u,v)\in R$ and let $\bar{\Omega}^2_{u,v}$ be defined as in~\eqref{eq:parallelogram_n2}, except with $\tfrac{\eta}{2}$ instead of $\eta$. We can then perform~\nameref{mixing_proof:step_n2} to~\nameref{mixing_proof:step_cross_sections}, so we obtain a domain $\bar{\Omega}^{\bar{n}_{u,v}}_{u,v}$ with cross-sections of length at least 1, for some $\bar{n}_{u,v}\geq n_{u,v}$.

Notice that the shrinked parallelogram at $n=2$ is contained in parallelograms for different initial conditions. Specifically, we have $\bar{\Omega}^2_{u,v}\subset\Omega^{2}_{x,y}$ for $(x,y)\in C_{u,v}$, where $C_{u,v} = T^{-2}(\bar{\Omega}^2_{u,v})$. In particular, $n_{x,y}\leq\bar{n}_{u,v}$, for all $(x,y)\in C_{u,v}$. Since $(\overset{\circ}{C_{u,v}})_{(u,v)\in [0,1]\times I}$ is an open cover of $[0,1]\times I$, by compacity, we can find a finite cover $\{C_{u_k,v_k}\}_{k=1}^K$. 
Clearly, $N=\max_{1\leq k\leq K} \bar{n}_{u_k,v_k}<\infty$ gives a uniform bound on $(n_{u,v})_{(u,v)\in[0,1]\times I}$. The bound is also valid on $R\times I$, because the whole construction is invariant with respect to horizontal translations.

Finally, for $(u,v)\in R$, we have that $\BiTPK^{N}((u,v), \cdot)$ is lower--bounded by $\left(\frac{\eta\epsilon}{2}\right)^{N-2}\mu_0^{N}\mu$, on $\Omega_{u,v}$ and  $\{\Omega_{(u,v)} + (k,0)\}_{k\in \mathbb{Z}}$ is a cover of $R$, where $\Omega_{u,v}\coloneqq\tilde{\Omega}^N_{u,v}$.

\stepparagraph{Conclusion}
\label{mixing_proof:step_conclusion}
We can now go back to $(\Frac\gamma_n, V_n)$. By lower--bounding $\BiTPK^N$ with a uniform measure, we can use the same arguments as in Section~\ref{sec:mixing_proof_iid} to conclude. For $A\in\Borel(\lbrack0,1\rbrack\times I)$, we have
\begin{align*}
\Frac_\star\BiTPK^{N}((u,v), A)
&= \BiTPK^{N}((u,v),\Frac^{-1}(A))&\\
&\geq \BiTPK^{N}((u,v), \Frac^{-1}(A)\cap \Omega_{(u,v)})&\\
&\geq C\mu(\Frac^{-1}(A)\cap \Omega_{(u,v)})& (\text{minorating on $\Omega_{(u,v)}$}) \\
&= C\mu(\cup_{k\in Z} A+(k,0)\cap \Omega_{(u,v)})& (\{A+(k,0)\}_k \text{ disjoint}) \\
&= C\sum_{k\in Z}\mu(A+(k,0)\cap \Omega_{(u,v)}) \\
&= C\sum_{k\in Z}\mu(A\cap (\Omega_{(u,v)}-(k,0)))& (\mu \text{ translation--invariant})\\
&\geq C\mu(\cup_{k\in Z} A\cap (\Omega_{(u,v)}-(k,0))) \\
&= C\mu(A)& (\{\Omega_{(u,v)} + (k,0)\}_{k\in \mathbb{Z}}\text{ is a cover of $R$}),
\end{align*}
where $ C = C_{\eta,\epsilon,\mu_0, N} \coloneqq \left(\frac{\eta\epsilon}{2}\right)^{N-2}\mu_0^{N}$.
The lower--bound is uniform in $(u,v)$ and also shows that the measure is non-trivial. We conclude the proof of Proposition~\ref{prop:gamma_beta_mixing} by applying Theorem~\ref{theorem:mixing}.
\subsubsection{Proof of Lemma~\ref{lemma:length of segment}}
\label{appendix:lemma_segment_length}
We recall that for some $l,k\in[-1,1]$,
\begin{equation*}
\RowVec{z_1,z_2} = T^{n+1}\RowVec{u,v} + l\RowVec{0,\eta} + kv_n,
\end{equation*}
where $v_n$ is given in~\eqref{eq:vn_definition}, so
\begin{equation}
\label{eq:plambda_definition}
P_\lambda
\coloneqq T^{-1}\RowVec{z_1,z_2+\lambda\eta}
= T^n\RowVec{u,v} + \eta(l+\lambda)\RowVec{-h,1} + k(1-\epsilon)\left(\RowVec{0,\eta}+v_n\right).
\end{equation}
For a parallelogram $\Omega$ generated by vectors $x,\,y$ and centered around the origin, we have
\begin{equation}
\label{eq:parallelogram_containment_center}
P \in \Omega \iff
\left\{\begin{array}{ccccc}
\ScalarProd{x}{y^\perp} &\leq& \ScalarProd{P}{y^\perp} &\leq& \ScalarProd{-x}{y^\perp}\\ 
\ScalarProd{-y}{x^\perp} &\leq& \ScalarProd{P}{x^\perp} &\leq& \ScalarProd{y}{x^\perp},\\
\end{array}
\right.
\end{equation}
where $(x_1,x_2)^\perp = (x_2,-x_1)$.
Combining~\eqref{eq:plambda_definition} with~\eqref{eq:parallelogram_containment_center}, we have
that $P_\lambda \in \Omega^{n}_{(u,v)}$ if and only if
\begin{align*}
&
\left\{\begin{array}{c}
\eta \ScalarProdR{\RowVec{0,1}}{v_n^\perp} \leq
\eta(l+\lambda) \ScalarProdR{\RowVec{-h,1}}{v_n^\perp}
+ k(1-\epsilon)\left(\eta\ScalarProdR{\RowVec{0,1}}{v_n^\perp} +\ScalarProdR{v_n}{v_n^\perp}\right)
\leq - \eta\ScalarProdR{\RowVec{0,1}}{v_n^\perp}\\ 
-\eta\ScalarProdR{v_n}{\RowVec{0,1}^\perp} \leq
\eta^2(l+\lambda)\ScalarProdR{\RowVec{-h,1}}{\RowVec{0,1}^\perp} + k(1-\epsilon)\left(\eta^2\ScalarProdR{\RowVec{0,1}}{\RowVec{0,1}^\perp} + \eta\ScalarProdR{v_n}{\RowVec{0,1}^\perp}\right)
\leq \eta\ScalarProdR{v_n}{\RowVec{0,1}^\perp}\\
\end{array}
\right.\\
\iff&
\left\{\begin{array}{ccccc}
\ScalarProdR{\RowVec{1,0}}{v_n}(-1+k(1-\epsilon)) &\leq&
-(l+\lambda)\ScalarProdR{\RowVec{1,h}}{v_n}
&\leq& \ScalarProdR{\RowVec{1,0}}{v_n}(1+k(1-\epsilon))\\ 
\ScalarProdR{v_n}{\RowVec{1,0}}(-1-k(1-\epsilon)) &\leq&
(-h\eta)(l+\lambda)
&\leq& \ScalarProdR{v_n}{\RowVec{1,0}}(1-k(1-\epsilon)).\\
\end{array}
\right.
\end{align*}
As $\ScalarProdR{\RowVec{1,h}}{v_n}>0$ and denoting
\begin{equation*}
a_n\coloneqq\frac{1}{h\eta}\ScalarProdR{\RowVec{1,0}}{v_n},\qquad
b_n \coloneqq 1-\frac{\ScalarProdR{\RowVec{0,h}}{v_n}}{\ScalarProdR{\RowVec{1,h}}{v_n}},
\end{equation*}
we have
\begin{equation*}
P_\lambda\in\Omega^n
\iff
\left\{\begin{array}{ccccc}
b_n(-1-k(1-\epsilon)) -l&\leq&
\lambda
&\leq& b_n(1-k(1-\epsilon)) -l\\ 
a_n(-1+k(1-\epsilon)) -l &\leq&
\lambda
&\leq& a_n(1+k(1-\epsilon))-l.\\
\end{array}
\right.
\end{equation*}
Finally, taking into account that $\lambda\in[-1,1]$, we obtain that
\begin{equation*}
\lambda\in[\max(-1, b_n(-1-k(1-\epsilon))-l, a_n(-1+k(1-\epsilon)) -l),
\min(1, b_n(1-k(1-\epsilon))-l, a_n(1+k(1-\epsilon)) -l)],
\end{equation*}
which is of length
\begin{align*}
\min(&1, b_n(1-k(1-\epsilon))-l, a_n(1+k(1-\epsilon)) -l)+\\
+ \min(&1, b_n(1+k(1-\epsilon))+l, a_n(1-k(1-\epsilon))+l)=\\
=\min(&2\min(1,a_n,b_n), (a_n+b_n)(1-k(1-\epsilon)),
1+b_n(1+k(1-\epsilon))+l,\\
&1+a_n(1-k(1-\epsilon))+l, 1+b_n(1-k(1-\epsilon))-l, 1+a_n(1+k(1-\epsilon)-l
)
\end{align*}
We claim that for $n\geq 2$,
\begin{equation}
\label{eq:scalar_product_controls}
a_n \geq 1,\qquad
b_n \geq\frac{1}{2}.
\end{equation}
Combining~\eqref{eq:scalar_product_controls} with $l,k\in[-1,1]$, $0<\epsilon\leq \tfrac{1}{2}$, we conclude that
the length of~\eqref{eq:lambda_segment} is at least $\tfrac{\epsilon}{2}$.

It remains to show~\eqref{eq:scalar_product_controls}. For $a_n$, we proceed by induction. Using $v_2=T\RowVec{0,\eta}$ for $n=2$ and
$v_3 = (1-\epsilon)\eta\RowVec{3h,2}$, we verify that $a_2,\,a_3\geq 1$. Notice that $\ScalarProdR{\left(T\RowVec{x,y}\right)}{\RowVec{1,0}}= \ScalarProdR{\left(\RowVec{x,y} + \RowVec{yh,0}\right)}{\RowVec{1,0}} =\ScalarProdR{\RowVec{x,y}}{\RowVec{1,h}}$.
Then,
\begin{equation*}
\ScalarProdR{v_{n+1}}{\RowVec{1,0}}
= (1-\epsilon)\ScalarProdR{\left(T\RowVec{0,\eta} + T(v_n)\right)}{\RowVec{1,0}}
= (1-\epsilon)\left[h\eta + \ScalarProdR{v_n}{\RowVec{1,h}}\right].
\end{equation*}
Using the induction hypothesis, $\ScalarProdR{v_n}{\RowVec{1,0}}\geq h\eta$ combined with $\ScalarProdR{v_n}{\RowVec{0,h}}\geq0$,
\begin{equation*}
\ScalarProdR{v_{n+1}}{\RowVec{1,0}} \geq 2h\eta(1-\epsilon)\geq h\eta,
\end{equation*}
since $\epsilon\leq\tfrac{1}{2}$.

For $b_n$, we can calculate directly $b_2=\tfrac{h\eta}{h\eta+h\eta} = \tfrac{1}{2}$. For $n\geq 3$, we can express $v_n$ using~\eqref{eq:vn_definition}, so that
\begin{align*}
\frac{\ScalarProdR{\RowVec{1, 0}}{v_{n}}}{\ScalarProdR{\RowVec{1,h}}{v_n}}
=& \frac{\ScalarProdR{\RowVec{1, 0}}{\left(T\RowVec{0,\eta} + T(v_{n-1})\right)}}{\ScalarProdR{\RowVec{1,h}}{\left(T\RowVec{0,\eta} + T(v_{n-1})\right)}}\\
=& \frac{h\eta + \ScalarProdR{\RowVec{1, 0}}{T(v_{n-1})}}{2h\eta + \ScalarProdR{\RowVec{1,h}}{T(v_{n-1})}}\\
=& \frac{1}{2} + \frac{1}{2}\frac{\ScalarProdR{\left(\RowVec{1, 0}-\RowVec{0,h}\right)}{T(v_{n-1})}}{2h\eta + \ScalarProdR{\RowVec{1,h}}{T(v_{n-1})}}\\
\geq&\frac{1}{2},
\end{align*}
where the last inequality follows from
\begin{equation*}
\ScalarProdR{\left(\RowVec{1, 0}-\RowVec{0,h}\right)}{T(v_{n-1})} = \ScalarProdR{\RowVec{1, 0}}{v_{n-1}} + \ScalarProdR{\RowVec{0,h}}{v_{n-1}} - \ScalarProdR{\RowVec{0,h}}{v_{n-1}} \geq 0.
\end{equation*}
{\flushright\qedsymbol}

\subsubsection{Proof of~\eqref{claim:DE_geq_AB}}
\label{proof:claim_DE_geq_AB}
Notice first that $L(A)=0$ and $L(P)=0$ for any $P\in R\times\{\vmin\}$ to the left of $A$, so that $0\leq x_B, x_D, x_E$. Second, consider $P= A + \RowVec{x_B+(1-\epsilon/2)\eta h\tfrac{1}{1-b},0}$. As $L(P)=\tfrac{\epsilon}{2}$, we have that $\{P\mid L(P)\geq\tfrac{\epsilon}{2}\}\neq \emptyset$, so $D$ and $E$ exist. In addition, we know that $x_D\geq x_B+(1-\epsilon/2)\eta h\tfrac{1}{1-b}$. In particular, $b\leq 1$ implies that $x_B<x_D$.

Since $A\in\Omega^{\nzero}_{u,v}$, we can write $A=T^{\nzero}\RowVec{u,v} + l_A\eta\RowVec{0,\eta} - v_{\nzero}$. By definition, $\nzero$ is the first time such that $\Omega^\nzero\cap R^c$ has non-trivial measure, so, using the relation between $\Omega^{\nzero-1}$ and $\Omega^{\nzero}$, we can conclude that $l_A\leq 0\leq 1-\epsilon/2$.

We distinguish two cases, depending which one of $\eta\epsilon h/2$ or $x_B$ is greater.
First, if $\eta h\epsilon/2\leq x_B$, then $x_E=\eta h\epsilon/2$. Indeed, the triangle formed by $A$, $E$ and $A + \RowVec{0,\eta\epsilon/2}$ is in $\Omega^{\nzero}$, so $L\left(A+\RowVec{\eta h\epsilon/2,0}\right)\geq \epsilon/2$. Therefore,
\begin{equation*}
\begin{aligned}
x_D - x_E
&\geq x_B + (1-\tfrac{\epsilon}{2})\eta h\tfrac{1}{1-b} - \eta h\tfrac{\epsilon}{2}\\
&\geq x_B +\eta h (2(1-\tfrac{\epsilon}{2}) - \tfrac{\epsilon}{2})\\
&\geq x_B + \tfrac{5}{4}\eta h,
\end{aligned}
\end{equation*}
where in the last two inequalities, we have use that $\tfrac{1}{2}\leq b$ and $\epsilon\leq \tfrac{1}{2}$.

Next, if $x_B<\eta h\epsilon/2$, then $\eta h\epsilon/2\leq x_E$. So, for $x\leq \eta h$,
\begin{equation}
\label{eq:LA_formula}
L\left(A+\RowVec{x,0}\right) = \frac{1}{\eta h}\left(x - \tfrac{\ScalarProdR{\RowVec{0, h}}{v_n}}{\ScalarProdR{\RowVec{1, h}}{v_n}}\ (x-x_B)_+\right) = \frac{1}{\eta h}(xb  -x_B(1-b)).
\end{equation}
Notice that $L\left(A+ \RowVec{\eta h,0}\right) \geq \frac{\epsilon}{2}$ for $\epsilon \leq \frac{2}{5}$ small enough,
\begin{align*}
L\left(A+ \RowVec{\eta h,0}\right) - \frac{\epsilon}{2}
&= \frac{1}{\eta h}(\eta h b - x_B(1-b)) - \frac{\epsilon}{2}\\
&= b - (1-b)\frac{x_B}{\eta h} - \frac{\epsilon}{2}\\
&\geq b-(1-b)\frac{\epsilon}{2} - \frac{\epsilon}{2}\\
&= b\left(1+\frac{\epsilon}{2}\right) - \frac{3}{2}\epsilon\\
&\geq \frac{1}{2}\left(1-\frac{5}{2}\epsilon\right),
\end{align*}
so $x_E\leq \eta h$. Using~\eqref{eq:LA_formula}, we find that $x_E=\frac{1}{b}(\eta h\epsilon/2 + x_B(1-b))$. Finally,
\begin{align*}
x_D - x_E - x_B
&= x_B\left(1-\tfrac{1-b}{b}\right) + \eta h\tfrac{1-\epsilon/2}{1-b} - \tfrac{1}{b}\eta h\epsilon/2- \eta h\tfrac{\epsilon}{2}\\
&= x_B(2-\tfrac{1}{b}) + \eta h\left(\tfrac{1}{1-b}+\tfrac{\epsilon}{2}(\tfrac{1}{b} - \tfrac{1}{1-b})\right)- \eta h\tfrac{\epsilon}{2}\\
&= x_B(2-\tfrac{1}{b}) + \tfrac{\eta h}{1-b}\left(1-\tfrac{\epsilon}{b}(b-\tfrac{1}{2}))\right)- \eta h\tfrac{\epsilon}{2}.
\end{align*}
Since $\tfrac{1}{2}\leq b\leq 1$, we have $\tfrac{b-1/2}{b}\leq 1$ and $\tfrac{1}{1-b}\geq 2$, so that
\begin{equation*}
x_D - x_E - x_B
\geq 2(1-\epsilon)\eta h - \eta h\epsilon/2\geq \eta h(1-\tfrac{3}{2}\epsilon).
\end{equation*}
Combining the two cases with $\epsilon<\tfrac{1}{2}$, we conclude that
\begin{equation*}
x_D-x_E\geq x_B + \eta h\min\left(\left(1-\tfrac{\epsilon}{2}\right), \tfrac{5}{4}\right) \geq x_B + \tfrac{1}{4}\eta h.
\end{equation*}
{\flushright{\qedsymbol}}

\section{Mixing-preserving operations: mixing coefficients of $(X_n)_{n\in\N}$}
\label{appendix:mixing_by_measurable_mapping}
\begin{proposition}
	\label{prop:mixing_by_measurable_mapping}
	Let $X_n$ be as in\eqref{eq:signal_window_definition}. For any $k\in \N$,
	\begin{equation*}
	\beta_X(k+M-1) \leq \beta_S(k) \leq \beta_{\Frac(\gamma)}(k) + \beta_W(k).
	\end{equation*}
\end{proposition}
The proposition is a consequence of Lemmata~\ref{lemma:mixing_sum} and~\ref{lemma:mixing_vector}, combined with the fact that $\pattern$ is continuous, so $\beta_{\pattern(\gamma)}(k)\leq \beta_{\Frac(\gamma)}(k)$. The proofs of the lemmata essentially consist in manipulating the definitions.
\begin{lemma}
	\label{lemma:mixing_sum}
	For two random variables $U:(\Omega_U, \sigma^U)\rightarrow \R$, $V:(\Omega_V, \sigma^V)\rightarrow \R$ with $(\beta_U(k))_{k\in\N},\,(\beta_V(k))_{k\in\N}\in\ell^1$ summable, we have
	${\beta_{U+V}(k)\leq \beta_{U}(k)+\beta_V(k)}$.
	If $U$ and $V$ are defined on the same probability space, but are independent, the same holds true.
\end{lemma}
\begin{proof}
	Define $Z\coloneqq U+V$. Then, $Z$ is $(\Omega_Z, \sigma^Z)$-measurable, where $\Omega_Z=\Omega_U\times\Omega_V$ and $\sigma^Z = \sigma^U\otimes\sigma^V$.
	As $\sigma^Z$ is generated by products of elements from $\sigma^U$ and $\sigma^V$, we only need to consider (countable) partitions $\mathcal{A}_U,\mathcal{B_U}$ and $\mathcal{A}_V,\mathcal{B}_V$ of $\sigma^U_{-\infty,0},\sigma^U_{k,\infty}$ and $\sigma^V_{-\infty,0},\sigma^V_{k,\infty}$ respectively.
	For any $A_U\in\mathcal{A}_U,A_V\in\mathcal{A}_V$ and $B_U\in\mathcal{B}_U,B_V\in\mathcal{B}_V$, by definition of the product probability measure, 
	\begin{align*}
	P((A_U{\times} A_V) \cap (B_U\times B_V)) {-} P(A_U{\times} A_V) P(B_U{\times} B_V)
	&= (P_U(A_U\cap B_U) {-} P_U(A_U)P_U(B_U))P_V(A_V\cap B_V)\\
	&{+} P_U(A_U)P_U(B_U)(P_V(A_V\cap B_V){-}P_V(A_V)P_V(B_V)).
	\end{align*}
	Since $\beta_U$ is is summable, $\sum_{A_U, B_U}\vert P_U(A_U\cap B_U) - P_U(A_U)P_U(B_U)\vert<\infty$ (idem for $V$), so we can regroup terms and
	\begin{small}
		\begin{equation*}
		\begin{array}{rlc}
		{\sum\limits_{\substack{A_U\in\mathcal{A}_U, A_V\in\mathcal{A}_V,\\ B_U\in\mathcal{B}_U,B_V\in\mathcal{B}_V}}} P((A_U{\times} A_V) \cap (B_U{\times} B_V)) - P(A_U{\times} A_V) P(B_U{\times} B_V)
		{=}& \multicolumn{2}{l}{\sum\limits_{\substack{A_U\in\mathcal{A}_U,\\ B_U\in\mathcal{B}_U}}(P_U(A_U\cap B_U) {-} P_U(A_U)P_U(B_U))}\\
		&{\times}\sum\limits_{\substack{A_V\in\mathcal{A}_V,\\ B_V\in\mathcal{B}_V}} P_V(A_V\cap B_V)&(=1)\\
		&{+} \sum_{\substack{A_U\in\mathcal{A}_U,\\ B_U\in\mathcal{B}_U}}P_U(A_U)P_U(B_U)&(=1)\\
		&\multicolumn{2}{l}{\times\sum\limits_{\substack{A_V\in\mathcal{A}_V,\\ B_V\in\mathcal{B}_V}}(P_V(A_V\cap B_V){-}P_V(A_V)P_V(B_V))}\\
		\leq&\beta_U(k) +\beta_V(k).&
		\end{array}
		\end{equation*}
	\end{small}
	We conclude by taking the sup over partitions of $\Omega_Z.$
\end{proof}
\begin{lemma}
	\label{lemma:mixing_vector}
	Consider $S=(S_i)_{i\in\N}$ with coefficients $\beta_S(k)$ and define $X_n=(S_{n},\ldots, S_{n+M-1})$.
	Then, $${\beta_X(k+M-1)\leq \beta_S(k).}$$
\end{lemma}
\begin{proof}
	First, note that the $\sigma$-algebra generated by a vector coincides with the $\sigma$-algebra generated by its components
	\begin{align*}
	\sigma(X_{n_1},\ldots X_{n_2})&=\sigma((S_{n_1},\ldots, S_{n_1+M-1}), \ldots, (S_{n_2},\ldots, S_{n_2 +M-1}))\\
	&= \sigma(S_{n_1},\ldots,S_{n_2+M-1})\\
	&= \sigma^S_{n_1,n_2+M-1}.
	\end{align*}
	
	Then, any partition $\mathcal{A}\subset\sigma^X_{n_1,n_2}$ is also in $\sigma^S_{n_1,n_2+M-1}$. Since $\beta_X$ is defined as a sup over such partitions, $\beta_X(k+M-1)\leq \beta_S(k)$.
	For $k\leq M$, we can take $\mathcal{A}=\mathcal{B}\subset \sigma(S_k)$. Since $S_k$ is a continuous random variable, $\beta_X(k)=1$.
\end{proof}

\section{Gaussian approximation for dependent data}
\label{appendix:Gaussian_approximation_results}
\begin{theorem}[{\cite[Theorem 11.22]{kosorok_introduction_2008}}]
	\label{thm:kosorok}
	Let $(X_n)_{n\in\N}\subset\R^d$ be a stationary sequence and consider a functional family $\FunctionalFamily=(F_t)_{t\in \FunctionalDomain}$ with finite bracketing entropy. Suppose there exists $r\in \rbrack2,\infty\lbrack$, such that
	\begin{equation}
	\label{eq:beta_condition}
	\sum_{k=1}^\infty k^{\tfrac{2}{r-2}} \beta_X(k)<\infty,
	\end{equation}
	Then, $\sqrt{N}(\hat{F}_t- F_t)$ converges to a tight, zero--mean Gaussian $G_d$ process with covariance~\eqref{eq:limit_covariance}.
\end{theorem}
\begin{theorem}[{\cite[Theorem1]{buhlmannBlockwiseBootstrapGeneral1995}}]
	\label{thm:buhlmann}
	Let $(X_n)_{n\in\N}\subset\R^d$ be a stationary sequence and consider a functional family $\FunctionalFamily=(F_t)_{t\in \FunctionalDomain}$ with finite bracketing entropy.
	Suppose that $\beta_X(k)\xrightarrow[k\to\infty]{} 0$ decrease exponentially and that $\FunctionalFamily$ satisfies~{(\ref{eq:kernel_compact_support},\ref{eq:kernel_lipschitz})}.
	Let the bootstrap sample be generated with the
	Moving Block Bootstrap,
	where the block size $L(n)$ satisfying $L(n)\rightarrow \infty$ and $L(n)=\bigO(n^{1/2-\epsilon})$ for some $0<\epsilon<\tfrac{1}{2}$. Then,
	\begin{equation*}
	\sqrt{N}(\hat{F}_N^* - \E^*[\hat{F}_N^*]) \rightarrow^* G_d \qquad \text{ in probability,}
	\end{equation*}
	where $G_d$ is the zero-mean Gaussian Process with the covariance~\eqref{eq:limit_covariance}.
\end{theorem}

\section{Proofs of Propositions~\ref{prop:upper_bound_p_persistence} and~\ref{prop:continuity_truncated_persistence}}
\label{appendix:proof_continuity_truncated_persistence}
\begin{proof}[Proof of Proposition~\ref{prop:upper_bound_p_persistence}]
We first note that when $A_h\leq\epsilon$, then $\pers_{p,\epsilon}^p(h)=0$.
For the non-trivial case, we follow the proof of Theorem 4.13 in~\cite{perez_c0-persistent_2022}. An upper-bound of the covering number of the image of $h$, at radius $\tau>0$ is $T(2\Lambda/\tau)^{1/\alpha}+1$, so that
\begin{equation*}
\pers_{p,\epsilon}^p(h)
\leq p\int_\epsilon^{A(f)} \left(T\left(\frac{2\Lambda}{\tau}\right)^{1/\alpha}+1\right) (\tau-\epsilon)^{p-1}d\tau
= (A_h-\epsilon)^p + pT(2\Lambda)^{1/\alpha}\int_\epsilon^{A(f)} \frac{(\tau-\epsilon)^{p-1}}{\tau^{1/\alpha}}d\tau
\end{equation*}
We recall that since $\frac{A_h}{\tau}\geq 1$ and $\frac{1}{\alpha}\leq p-1$, $(\frac{A_h}{\tau})^{1/\alpha}\leq (\frac{A_h}{\tau})^{p-1}$, so
\begin{equation*}
\frac{(\tau-\epsilon)^{p-1}}{\tau^{1/\alpha}} = \frac{1}{A_h^{1/\alpha}}\left(\frac{A_h}{\tau}\right)^{1/\alpha}(\tau-\epsilon)^{p-1} \leq
A_h^{p-1-1/\alpha}\left(1 - \frac{\epsilon}{\tau}\right)^{p-1}.
\end{equation*}
Finally, by recognizing that $1-\epsilon/\tau\leq 1-\epsilon/A_h,$ we obtain
\begin{align*}
\pers_{p,\epsilon}^p(h)
&\leq (A_h-\epsilon)^p + pT(2\Lambda)^{1/\alpha}A_h^{p-1-1/\alpha} (1-\epsilon/A_h)^{p-1} (A_h-\epsilon)\\
&\leq(A_h-\epsilon)(1-\epsilon/A_h)^{p-1}[A_h^{p-1}+ pT (2\Lambda)^{1/\alpha}A_h^{p-1-1/\alpha}]\\
&\leq(A_h-\epsilon)^p\left(1+pT\left(\tfrac{2\Lambda}{A_h}\right)^{1/\alpha}\right)\\
&\leq(A_h-\epsilon)^p\left(1+pT\left(\tfrac{2\Lambda}{\epsilon}\right)^{1/\alpha}\right),
\end{align*}
where we have used that $\epsilon^{1/\alpha} \leq A_h^{1/\alpha}$
\end{proof}

By H\"older continuity, $A_h\leq T^\alpha \Lambda$, so the ratio $\tfrac{T\Lambda^{1/\alpha}}{A_h^{1/\alpha}}\geq 1$ denotes how small the amplitude of $h$ is relative to what it could be, under the H\"older assumption. Interestingly, that term increases as $A_h$ gets smaller, but the whole bound is indeed increasing in $A_h$, which is of the order of $A_h^p + A_h^{2-1/\alpha}$.

\begin{proof}[Proof of Proposition~\ref{prop:continuity_truncated_persistence}]
Let $f,g\in C([0,T])$ such that $\Vert f-g\Vert_\infty<\epsilon/4$. Let $\Gamma:D(f)\rightarrow D(g)$ be a matching. Recall that $\vert w_\epsilon(b,d) - w_\epsilon(\eta_b, \eta_d)\vert \leq \vert b-\eta_b\vert + \vert d-\eta_d\vert \leq 2\Vert (b,d)-(\eta_b,\eta_d)\Vert_\infty$. In addition, if $d-b <\epsilon/2$, then both $w_\epsilon(b,d)=0 = w_\epsilon(\Gamma(b,d))$. Using the bound on the difference of $p$-powers as in the proof of Proposition~\ref{prop:truncated_persistence_continuity_bottleneck},
\begin{align*}
\left\vert \sum_{(b,d)\in D(f)} w_\epsilon(b,d)^p - \sum_{(b',d')\in D(g)} w_\epsilon(b',d')^p\right\vert
&\leq p\sum_{(b,d)\in D(f)} \vert w_\epsilon(b,d) - w_\epsilon(\Gamma(b,d))\vert \max \{w_\epsilon(b,d)^{p-1}, w_\epsilon(\Gamma(b,d))^{p-1}\}\\
&\leq 2p \Vert f-g\Vert_\infty \sum_{\substack{(b,d)\in D(f)\\ d-b\geq\epsilon/2}}\max \{w_\epsilon(b,d)^{p-1}, w_\epsilon(\Gamma(b,d))^{p-1}\}\\
&\leq p \left( \sum_{\substack{(b,d)\in D(f)\\ d-b\geq\epsilon/2}} (w_\epsilon(b,d)+2\epsilon/4)^{p-1}\right) \Vert f-g\Vert_\infty.
\end{align*}
Since $f$ is continuous on a compact domain, it is uniformly continuous, so the right-hand side is finite and depends only on $f$.

For the Lipschitz character, we follow the proof of~\cite[Lemma 3.20]{perez_c0-persistent_2022}. For $f,g\in C^\alpha_\Lambda([0,T])$,
\begin{align*}
\left\vert \sum_{(b,d)\in D(f)} w_\epsilon(b,d)^p - \sum_{(b',d')\in D(g)} w_\epsilon(b',d')^p\right\vert
&\leq p\sum_{(b,d)\in D(f)} \vert w_\epsilon(b,d) - w_\epsilon(\Gamma(b,d))\vert \max \{w_\epsilon(b,d)^{p-1}, w_\epsilon(\Gamma(b,d))^{p-1}\}\\
&\leq 2p \Vert f-g\Vert_\infty\left( \sum_{(b,d)\in D(f)} w_\epsilon(b,d)^{p-1} + \sum_{(b',d')\in D(g)} w_\epsilon(b',d')^{p-1} \right)\\
&= 2p (\pers_{p-1,\epsilon}^{p-1}(D(f)) + \pers_{p-1,\epsilon}^{p-1}(D(g))\Vert f-g\Vert_\infty.
\end{align*}
By Lemma~\ref{prop:upper_bound_p_persistence}, $\pers_{p-1,\epsilon}^{p-1}(D(f))\leq \frac{2^{1/\alpha}}{1-1/(p-1)\alpha} \Lambda^{p-1} T^{(p-1)\alpha-1}$, so that
\begin{equation*}
\vert\pers_{p,\epsilon}^p(D(f)) - \pers_{p,\epsilon}^p(D(g))\vert
\leq \frac{2^{2+1/\alpha}}{1-1/(p-1)\alpha} \Lambda^{p-1} T^{(p-1)\alpha-1}\Vert f-g\Vert_\infty.
\end{equation*}
\end{proof}

\section{Lipschitz constant for $\kernel^{pi}$ and $\kernel^{pi,t}$}
\label{appendix:persistence_image_lipschitz_constant}
First, $(x,y)\mapsto \exp(-(x^2+y^2))$ is $\tfrac{2\sqrt{2}}{e}-$Lipschitz with respect to the Euclidean norm, so $\tfrac{4}{e}-$Lipschitz for the Minkowski norm.
Let us now consider $\kernel^{pi,t}(b,d)(x,y) = \tfrac{1}{2\pi\sigma^2}\left(2-\tfrac{\Vert (b,d) - (x,y)\Vert_\infty}{\sigma}\right)_+^{r}\exp\left(-\tfrac{(b-x)^2+(d-y)^2}{2\sigma^2}\right)$. Then, for ${r>1}$,
\begin{align*}
&\left\vert\left(2-\tfrac{\Vert (b,d) - (x,y)\Vert_\infty}{\sigma} \right)_+^{r} - \left(2-\tfrac{\Vert (b',d') - (x,y)\Vert_\infty}{\sigma}\right)_+^{r}\right\vert
= \left\vert \int_0^1\tfrac{d}{dt} \left(2-\tfrac{\Vert (b,d) + (b'-b, d'-d)t - (x,y)\Vert_\infty}{\sigma}\right)_+^{r}dt\right\vert\\
\leq& \int_0^1\left\vert r\left(2-\tfrac{\vert b+(b'-b)t-x\vert}{\sigma}\right)_+^{r-1}(-1)^{b-x>b'-bt}\tfrac{(b'-b)}{\sigma}\indicator{\vert b + (b'-b)t-x\vert \geq \vert d+(d-d')t-y\vert}\right. \\
&+\left. r\left(2-\tfrac{\vert d+(d'-d)t-y\vert}{\sigma}\right)_+^{r-1}(-1)^{d-y>d'-dt}\tfrac{(d'-d)}{\sigma}\indicator{\vert b + (b'-b)t-x\vert \leq \vert d+(d-d')t-y\vert}\right\vert dt.\\
\leq& \int_0^1\tfrac{r}{\sigma}\left(\left(2-\tfrac{\vert b+(b'-b)t-x\vert}{\sigma}\right)_+^{r-1} \vert b-b'\vert + r\left(2-\tfrac{\vert d+(d'-d)t-y\vert}{\sigma}\right)_+^{r-1}\vert d-d'\vert \right)dt\\
\leq& \tfrac{r}{\sigma}\left((2-\tfrac{\min(\vert b-x\vert, \vert b'-x\vert)}{\sigma})_+^{r-1} \vert b-b'\vert + r(2-\tfrac{\min(\vert d-y\vert, \vert d'-y\vert)}{\sigma})_+^{r-1}\vert d-d'\vert \right)\\
\leq& \tfrac{2r}{\sigma}(2-\tfrac{\min(\Vert(b,d)-(x,y)\Vert_\infty, \Vert(b',d')-(x,y)\Vert_\infty)}{\sigma})_+^{r-1}\Vert (b,d)-(b',d')\Vert_\infty\\
\leq& \tfrac{2^r r}{\sigma} \Vert (b,d)-(b',d')\Vert_\infty.
\end{align*}
Then,
we obtain
\begin{align*}
\vert \kernel^{pi,t}(b,d)(x,y) - \kernel^{pi,t}(b',d')(x,y)\vert
\leq& \tfrac{1}{2\pi\sigma^2}\left\vert \left(2-\tfrac{\Vert (b,d) - (x,y)\Vert_\infty}{\sigma}\right)_+^{r} - \left(2-\tfrac{\Vert (b',d') - (x,y)\Vert_\infty}{\sigma}\right)_+^{r}\right\vert \exp\left(-\tfrac{(b-x)^2+(d-y)^2}{2\sigma^2}\right) \\
&+ \tfrac{1}{2\pi\sigma^2}\left(2-\tfrac{\Vert (b',d') - (x,y)\Vert_\infty}{\sigma}\right)_+^{r} \left\vert \exp\left(-\tfrac{(b-x)^2+(d-y)^2}{2\sigma^2}\right) - \exp\left(-\tfrac{(b'-x)^2+(d'-y)^2}{2\sigma^2}\right) \right\vert\\
\leq& \tfrac{1}{2\pi\sigma^2}\tfrac{2^r r}{\sigma} \Vert (b,d) - (b',d')\Vert_\infty + \tfrac{1}{2\pi\sigma^2} 2^r \tfrac{4}{e}\left\Vert\left(\tfrac{b-x}{\sigma}, \tfrac{d-y}{\sigma}\right) - \left(\tfrac{b'-x}{\sigma}, \tfrac{d'-y}{\sigma}\right)\right\Vert_\infty \\
\leq& \tfrac{2^{r-1}}{\pi\sigma^3}\left(r + 2\right)\Vert (b,d)-(b'd')\Vert_\infty.
\end{align*}

\section{Moments of the H\"older constant of a stochastic process}
\label{appendix:holder_moments}
Let $(W_t)_{t\in[0,T]}$ be a stochastic process. A path $t\mapsto W_t(\omega)$ is said to be $\alpha$-H\"older if $\vert W_t(\omega) - W_s(\omega)\vert \leq \Lambda_{W(\omega)}\vert s-t\vert^\alpha$, for any $s,t\in [0,T]$. Many processes, for example Gaussian processes, do not admit a uniform constant.
Based on~\cite{azais_level_2009,hu_multiparameter_2013,mathoverflow_kolmogorov_279085}, we will now give a condition under which $\Lambda_{W,\omega}$ is a random variable and we will calculate its moments.
\begin{proposition}[{\cite[Proposition 1.11]{azais_level_2009}}]
	\label{proposition:kolmogorov_implies_holder}
	Suppose $W$ satisfies~\eqref{eq:Kolmogorov_condition} with $K_{\KolmogP,\KolmogR}$ and let $\alpha\in\rbrack 0, \tfrac{\KolmogR}{\KolmogP}\lbrack$. Then, there exists a version $(V_t)_{t\in[0,1]}$ of $W$ and a random variable $\Lambda_{V,\alpha}>0$, such that, for all $s,t\in[0,1]$,
	\begin{equation*}
	P(\vert V_t - V_s\vert \leq \Lambda_{V,\alpha}\vert t-s\vert^\alpha) = 1 \quad\text{and}\quad P(W(t)=V(t)) = 1.
	\end{equation*}
\end{proposition}

\begin{theorem}[{\cite{mathoverflow_kolmogorov_279085}}]
	\label{thm:holder_constant_moments}
	Let $\KolmogP\in \mathbb{N}$ be such that $K_{\KolmogP, \alpha \KolmogP}<\infty$ and $1-\alpha>\tfrac{1}{\KolmogP}$, $\KolmogP\geq2$,
	\begin{equation*}
	\E[\Lambda_{W}] \leq 16\ \tfrac{\alpha + 1}{\alpha} T K_{\KolmogP,\KolmogP\alpha+1}^{1/\KolmogP}.
	\end{equation*}
	In addition,
	\begin{equation*}
	\E[\Lambda_{W}^k] \leq
	\begin{cases}
	\left(2^{3+2/\KolmogP}\ \tfrac{\alpha + 2/\KolmogP}{\alpha}\right)^k K_{\KolmogP,\KolmogP\alpha+1}^{k/\KolmogP},& \text{for }0<k\leq \KolmogP,\\
	\left(2^{3+2/\KolmogP}\ \tfrac{\alpha + 2/\KolmogP}{\alpha}\right)^k K_{k,k(\alpha + 2/\KolmogP)-1},&\text{for } k>\KolmogP.
	\end{cases}
	\end{equation*}
\end{theorem}

\begin{lemma}[Garsia--Rodemich--Rumsey Inequality~{\cite[Lemma 1.1]{hu_multiparameter_2013}}]
	\label{lemma:garsia_rodemich_rumsey}
	Let $G:\R_+ \rightarrow \R_+$ be a non--decreasing function with $\lim_{x\rightarrow\infty}G(x) = \infty$ and $\delta:[0,T]\rightarrow [0,T]$ continuous and non--decreasing with $\delta(0)=0$. Let $G^{-1}$ and $\delta^{-1}$ be lower--inverses.
	Let $f:[0,T]\rightarrow \R$ be a continuous functions such that
	\begin{equation*}
	\int_0^T\int_0^T G\left(\frac{\vert f(x)-f(y)\vert}{\delta(x-y)}\right) dxdy \leq B<\infty.
	\end{equation*}
	Then, for any $s,t\in[0,T]$,
	\begin{equation*}
	\vert f(s) - f(t)\vert \leq 8 \int_{0}^{\vert s-t\vert} G^{-1}(4B/u^2) d\delta(u).
	\end{equation*}
\end{lemma}
\begin{proof}[Proof of Theorem~\ref{thm:holder_constant_moments}]
	Consider a path $W(\omega)$ of the stochastic process and set\\
	$B(\omega)\coloneqq \int_0^T\int_0^T G\left(\frac{\vert W_t(\omega)W_s(\omega)\vert}{\delta(t-s)}\right) dtds$, where
	$G(u) = u^{\KolmogP}$ and $\delta(u) = u^{\alpha+2/\KolmogP}$. Then, $G^{-1}(u)= u^{1/\KolmogP}$ and $\tfrac{d}{du}\delta = (\alpha + 2/\KolmogP) u^{\alpha + 2/\KolmogP -1}$. Applying Lemma~\ref{lemma:garsia_rodemich_rumsey},
	\begin{align*}
	\vert W_t(\omega) - W_s(\omega)\vert
	&\leq 8 \int_{0}^{\vert s-t\vert} G^{-1}(4B(\omega)/u^2)d\delta(u)\\
	&\leq 8\int_0^{\vert t-s\vert} \left(\frac{4B(\omega)}{u^2}\right)^{1/\KolmogP} (\alpha + 2/p) u^{\alpha + 2/\KolmogP -1} du\\
	&\leq 8(4B(\omega))^{1/\KolmogP} (\alpha + 2/\KolmogP)\int_0^{\vert t-s\vert}u^{\alpha-1}du\\
	&= 8(4B(\omega))^{1/\KolmogP}\ \tfrac{\alpha + 2/\KolmogP}{\alpha} \vert t-s\vert^\alpha.
	\end{align*}
	As this is valid for any $s,\,t\in[0,T]$, $\Lambda_W(\omega) \leq 8(4B(\omega))^{1/\KolmogP}\ \tfrac{\alpha + 2/\KolmogP}{\alpha}$.
	By Jensens' inequality,
	\begin{equation}
	\label{eq:Jensen}
	\E[\Lambda_W]\leq 2^{3+2/\KolmogP}\ \tfrac{\alpha + 2/\KolmogP}{\alpha} \E[B(\omega)^{1/\KolmogP}]
	\leq 2^{3+2/\KolmogP}\ \tfrac{\alpha + 2/\KolmogP}{\alpha} \E[B(\omega)]^{1/\KolmogP}.
	\end{equation}
	By linearity of expectation,
	\begin{align*}
	\E\left[\int_0^T\int_0^T G\left(\frac{\vert W_t(\omega)W_s(\omega)\vert}{\delta(t-s)}\right) dtds\right]
	&= \int_0^T\int_0^T \frac{\E[\vert W_t(\omega)W_s(\omega)\vert^{\KolmogP}]}{\delta(t-s)^{\KolmogP}} dtds \\
	&= \int_0^T\int_0^T \frac{\E[\vert W_t(\omega)W_s(\omega)\vert^{\KolmogP}]}{\vert t-s\vert^{p\alpha + 2}} dtds\\
	&\leq \int_0^T\int_0^T K_{p,p\alpha + 1}dtds\\
	&= T^2 K_{\KolmogP,\KolmogP\alpha + 1}.
	\end{align*}
	Finally, $\E[\Lambda_W]\leq 2^{3+2/\KolmogP}\ \tfrac{\alpha + 2/\KolmogP}{\alpha} T^{2/\KolmogP} K_{\KolmogP,\KolmogP\alpha+1}^{1/\KolmogP}$, as long as $\KolmogP\alpha + 1 \leq \KolmogP$ and we can simplify the constants if $\KolmogP>2$.
	Consider now the higher moments. If $k\leq \KolmogP$, we can still apply Jensens' inequality in~\eqref{eq:Jensen}:
	\begin{equation*}
	\E[\Lambda_W^k]\leq \left(2^{3+2/\KolmogP}\ \tfrac{\alpha + 2/\KolmogP}{\alpha}\right)^{\KolmogP}\E[B(\omega)^{k/\KolmogP}]
	\leq \left(2^{3+2/\KolmogP}\ \tfrac{\alpha + 2/\KolmogP}{\alpha}\right)^k \E[B(\omega)]^{k/\KolmogP}\leq \left(2^{3+2/\KolmogP}\ \tfrac{\alpha + 2/\KolmogP}{\alpha}\right)^k K_{\KolmogP,\KolmogP\alpha+1}^{k/\KolmogP}.
	\end{equation*}
	However, if $k\geq \KolmogP$,
	\begin{align*}
	\E\left[\left(\int_0^T\int_0^T G\left(\frac{\vert W_t(\omega)W_s(\omega)\vert}{\delta(t-s)}\right) dtds\right)^{k/\KolmogP}\right]
	&= \int_0^T\int_0^T \frac{\E[\vert W_t(\omega)W_s(\omega)\vert^k]}{\delta(t-s)^k} dtds \\
	&= \int_0^T\int_0^T \frac{\E[\vert W_t(\omega)W_s(\omega)\vert^k]}{\vert t-s\vert^{k\alpha + 2k/\KolmogP}} dtds\\
	&\leq \int_0^T\int_0^T K_{k,k(\alpha + 2/\KolmogP)-1}dtds\\
	&= T^2 K_{k,k(\alpha + 2/\KolmogP)-1}.
	\end{align*}
\end{proof}

\end{document}